\documentclass[reqno,11pt]{amsart}

\usepackage{amsmath,amssymb,amsthm,mathrsfs,esint}
\usepackage{epsfig,color}
\usepackage[latin1]{inputenc}

\usepackage[numbers]{natbib}

 \allowdisplaybreaks
\numberwithin{equation}{section}

\def\Omstar{{\Om^*}}
\def\H{\mathcal H}

\def\R{\mathbb R}
\def\N{\mathbb N}

\def\e{\varepsilon}
\def\s{\sigma}
\def\S{\Sigma}
\def\vphi{\varphi}
\def\Div{{\rm div}\,}
\def\om{\omega}
\def\l{\lambda}
\def\g{\gamma}
\def\k{\kappa}
\def\Om{\Omega}
\def\de{\delta}
\def\Id{{\rm Id}}

\def\spt{{\rm spt}}

\def\pa{\partial}

\newcommand{\hd}{\mathrm{hd}}

\renewcommand{\a}{\alpha}
\renewcommand{\b}{\beta}

\renewcommand{\l}{\lambda}

\newcommand{\Lip}{{\rm Lip}}

\renewcommand{\Div}{{\rm div \,}}
\newcommand{\ov}{\overline}

\newcommand{\diam}{\mathrm{diam}}
\newcommand{\cc}{\subset\subset}

\newcommand{\op}[1]{\operatorname{\text{\rm #1}}}

\def\C{\mathbf{C}}
\def\D{\mathbf{D}}
\def\K{\mathbf{K}}

\theoremstyle{plain}
\newtheorem{theorem}{Theorem}[section]
\newtheorem{lemma}[theorem]{Lemma}

\newtheorem{proposition}[theorem]{Proposition}

\newtheorem{remark}[theorem]{Remark}

\def\SS{\mathbb{S}}

\title{Isoperimetry with upper mean curvature bounds and sharp stability estimates}

\author{B. Krummel}
\address{Mathematics Department, The University of Texas at Austin,
2515 Speedway Stop C1200, RLM 8.100 Austin, TX 78712, USA}
\email{bkrummel@math.utexas.edu}

\author{F. Maggi}
\address{Mathematics Department, The University of Texas at Austin,
2515 Speedway Stop C1200, RLM 8.100 Austin, TX 78712, USA}
\email{maggi@math.utexas.edu}

\begin{document}

\begin{abstract}
{\rm It was proved by Almgren that among boundaries whose mean curvature is bounded from above, perimeter is uniquely minimized by balls. We obtain sharp stability estimates for Almgren's isoperimetric principle and, as an application, we deduce a sharp description of boundaries with almost constant mean curvature under a total perimeter bound which prevents bubbling.}
\end{abstract}

\maketitle

\section{Introduction} \subsection{Overview} Our starting point is Almgren's paper \cite{almgren86}, where various optimal isoperimetric theorems, involving generalized surfaces and mappings in arbitrary codimension, are introduced. The main goal of \cite{almgren86} is proving the Euclidean isoperimetric inequality in higher codimension. Omitting to specify the crucial point of what is meant by ``area minimization with fixed boundary'', this is the statement: if $S$ is a $n$-dimensional compact surface without boundary in $\R^{n+k}$, $k\ge 1$, and $\Om_S$ is any $(n+1)$-dimensional area minimizing surface spanned by $S$, then
\begin{equation}
  \label{almgren isoperimetric inequality k}
  \frac{\H^n(S)}{\H^{n+1}(\Om_S)^{n/(n+1)}}\ge\frac{\H^n(D)}{\H^{n+1}(\Om_D)^{n/(n+1)}}
\end{equation}
where $\H^m$ is the $m$-dimensional Hausdorff measure in $\R^{n+k}$, $D$ is a unit radius $n$-dimensional sphere in $\R^{n+k}$, and thus $\Om_D$ is a unit radius $(n+1)$-dimensional ball in $\R^{n+k}$.

Almgren's proof of \eqref{almgren isoperimetric inequality k} roughly goes as follows. Assume that $S$ minimizes the left-hand side of \eqref{almgren isoperimetric inequality k} among boundary-less surfaces enclosing a minimal $(n+1)$-area equal to $\H^{n+1}(\Om_D)$. By a first variation argument one finds $|\vec{H}_S|\le n$, where $\vec{H}_S$ denotes the mean curvature vector of $S$ (with the convention that $|\vec{H}_D|=n$ for the unit $n$-sphere $D$). The proof is then completed by proving (see below for more details on this point) the following isoperimetric principle: if $S$ is a boundary-less surface with $|\vec{H}_S|\le n$, then $\H^n(S)\ge\H^n(D)$. This last fact is what we call here {\it Almgren's isoperimetric principle}.

The goal of our paper is addressing the stability of Almgren's isoperimetric principle in the codimension one case $k=1$. This case is relevant in the study of hypersurfaces with almost constant mean curvature, which, as discussed below, is in turn motivated by applications to capillarity theory and geometric flows. We obtain a sharp stability analysis of Almgren's principle, and we deduce from it new sharp results on hypersurfaces with almost constant mean curvature.

It is now convenient to restate Almgren's principle for smooth codimension one boundaries by taking advantage of the fact that, when $k=1$, Plateau's problem is trivial (as each boundary-less hypersurface $S$ is the boundary of just one set $\Om_S$ of finite volume): {\it if $\Om$ is a (non-empty) bounded open set with smooth boundary in $\R^{n+1}$ and $H_\Om$ denotes the mean curvature of $\pa\Om$ (computed with respect to the outer unit normal $\nu_\Om$ to $\Om$), then
\begin{equation}
  \label{almgren principle}
  H_\Om(x)\le n\quad\forall x\in\pa\Om\qquad\mbox{implies}\qquad P(\Om)\ge P(B_1)\,,
\end{equation}
with $P(\Om)=P(B_1)$ if and only if $\Om$ is a unit radius ball.} Here $n\ge 1$, $B_r(x)=\{y\in\R^{n+1}:|x-y|<r\}$ (for $x\in\R^{n+1}$ and $r>0$), $B_1=B_1(0)$ (so that $H_{B_1}=n$) and $P(\Om)=\H^n(\pa\Om)$ is the perimeter of $\Om$.


Almgren's proof of \eqref{almgren principle} is beautifully simple (and, quite interestingly, very close to the argument used in the theory of fully nonlinear elliptic equations to obtain the fundamental Alexandrov-Bakelman-Pucci estimate; see \cite{caffarellicabre}). If $A$ denotes the convex envelope of $\Om$, then the Gaussian curvature $K_A$ of $\pa A$ is equal to the Jacobian of the outer unit normal map $\nu_A:M\to\SS^n$ (where $\SS^n$ is the unit sphere) which in turn is injective by convexity. Hence, by the area formula
\[
P(B_1)=\H^n(\SS^n)=\int_{\pa A}\,K_A\,.
\]
Now, $K_A$ is the product of $n$ non-negative principal curvatures, so that by the arithmetic-geometric mean inequality $K_A\le(H_A/n)^n$; and, actually, $K_A=0$ outside of the contact set $ \pa\Om\cap \pa A$. Since, by assumption, $H_A=H_\Om\le n$ on $\pa A\cap\pa \Om$,
\[
P(B_1)=\int_{\pa A\cap\pa \Om}K_A\le\int_{\pa A\cap\pa \Om}\Big(\frac{H_A}{n}\Big)^n\le\H^n(\pa A\cap\pa \Om)\le \H^n(\pa\Om)=P(\Om)\,.
\]
We have thus proved that
\begin{eqnarray}\label{almgren identity}
  P(\Om)-P(B_1)=\H^n(\pa \Om\setminus\pa A)+\underbrace{\int_{\pa A\cap\pa \Om}\bigg(1-\Big(\frac{H_\Om}{n}\Big)^n\bigg)}_\text{$\ge0$ as $H_\Om\le n$ on $\pa\Om$}
  +
  \underbrace{\int_{\pa A\cap\pa \Om}\Big(\frac{H_A}{n}\Big)^n-K_A}_\text{$\ge0$ by $A$ convex, a.-g. mean inequality}
\end{eqnarray}
which clearly implies \eqref{almgren principle}.

Identity \eqref{almgren identity} is the starting point for discussion the rigidity assertion that, if $H_\Om\le n$ and $P(\Om)=P(B_1)$, then $\Om$ is a ball. Indeed, by combining $H_\Om\le n$ and $P(\Om)=P(B_1)$ into \eqref{almgren identity} we find that $\Om$ is convex and that $\pa\Om$ has constant mean curvature (equal to $n$) and it is umbilical at each of its points.

Each one of the last two properties individually implies that $\Om=B_1(x)$ for some $x\in\R^{n+1}$, in the first case thanks to Alexandrov's theorem (see \eqref{alexandrov theorem} below), and in the second case thanks to the {\it Nabelpunktsatz}:
\begin{equation}
  \label{nabelpunktsatz}
  \mbox{$\pa\Om$ is umbilical at each point if and only if $\Om=B_r(x)$ for some $r>0$ and $x\in\R^{n+1}$}\,.
\end{equation}
A third way of deducing from \eqref{almgren identity} that if $H_\Om\le n$ and $P(\Om)=P(B_1)$, then $\Om$ is a unit radius ball, is by exploiting the Euclidean isoperimetric inequality (see \eqref{isoperimetric inequality} below). Indeed, \eqref{almgren identity} implies $H_\Om=n$ on $\pa\Om$, and then by the divergence theorem and by the tangential divergence theorem one finds
\begin{eqnarray*}
(n+1)\,|\Om|&=&\int_\Om\,\Div\,x=\int_{\pa\Om}x\cdot \nu_\Om=\frac1n\,\int_{\pa\Om}H_\Om\,x\cdot \nu_\Om
\\
&=&\frac1n\,\int_{\pa\Om}\Div^{\pa\Om}(x)=\H^n(\pa\Om)=P(\Om)=P(B_1)=(n+1)\,|B_1|
\end{eqnarray*}
that is, $|\Om|=|B_1|$ (here $|\Om|=\H^{n+1}(\Om)$ is the volume of $\Om$). This last information combined with $P(\Om)=P(B_1)$ says that $\Om$ is an equality case in \eqref{isoperimetric inequality}, and thus that $\Om=B_1(x)$ for some $x\in\R^{n+1}$.

We aim at obtaining sharp stability estimates for the isoperimetric principle \eqref{almgren principle}. This is achieved in Theorem \ref{thm main}, Theorem \ref{thm structure} and Theorem \ref{thm sharp u} below, where the structure of sets with small $P(\Om)-P(B_1)$ is fully described and sharply quantified in terms of various notions of distance of $\Om$ from being a ball. As a by-product we obtain a new sharp stability result for Alexandrov's theorem, concerning the quantitative description of boundaries with almost-constant mean curvature: see Theorem \ref{thm alex} and Theorem \ref{thm alex L2} below.

The rest of this introduction is organized as follows. In section \ref{section stability related} we recall some stability results for related isoperimetric principles, which serves to illustrate the context of our main theorems. In section \ref{section main results} we state our main stability theorems for Almgren's isoperimetric principle, while in section \ref{section alex intro} we discuss the application to Alexandrov's theorem. Finally, in section \ref{section organization}, we address the organization of the paper.

\subsection{Stability theory for related isoperimetric principles}\label{section stability related} As noticed above, the characterization of equality cases in Almgren's principle can be addressed by exploiting either the Eucldiean isoperimetric inequality, Alexandrov's theorem or the {\it Nabelpunktsatz}. A presentation of some of the various stability theorems that have been obtained for these three isoperimetric principles is a necessary premise to the statement of our main results. We shall discuss in detail the situation for the Euclidean isoperimetric inequality and for Alexandrov's theorem, since Almgren's isoperimetric principle is sitting, so to say, in between these two theorems (see Remark \ref{remark sitting}). The {\it Nabelpunktsatz} also has a stability theory with sharp and non-sharp results, for which we refer readers to the seminal papers \cite{delellismuller1,delellismuller2} in the two-dimensional case, and to \cite{perez} for additional results in higher dimension.

Let us recall that given a Borel set $\Om\subset\R^{n+1}$ with finite and positive volume, the {\it Euclidean isoperimetric inequality} says
\begin{equation}
  \label{isoperimetric inequality}
  P(\Om)\ge P(B_1)\,\Big(\frac{|\Om|}{|B_1|}\Big)^{n/(n+1)}\,,
\end{equation}
where equality holds if and only if $\Om=B_r(x)$ for some $r>0$ and $x\in\R^{n+1}$. (In this generality, $P(\Om)$ denotes the distributional perimeter of $\Om$.) A sharp stability estimate for \eqref{isoperimetric inequality} is the improved isoperimetric inequality
\begin{equation}
  \label{quantitative isoperimetric inequality F asymmetry}
  P(\Om)\ge P(B_1)\,\Big(\frac{|\Om|}{|B_1|}\Big)^{n/(n+1)}\,\Big\{1+c(n)\,\a(\Om)^2\Big\}
\end{equation}
where $c(n)>0$ and $\a(\Om)$ denotes the {\it Fraenkel asymmetry of $\Om$}, defined as
\[
\a(\Om)=\inf\Big\{\frac{|\Om\Delta B_r(x)|}{|\Om|}:|B_r(x)|=|\Om|\,,x\in\R^{n+1}\Big\}\,;
\]
see \cite{fuscomaggipratelli,maggibams,FigalliMaggiPratelliINVENTIONES,CicaleseLeonardi}. Inequality \eqref{quantitative isoperimetric inequality F asymmetry} is sharp in the sense that no function of $\a(\Om)$ converging to $0$ more slowly than $\a(\Om)^2$ can appear on the right hand side of \eqref{quantitative isoperimetric inequality F asymmetry}. When considering some {\it a priori} geometric bound on $\Om$ one can obtain a qualitatively stronger information than a control on $\a(\Om)$. This kind of result is more conveniently stated by introducing the {\it isoperimetric deficit of $\Om$}
\[
\de_{{\rm iso}}(\Om)=\frac{P(\Om)\,|B_1|^{n/(n+1)}}{P(B_1)\,|\Om|^{n/(n+1)}}-1
\]
(a non-negative, scale invariant quantity which vanishes if and only if $\Om$ is a ball), in terms of which \eqref{quantitative isoperimetric inequality F asymmetry} takes the form
\[
\de_{{\rm iso}}(\Om)\ge c(n)\,\a(\Om)^2\,.
\]
We also recall the further improvement appeared in \cite{fuscojulin}, namely
\[
\de_{{\rm iso}}(\Om)\ge c(n)\,\Big\{\a(\Om)^2+\min_{x_0\in\R^{n+1}}\fint_{\pa\Om}\Big|\nu_\Om(x)-\frac{x-x_0}{|x-x_0|}\Big|^2\,d\H^n_x\Big\}\,.
\]
Denoting by $\hd$ the Hausdorff distance between compact subsets of $\R^{n+1}$, we introduce the {\it Hausdorff asymmetry of $\Om$}
\[
\hd_\a(\Om):=\inf\Big\{\frac{\hd(\pa\Om,\pa B_r(x))}r:|B_r(x)|=|\Om|\,,x\in\R^{n+1}\Big\}\,,
\]
and then recall the main result from \cite{Fuglede}: if $\Om$ is a convex set with $\de(\Om)\le\e$ for a suitable $\e$ depending on $n$ only, then
\begin{equation}
  \label{quantitative isoperimetric inequality hd asymmetry}
  c(n)\,\hd_\a(\Om)\le\left\{
  \begin{split}
    &\de_{{\rm iso}}(\Om)^{1/2}\,,\hspace{3cm}\mbox{if $n=1$}\,,
    \\
    &\de_{{\rm iso}}(\Om)^{1/2}\log^{1/2}(1/\de_{{\rm iso}}(\Om))\,,\hspace{0.5cm}\mbox{if $n=2$}\,,
    \\
    &\de_{{\rm iso}}(\Om)^{1/n}\,,\hspace{3cm}\mbox{if $n\ge 3$}\,.
  \end{split}
  \right .
\end{equation}
We notice that inequality \eqref{quantitative isoperimetric inequality hd asymmetry} also holds (with same exponents) whenever $\Om$ satisfies a uniform cone condition \cite{fuscogellipisante} or a uniform John's domain condition \cite{rajalazhong}. For a recent survey on \eqref{quantitative isoperimetric inequality F asymmetry} and related issues, see \cite{fuscosurveystab}.

We now discuss some stability results for {\it Alexandrov's theorem}: if $\Om$ is an open set in $\R^{n+1}$ with boundary of class $C^2$, then
\begin{equation}
  \label{alexandrov theorem}
  \mbox{$H_\Om$ is constant if and only if $\Om=B_r(x)$ for some $r>0$ and $x\in\R^{n+1}$.}
\end{equation}
The stability problem for Alexandrov's theorem amounts in understanding the geometry of boundaries with almost-constant mean curvature. To this end it is convenient to introduce the positive quantity
\begin{equation}\label{H0Omega}
H^0_\Om=\frac{n\,P(\Om)}{(n+1)|\Om|}\,,
\end{equation}
which has the following property: if there exists $c\in\R$ such that $H_\Om=c$ on $\pa\Om$, then $c=H^0_\Om$. Next, we define the {\it constant mean curvature deficit of $\Om$} as
\begin{equation}
  \label{alexandrov deficit}
\de_{{\rm cmc}}(\Om)=\Big\|\frac{H_\Om}{H_\Om^0}-1\Big\|_{L^\infty(\pa\Om)}\,.
\end{equation}
This quantity is scale invariant and by \eqref{alexandrov theorem} it vanishes if and only if $\Om$ is a ball. The use of the $L^\infty$-norm in the definition of $\de_{{\rm cmc}}(\Om)$ arises naturally in the study of capillarity theory, see \cite[Section 1.2]{ciraolomaggi}. The consideration of an $L^2$-type deficit would be interesting in view of applications to mean curvature flows.

A stability estimate in terms of $\de_{{\rm cmc}}(\Om)$ has been obtained in \cite{ciraolovezzoni} under the assumption that $\Om$ satisfies an interior/exterior ball condition of radius $\rho>0$ at each point of its boundary: if $\de_{{\rm cmc}}(\Om)\le\de_0(n,\rho,P(\Om))$, then
\begin{equation}
  \label{ciraolovezzoni inq}
  \hd_\a(\Om)\le C(n,\rho,P(\Om))\,\de_{{\rm cmc}}(\Om)\,.
\end{equation}
The decay rate of $\hd_\a(\Om)$ in terms of $\de_{{\rm cmc}}(\Om)$ in \eqref{ciraolovezzoni inq} is sharp. This result is obtained by making quantitative the original moving planes argument of Alexandrov, and using some kind of uniform ball condition seems unavoidable to this end. In view of applications to the study of local minimizers or critical points of capillarity-type energies this assumption is too restrictive. Moreover, an important consequence of the uniform ball assumption is that it prevents the observation of bubbling phenomena. Bubbling is observed, for example, by truncating and then smoothly completing unduloids with very thin necks. In this way one can construct sets $\Om$ with $\de_{{\rm cmc}}(\Om)$ arbitrarily small that are converging to {\it arrays of tangent balls}, rather than to a single ball. As shown in \cite{ciraolomaggi} this is actually the only mechanism by which one can construct boundaries with almost constant mean curvature: more precisely, working with a set $\Om$ that has been rescaled to that $H^0_\Om=n$, one has that if $L\in\N$, $\tau\in(0,1)$, and
\[
P(\Om)\le (L+\tau)\,P(B_1)\qquad \de_{{\rm cmc}}(\Om)\le \de_0
\]
then there exists a finite union $G$ of (at most $L$) tangent unit radius balls such that
\[
\max\Big\{|P(\Om)-P(G)|,|\Om\Delta G|,\hd(\pa\Om,\pa G)\Big\}\le C_0\,\de_{{\rm cmc}}(\Om)^{\a}\,;
\]
moreover, denoting by $\S$ the part of $\pa G$ obtained by removing a finite family of spherical caps, each with diameter bounded by $\de_{{\rm cmc}}(\Om)^{\a}$, there exists a map $u\in C^1(\S)$ such that
\begin{eqnarray*}
S=\big\{(1+u(x))\,\nu_G(x):x\in\S\big\}\subset\pa\Om\,,\qquad \H^n\big(\pa\Om\setminus S\big)\le
C_0\,\de_{{\rm cmc}}(\Om)^{\a}\,,
\end{eqnarray*}
and $\|u\|_{C^1(\S)}\le C_0\, \de_{{\rm cmc}}(\Om)^{\a}$. The constants $\de_0$ and $C_0$ depend on $L$, $\l$ and $n$ only, and $\a={\rm }O(n^{-p})$ for explicit values of $p\in\N$. This quantitative description of bubbling is not sharp, and an open problem is that to refining it to obtain sharp decay rates.

\subsection{Main results}\label{section main results} Our first main result is a sharp stability theorem for Almgren's isoperimetric principle \eqref{almgren principle}. Here and in the following we set
\[
\de(\Om)=P(\Om)-P(B_1)
\]
so that $\de(\Om)\ge0$ for every open set with smooth boundary such that $H_\Om\le n$ thanks to Almgren's principle.

\begin{theorem}[Main stability inequality]\label{thm main}
  For every $n\ge 1$ there exists positive constants $\de_0(n)$ and $c_0(n)$ with the following property. If $\Om\subset\R^{n+1}$ is a bounded, open set with smooth boundary such that $H_\Om(x)\le n$ for every $x\in\pa\Om$ and $\de(\Om)\le\de_0(n)$, then there exists $x\in\R^n$ such that
  \begin{equation}
    \label{quantitative almgren principle sharp Omega}
      P(\Om)\ge P(B_1)+c_0(n)\bigg\{|\Om\Delta B_1(x)|+\inf\Big\{\e>0:\Om\subset B_{1+\e}(x)\Big\}\bigg\}\,.
  \end{equation}
\end{theorem}

Estimate \eqref{quantitative almgren principle sharp Omega} says that $\de(\Om)$ controls linearly the Fraenkel asymmetry of $\Om$ and ``one side'' of its Hausdorff asymmetry whenever $\de(\Om)$ is small enough. The decay rate is sharp, in the sense that it is not possible to control these quantities by any function of $\de(\Om)$ going to zero faster than $\de(\Om)$ itself. A simple example showing this is obtained by considering the family of sets $\Om_t=B_{1+t}$ as $t\to 0^+$. Moreover, outside of the regime when $\de(\Om)$ is small we cannot expect to control the geometry of $\Om$, and it is not even true that $\a(\Om)={\rm O}(\de(\Om))$: to see this, pick any bounded smooth set $E$, set $\Om=R\,E$ for $R$ large enough to entail $H_\Om\le n$, and then $\a(\Om)={\rm O}(|\Om|)={\rm O}(R^{n+1})={\rm O}(P(\Om)^{(n+1)/n})={\rm O}(\de(\Om)^{(n+1)/n})$ as $R\to\infty$.

We also notice that one cannot hope to obtain a better type of geometric information on the boundary of $\Om$. A first example showing this, that can be observed already in dimension $n=1$, is obtained by letting $\Om$ be a unit ball with arbitrarily many tiny holes, whose boundaries have large but negative mean curvature, and whose presence prevents $\Om$ from containing a ball of radius $1-\e$ (i.e., the other ``side'' of the Hausdorff asymmetry estimate does not hold). If $n=1$ this kind of problem can be avoided by assuming that $\pa\Om$ is connected, but in dimension $n\ge 2$ one can indeed draw the same conclusions by constructing sets $\Om$ satisfying the assumptions of Theorem \ref{thm main}, with $P(\Om)-P(B_1)$ arbitrarily small, and with arbitrarily long ``inner tentacles'' of very negative mean curvature; see
\begin{figure}
  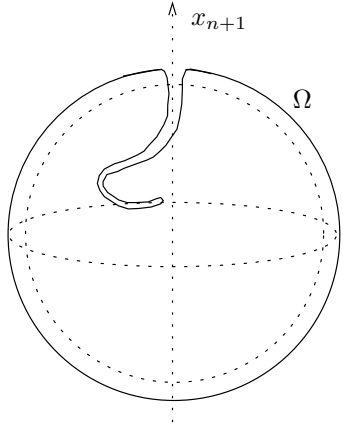\caption{{\small If $n\ge 2$, a set $\Om\subset\R^{n+1}$ with $H_\Om\le n$ can have an inner tentacle of length one with small volume and area, and perimeter arbitrarily close to $P(B_1)$. Notice that one needs to start from a ball with radius slightly larger than $1$ (and thus with mean curvature slightly smaller than $n$) to create a tentacle. Indeed, at the opening of the tentacle, $\Om$ turns faster than its reference ball.}}\label{fig innertentacle}
\end{figure}
Figure \ref{fig innertentacle}.

These two examples exploit the possibility for $H_\Om$ to be arbitrarily negative. Now there are two important remarks: first, if we assume a lower bound on the mean curvature, in addition to the upper bound $H_\Om\le n$, then it is possible to control the full Hausdorff asymmetry with $\de(\Om)$; second, given a set $\Om$ with $H_\Om\le n$ we can always find a set $E$ with $|H_E|\le n$ and whose distance from $\Om$ is controlled in terms of $\de(\Om)$. In our next result we start providing a complete quantitative description of the geometry of sets with small $\de(\Om)$. In particular we show that up to holes and inner tentacles of small perimeter, every such that is a $C^1$-small deformation of a unit ball.

\begin{theorem}[Structure of sets with small deficit]
  \label{thm structure}
  Let $n\ge 1$ and let $\Om\subset\R^{n+1}$ be a bounded, open set with smooth boundary such that $H_\Om(x)\le n$ for every $x\in\pa\Om$.

  \medskip

  \noindent (i) If $\de(\Om)<P(B_1)$, then $\Om$ is connected and there exists a bounded open set $\Om^*$ such that $\pa\Om^*$ is connected and
  \begin{eqnarray*}
    &&\mbox{$\Om\subset\Om^*$ with $\pa\Om^*\subset\pa\Om$}\,,
    \\
    &&\diam(\Om)=\diam(\Om^*)\,,
    \\
    &&\H^n(\pa\Om\setminus\pa\Om^*)\le C(n)\,\de(\Om)\,,
    \\
    &&|\Om^*\setminus\Om|\le C(n)\,\de(\Om)^{(n+1)/n}\,.
  \end{eqnarray*}

  \medskip

  \noindent (ii) If $\de(\Om)\le\de_0(n)$, then there exists an open bounded set $E$ with boundary of class $C^{1,1}$ such that
  \begin{eqnarray*}
    &&\Om\subset E
    \\
    &&\diam(\Om)=\diam(E)\,,
    \\
    &&|E\setminus\Om|+\H^n\big(\pa E\setminus\pa\Om\big)\le C(n)\,\de(\Om)
    \\
    &&\|H_E\|_{L^\infty(\pa E)}\le n\,.
  \end{eqnarray*}
  In addition,  up to translations,
  \begin{equation}
    \label{u}
      \pa E=\big\{(1+u(x))x:x\in\SS^n\big\}
  \end{equation}
  for some function $u\in C^1(\SS^n)$, and for every $\e>0$
  \[
  \de(\Om)\le\de_0(n,\e)\qquad\Rightarrow\qquad\|u\|_{C^1(\SS^n)}\le\e\,.
  \]
\end{theorem}

\begin{remark}\label{remark structure}
  {\rm By Theorem \ref{thm structure}, if $\Om$ has $H_\Om\le n$ on $\pa\Om$ and $\de(\Om)$ small, then $\pa\Om$ has a large connected component $\pa\Om^*$ which accounts for all the perimeter of $\Om$ up to an error of order $\de(\Om)$. In turn, we can chop $\pa\Om^*$ where its mean curvature is less than $-n$, and complete it into a new set $E$ with bounded mean curvature; see
  \begin{figure}
  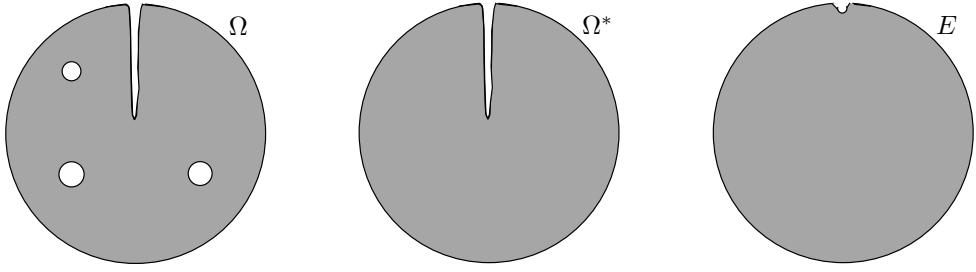\caption{{\small Theorem \ref{thm structure}. A set $\Om$ with small $\de(\Om)$ is connected, and all the connected components of its boundary but one have small perimeter. The set $\Om^*$ is obtained by removing them, and may contain (if $n\ge 2$) inner tentacles of order one length. Finally, the set $E$, which is essentially obtained by truncating the mean curvature $H_\Om$ where $H_\Om<-n$.}}\label{fig structure}
  \end{figure}
  Figure \ref{fig structure}. The error we make in doing this is linear in $\de(\Om)$ both in volume and perimeter. The new set $E$ is a small $C^1$-deformation of the sphere, and Theorem \ref{thm sharp u} below is the sharp stability theorem for this kind of sets. Thus by combining Theorem \ref{thm structure} and Theorem \ref{thm sharp u} we shall obtain a complete and sharp analysis of sets with $\de(\Om)$ small. Theorem \ref{thm main} will be a direct consequence of these results.}
\end{remark}

\begin{remark}
  {\rm Theorem \ref{thm structure} requires using a non-classical notion of mean curvature, suitable for boundaries of class $C^{1,1}$. As explained in more detail in section \ref{section notation} below, for every an open set with $C^{1,1}$-boundary $E$ there exists a function $H_E\in L^\infty(\H^n\llcorner\pa E)$ such that
  \[
  \int_{\pa E}\Div^{\pa E}X\,d\H^n=\int_{\pa E}(X\cdot\nu_{E})\,H_E\,d\H^n\qquad\forall X\in C^\infty_c(\R^{n+1};\R^{n+1})\,,
  \]
  (where $\Div^{\pa E} X=\Div X-\nu_{E}\cdot \nabla X[\nu_E]$). The function $H_E$ is the {\it generalized mean curvature of (the boundary of) $E$ (with respect to the outer unit normal $\nu_E$)}. In the specific case of Theorem \ref{thm structure}, $E$ is constructed by solving a penalized obstacle problem, see Proposition \ref{prop E construction}, and it will turn out that $\pa E$ is actually analytic, with constant mean curvature equal to $-n$, on $\pa E\setminus\pa\Om$. The construction of $\Om^*$ in statement (i) is, technically speaking, much simpler, as it is just based on the repeated application of Almgren's principle to the connected components of $\pa\Om$. From the formal point of view we shall just need part (ii) in the proof of Theorem \ref{thm main}, and part (i) has just been included because it is based on an explicit bound on $\de(\Om)$, and its proof is based on a very natural idea.}
\end{remark}

In order to complete the quantitative description of sets with small $\de(\Om)$ we are left to quantify the size of the function $u$ appearing in \eqref{u}. This is done in the next theorem.

\begin{theorem}[Stability of normal perturbations of $\SS^n$]
  \label{thm sharp u}
  Let $n\ge 1$, let $\Om$ be the open bounded set with smooth boundary in $\R^{n+1}$ with $H_\Om\le n$ $\H^n$-a.e. on $\pa \Om$, such that
  \[
  \int_{\pa \Om}x\,d\H^n_x=0\qquad
  \pa \Om=\big\{(1+u(x))x:x\in\SS^n\big\}\,,
  \]
  for a function $u:\SS^n\to\R$ such that
  \[
  \|u\|_{C^1(\SS^n)}\le \e(n)\,.
  \]
  If $\e(n)$ is suitably small, then
  \begin{equation}
    \label{sharp barycenter}
      \frac{\de(\Om)}{C(n)}\le\int_{\SS^n}u\le C(n)\,\de(\Om)
  \end{equation}
  and
  \begin{eqnarray}
  \label{sharp L1}
    \|u\|_{W^{1,1}(\SS^n)}&\le & C(n)\,\de(\Om)\,,
    \\\label{sharp c0+}
    \|u^+\|_{C^0(\SS^n)}&\le& C(n)\,\de(\Om)\,,
    \\\label{sharp c0}
    \|u\|_{C^0(\SS^n)}&\le&C(n)\,\left\{
    \begin{split}
      &\de(\Om)\,,&\qquad\mbox{if $n=1$}\,,
      \\
      &\de(\Om)\,\log\left(\frac{C(2)}{\de(\Om)}\right)\,,&\qquad\mbox{if $n=2$}\,,
      \\
      &\de(\Om)^{1/(n-1)}\,,&\qquad\mbox{if $n\ge 3$}\,.
    \end{split}
    \right .
    \\
    \label{sharp w12}
    \|u\|_{W^{1,2}(\SS^n)}&\le& C(n)\,\sqrt{\|u\|_{C^0(\SS^n)}\,\de(\Om)+\de(\Om)^2}\,.
  \end{eqnarray}
  All these estimates are sharp.
\end{theorem}

\begin{remark}
  {\rm Estimate \eqref{sharp L1}, \eqref{sharp c0+} and \eqref{sharp c0} can be read in more geometric by taking into account that
  \begin{eqnarray*}
  |\Om\Delta B_1|&\le&C(n)\,\|u\|_{L^1(\Om)}
  \\
  \inf\big\{\e>0:\Om\subset B_{1+\e}\big\}&\le&C(n)\,\|u^+\|_{C^0(\SS^n)}
  \\
  \hd(\pa\Om,\SS^n)&\le&C(n)\,\|u\|_{C^0(\SS^n)}\,.
  \end{eqnarray*}}
\end{remark}

\begin{remark}\label{remark sitting}
  {\rm It seems useful to illustrate the links and differences between the stability problems for the isoperimetric inequality, Alexandrov's theorem, and Almgren's isoperimetric principle.  Consider the functional $F(\Om)$ on sets with finite perimeter with positive and finite volume $\Om\subset\R^{n+1}$ defined by
  \[
	F(\Om) = \frac{P(\Om)}{|\Om|^{n/(n+1)}} \,.
\]
The isoperimetric theorem and Alexandrov's theorem say that the only global minimizers of $F$ are its only critical points, namely balls in $\R^{n+1}$.  Let $\Sigma$ denote the set of all balls in $\R^{n+1}$.  Stability for the isoperimetric inequality means controlling the distance of $\Om$ from $\Sigma$ in terms of the deviation of $F(\Om)$ from the minimum value of $F$.  Stability for Alexandrov's theorem means controlling the distance of $\Om$ from $\Sigma$ in terms of the size of $\de F$, the first variation of $F$. In this second stability problem a complication is due to the presence of ``critical points at infinity'' (here we are borrowing some terminology from the Yamabe problem, see \cite{bahribook}): precisely, arrays of almost tangent balls with equal radii connected by short necks provide families of almost critical points to $F$.  Stability for Almgren's isoperimetric principle means controlling, under a unilateral constraint on $\de F$, the distance of $\Om$ from $\Sigma$ in terms of the deviation of $F$ from its minimum value on the constrained class.

In each problem we permit different classes of variations of balls. Consider for example variations of the form $\pa\Om=\{ (1+u(x))\,x: x \in \SS^n \}$ for some $u\in C^1(\SS^n)$ with small $C^1$-norm. Taking $u$ to be constant correspond to scaling, so the average of $u$ (projection of $u$ on constants) tells how much we are scaling $\SS^n$ when deforming it into $\pa\Om$. For the isoperimetric inequality, minimizing perimeter with a fixed volume constraint means that the effect of scaling must be negligible, that is $\int_{\SS^n}u={\rm O}(\int_{\SS^n}u^2)$. By contrast, for Almgren's isoperimetric principle, the only constraint on $\int_{\SS^n}u$ is on its sign (which must be non-negative, as one must scale outward in order to preserve the condition $H_{\Omega} \leq n$) but not on its size. For Alexandrov's theorem we have no sign restriction, and $|\int_{\SS^n}u|$ is just controlled by the oscillation of the mean curvature from the constant value $n=H_{B_1}$.

With all this in mind, it seems unlikely that one can directly address stability for Almgren's isoperimetric principle from stability for the isoperimetric inequality or for Alexandrov's theorem. It is however possible to use stability for Almgren's isoperimetric principle to understand stability for Alexandrov's theorem, as we illustrate in the next section.}
\end{remark}

\subsection{A sharp estimate for boundaries with almost constant mean curvature}\label{section alex intro} Here we introduce a sharp stability result for Alexandrov's theorem. We address the issue under a global assumption aimed at preventing bubbling, as opposed to the local assumption of a uniform exterior/interior ball condition considered in \cite{ciraolovezzoni}. The assumption we make is that our sets $\Om$, after setting $H^0_\Om=n$ by scaling (recall \eqref{H0Omega}), satisfies $P(\Om)<2\,P(B_1)$. We show then that the constant mean curvature deficit $\de_{{\rm cmc}}(\Om)$ (defined in \eqref{alexandrov deficit}) controls linearly the Hausdorff asymmetry of $\Om$. We thus arrive to the same conclusion of \eqref{ciraolovezzoni inq} coming from a different direction.

\begin{theorem}
  \label{thm alex} If $\tau\in(0,1)$, $n\ge 1$, $\Om$ is a bounded open set in $\R^{n+1}$ with smooth boundary such that
  \[
  H_\Om^0=n\qquad P(\Om)\le 2\,\tau\,P(B_1)\qquad \de_{{\rm cmc}}(\Om)\le\de_0(n,\tau)
  \]
  then there exists $u\in C^{1,1}(\SS^n)$ such that, up to a translation, $\pa\Om=\{(1+u(x))\,x:x\in\SS^n\}$ with
  \[
  \|u\|_{C^1(\SS^n)} \le C(n)\,\de_{{\rm cmc}}(\Om)\,.
  \]
\end{theorem}

\begin{remark}
  {\rm The conclusion of Theorem \ref{thm alex} is sharp (think to ellipsoids with small eccentricities) and it implies in particular that
  \[
  \max\bigg\{\big|P(\Om)-P(B_1)\big|\,,\Big\|\nu_\Om-\frac{x}{|x|}\Big\|_{C^0(\pa\Om)}\,,|\Om\Delta B_1|,\hd(\pa\Om,\SS^n)\bigg\}
  \le C(n)\,\de_{{\rm cmc}}(\Om)\,.
  \]}
\end{remark}

As mentioned before, although using an $L^\infty$-deficit like $\de_{{\rm cmc}}(\Om)$ is sufficient in view of applications to capillarity theory, having in mind to address convergence to equilibrium in geometric flows (see, for example, \cite{ciraolofigallimaggi} for this kind of application of stability theorems to Yamabe-type fast diffusion equations) it would be interesting to obtain a result analogous to Theorem \ref{thm alex} with an $L^2$-deficit in place of $\de_{{\rm cmc}}(\Om)$.  In fact, without assuming pointwise bounds on the mean curvature of $\Om$, we can show that the $W^{1,2}$-distance of $\pa\Om$ to the unit sphere is bounded linearly in terms of the $L^2$-deficit $\|H_{\Om} - n\|_{L^2(\pa\Om)}$ whenever $\pa\Om$ is a sufficiently $C^1$-small perturbation of the unit sphere. Moreover, using slightly stronger integral deficits, we can also control the $C^0$-norm of $u$ in terms of the oscillation of the mean curvature.

\begin{theorem}
  \label{thm alex L2}
  If $n\ge 1$ and $\Om$ is an open set with $C^{1,1}$-boundary such that  $\int_{\pa\Om}x=0$ and $\pa \Om=\{(1+u(x))x:x\in\SS^n\}$ for a function $u\in C^{1,1}(\SS^n)$ with
  \[
  \|u\|_{C^1(\SS^n)}  \le \e(n)
  \]
  then
  \begin{equation} \label{alex L2 bound}
  \|u\|_{W^{1,2}(\SS^n)} \le C(n)\, \| H_{\Om} - n \|_{L^2(\pa\Om)}\,.
  \end{equation}
  Moreover, if
  \begin{equation}
    \label{p hp new}
      p\ge2\quad\mbox{for $n\le 3$}\,,\qquad p>\frac{n}2\quad\mbox{if $n\ge 4$}\,.
  \end{equation}
  then
  \begin{equation}
    \label{alex C0 bound new}
      \|u\|_{C^0(\SS^n)}\le C(n,p)\,\|H_\Om-n\|_{L^p(\pa\Om)}\,.
  \end{equation}
  Finally, if $\a\in(0,1)$ and $K>0$ is such that $\|u\|_{C^{1,\a}(\SS^n)}\le K$, then
  \begin{equation}
    \label{alex C1alpha}
      \|u\|_{C^{1,\a}(\SS^n)} \le C(n,K,\a)\,\de_{{\rm cmc}}(\Om)\,.
  \end{equation}
\end{theorem}

\subsection{Organization of the paper}\label{section organization} After introducing some notation and basic facts in section \ref{section notation}, in section \ref{section structure and reduction} we discuss the structure of sets with small Almgren's deficit and prove Theorem \ref{thm structure}. Section \ref{section normal def of SSn} is devoted to the study of normal deformations of $\SS^n$. There we obtain the various estimates from Theorem \ref{thm sharp u} (whose optimality is addressed in section \ref{section example}) which we use to complete the proof of Theorem \ref{thm main}. Finally, the applications to boundaries with almost constant mean curvature is discussed in section \ref{section sharp alexandrov}, where Theorem \ref{thm alex} and Theorem \ref{thm alex L2} are proved.

\medskip

\noindent {\bf Acknowledgment:} This work was supported by NSF through grants DMS Grant 1265910 and the DMS FRG Grant 1361122.

\section{Notation and terminology}\label{section notation} Here we gather some definitions and facts that are used throughout the paper. We refer to \cite{maggiBOOK}, and point out \cite{SimonLN,AFP,FedererBOOK,KrantzParks,GMSbook1,GMSbook2} as additional references.

\medskip

\noindent {\bf Rectifiable sets and mean curvature:} A Borel set $S\subset\R^{n+1}$ is {\it locally $\H^n$-rectifiable in $\R^{n+1}$} if there exists a family of maps $\{f_h\}_{h\in\N}\subset C^1(\R^n;\R^{n+1})$
\[
\H^n\Big(S\setminus\bigcup_{h\in\N}f_h(\R^n)\Big)=0
\]
and $\H^n(S\cap B_R)<\infty$ for every $R>0$. In particular, $\H^n\llcorner S$ is a Radon measure on $\R^{n+1}$. If $S$ is locally $\H^n$-rectifiable, then $S$ admits an {\it approximate tangent space} $T_xS$ at $\H^n$-a.e. $x\in S$, that is $T_xS$ is an hyperplane in $\R^{n+1}$ with the property that
\[
\lim_{r\to 0^+}\frac1{r^n}\int_S \vphi\Big(\frac{y-x}r\Big)\,d\H^n(y)=\int_{T_xS} \vphi\,d\H^n\qquad\forall \vphi\in C^0_c(\R^{n+1})\,;
\]
see e.g. \cite[Theorem 10.2]{maggiBOOK}. If for every such $x\in S$ we denote by $\nu(x)$ a unit normal vector to $T_xS$, then for every $X\in C^1_c(\R^{n+1};\R^{n+1})$ the formula
\[
\Div^SX(x)=\Div(X)(x)-\nabla X(x)[\nu(x)]\cdot\nu(x)
\]
defines a Borel map on $S$. The vector-valued distribution $\vec H_S$
\[
\langle\vec H_S,X\rangle=\int_S\,\Div^S X\,d\H^n\qquad X\in C^1_c(\R^{n+1};\R^{n+1})
\]
is called the {\it distributional mean curvature} of $S$. We say that $S$ has {\it generalized mean curvature} if, given a Borel map $\nu:S\to\SS^n$ such that $\nu(x)$ is normal to $T_xS$ for $\H^n$-a.e. $x\in S$, there exists $H_S\in L^1_{{\rm loc}}(\H^n\llcorner S)$ such that
\[
\langle\vec H_S,X\rangle=\int_S\,X\cdot\nu\,H_S\,d\H^n\qquad\forall X\in C^1_c(\R^{n+1};\R^{n+1})\,.
\]
Then $H_S$ is the {\it scalar mean curvature of $S$} with respect to $\nu$. If we have $H_S\in L^\infty(\H^n\llcorner S)$, then $S$ has {\it generalized bounded mean curvature}.

\medskip

\noindent {\bf Sets of finite perimeter:} A Borel set $\Om\subset\R^{n+1}$ is of finite perimeter in $\R^{n+1}$ if there exists a $\R^{n+1}$-valued Radon measure $\mu_\Om$ on $\R^{n+1}$ such that
\begin{equation}
  \label{div thm}
  \int_{\Om}\Div\,X(x)\,dx=\int_{\R^{n+1}}X\cdot\,d\mu_\Om\qquad\forall X\in C^1_c(\R^{n+1};\R^{n+1})\,.
\end{equation}
If $|\mu_\Om|$ denotes the total variation of $\mu_\Om$, then the set $\pa^*\Om$ of those $x\in\R^{n+1}$ such that
\begin{equation}
  \label{nuOm}
  \lim_{r\to 0^+}\frac{\mu_\Om(B_r(x))}{|\mu_\Om|(B_r(x))}
\end{equation}
exists and belongs to $\SS^n$, is called the {\it reduced boundary of $\Om$}, and the limit $\nu_\Om(x)\in\SS^n$ in \eqref{nuOm} is called the {\it measure-theoretic outer unit normal to $\Om$}. One always has that $\pa^*\Om$ is a locally $\H^n$-rectifiable set and that $T_x(\pa^*\Om)$ exists for every $x\in\pa^*\Om$ with $T_x(\pa^*\Om)=\nu_\Om(x)^\perp$; moreover, $\mu_\Om=\nu_\Om\,\H^n\llcorner\pa^*\Om$, so that \eqref{div thm} takes the form
\[
  \int_{\Om}\Div\,X(x)\,dx=\int_{\pa^*\Om}X\cdot\nu_\Om\,d\H^n\qquad\forall  X\in C^1_c(\R^{n+1};\R^{n+1})\,;
\]
see \cite [Chapter 15]{maggiBOOK}. If $|\Om\Delta\Om'|=0$, then clearly $\mu_{\Om}=\mu_{\Om'}$ and $\pa^*\Om=\pa^*\Om'$, although the topological boundaries of $\Om$ and $\Om'$ may largely differ. However, up to replace $\Om$ with an $\Om'$ such that $|\Om\Delta\Om'|=0$, it is always possible to obtain
\[
\pa\Om=\big\{x\in\R^{n+1}:0<|\Om\cap B_r(x)|<|B_r(x)|\quad\forall r>0\big\}=\spt\mu_\Om=\ov{\pa^*\Om}\,;
\]
see \cite[Chapter 12]{maggiBOOK}. We shall always assume that our sets of finite perimeter have been normalized so that these identities are in force. Given a Borel set $\Om\subset\R^{n+1}$ and $t\in[0,1]$, we set
\[
\Om^{(t)}=\Big\{x\in\R^{n+1}:\lim_{r\to 0^+}\frac{|\Om\cap B_r(x)|}{|B_r(x)|}=t\Big\}
\]
for the set of points of density $t$ of $\Om$. If $\Om$ is a set of locally finite perimeter in $\R^{n+1}$, then by a result of Federer \cite[Theorem 16.2]{maggiBOOK} we have
\[
\R^{n+1}\underset{\H^n}{=}\Om^{(0)}\cup\Om^{(1)}\cup\pa^*\Om
\]
where
\[
A\underset{\H^n}{=}B\qquad\mbox{means}\qquad \H^n(A\Delta B)=0\,.
\]
A set of finite perimeter $\Om$ has {\it generalized mean curvature} $H$ in an open set $A\subset\R^{n+1}$ if $H\in L^1_{{\rm loc}}(\H^n\llcorner(A\cap\pa^*\Om))$ is such that
\[
\int_{\pa^*\Om}\,\Div^{\pa^*\Om}X\,d\H^n=\int_{\pa^*\Om}\,X\cdot\nu_\Om\,H\,d\H^n\qquad\forall X\in C^1_c(A;\R^{n+1})\,.
\]
In this case, $H_\Om$ is uniquely determined ($\H^n$-a.e. on $A\cap\pa^*\Om$) and we set $H=H_\Om$. Notice that {\it with this convention $H_\Om\ge0$ for smooth convex sets and $H_\Om=n$ if $\Om$ is a ball of unit radius}. When $H_\Om\in L^\infty(\H^n\llcorner(A\cap\pa^*\Om))$ we say that $\Om$ has {\it generalized bounded mean curvature in $A$}. An important example to keep in mind in our analysis is the following: if $\Om$ is an open set in $\R^{n+1}$ which, nearby $0\in\pa\Om$, is the epigraph in the $e_{n+1}$-direction of a function $u\in C^{1,1}(D)$ with $u(0)=0$ and $D$ a ball in $\R^n$ centered at $0$, then $\Om$ has generalized bounded mean curvature in said neighborhood of $0$, and
\begin{equation}
  \label{formula HOm for Om C11}
  H_\Om(x+u(x)\,e_{n+1})=\frac{\Delta u(x)}{\sqrt{1+|\nabla u(x)|^2}}-\frac{\nabla^2u(x)[\nabla u(x)]\cdot\nabla u(x)}{(1+|\nabla u(x)|^2)^{3/2}}
\end{equation}
for a.e. $x\in D$. Here $\nabla^2 u$ stands the distributional gradient of $u\in C^{1,1}(D)=W^{2,\infty}(D)$, so that $\nabla^2u(x)$ is indeed uniquely determined a.e. on $D$.

\medskip

\noindent {\it Perimeter almost-minimizers}: We say that $E\subset\R^{n+1}$ is a perimeter $(\Lambda,r_0,\a)$-minimizer in some open set $A$ if for some $\a\in(0,1)$
\[
P(E;W)\le P(F;W)+\Lambda\,r^{n+2\a}
\]
whenever $E\Delta F\cc W\cc A$ with $\diam(W)=r<r_0$. In this context, the classical $\e$-regularity theorem and dimension reduction scheme lead to the following statement: if $E$ is a $(\Lambda,r_0,\a)$-minimizer in $A$, then there exists a set $\S\subset A\cap\pa E$, relatively closed in $A$, such that $E$ is an open set with $C^{1,\a}$-boundary in $A\setminus\S$ and the Hausdorff dimension of $\S$ is at most $n-7$ (actually, $\S$ is locally finite in $A$ if $n=7$, and $\S=\emptyset$ if $n\le 6$); see, e.g. \cite{Tamaniniholder,maggiBOOK}. The result is sharp, in the sense that every open set with $C^{1,\a}$-boundary is a $(\Lambda,r_0,\a)$-minimizer; see \cite[Section 1.6]{Tamanini}. For the reader's convenience this last fact is recalled in the following proposition, where for $x\in\R^{n+1}$, $\nu\in\SS^n$ and $r>0$ we set
\begin{equation}\label{disks and cylinders}
  \begin{split}
  &\D_r^\nu(x)=\big\{y\in\R^{n+1}:y=x+z\,,z\cdot\nu=0\,,|z|<r\big\}
  \\
  &\C_r^\nu(x)=\big\{y\in\R^{n+1}:y=x+z+t\nu\,,z\in\D_r^\nu\,,|t|<r\big\}\,.
  \end{split}
\end{equation}

\begin{proposition}\label{prop tamanini open set c1alfa}
  If $\Om$ is an open set with $C^{1,\a}$-boundary in the open set $A$, then for every $A'\cc A$ there exist $\Lambda\ge0$ and $r_0>0$ such that $\Om$ is a $(\Lambda,r_0,\a)$-minimizer in $A'$.
\end{proposition}

\begin{proof}[Proof of Proposition \ref{prop tamanini open set c1alfa}]
  By definition, the fact that $\Om$ is an open set with $C^{1,\a}$-boundary in $A$ means that for every $x_0\in A\cap\pa\Om$ there exist $r_0>0$, $\nu_0\in\SS^n$ and $u_0\in C^{1,\a}(\D_{2r_0}^{\nu_0})$ such that $u_0(0)=0$, $\nabla u_0(0)=0$, $\Lip(u_0)\le 1$ and
  \[
  \Om\cap\C_{2r_0}^{\nu_0}(x_0)=\big\{y\in\R^{n+1}:y=x_0+z+t\nu_0\,,z\in\D_{2r_0}^{\nu_0}\,,u_0(z)<t<2r_0\big\}\,.
  \]
  Since $A'\cc A$ we can consider a same value of $r_0$ for every $x_0\in A'\cap\pa\Om$, and also require that the $4\,r_0$-neighborhood of $A'$ is compactly contained in $A$.

  Now let $F$ be such that $F\Delta \Om\cc W$ for some set $W\cc A'$ with $\diam(W)=s<r_0$. Since we aim to prove
  \begin{equation}
    \label{tamanini check}
      P(\Om;W)\le P(F;W)+\Lambda\,s^{n+2\a}\,,
  \end{equation}
  we can assume that $W\cap\pa\Om\ne\emptyset$, for otherwise $P(\Om;W)=0$. Thus we can find $x_0\in A'\cap\pa\Om$ such that
  \[
  W\subset B_s(x_0)\cc \C_r^{\nu_0}(x_0)\qquad\forall r>s\,.
  \]
  Let $r\in(s,2r_0)$. By applying the divergence theorem on $F\cap\C_{r}^{\nu_0}(x_0)$ to the vector field $X(x)=\vphi(x)\,\nu_0$, where $\vphi\in C^\infty_c(A)$ with $0\le\vphi\le 1$ and $\vphi=1$ on the $4r_0$-neighborhood of $A'$, we find that
  \begin{equation}
    \label{zero divergence}
      0=\int_{F\cap \C_{r}^{\nu_0}(x_0)}\Div X=\int_{\pa^*(\C_{r}^{\nu_0}(x_0)\cap F)}\nu_0\cdot\nu_{\C_{r}^{\nu_0}(x_0)\cap F}\,.
  \end{equation}
  Now, for a.e. $r\in (s,2r_0)$ we have
  \begin{eqnarray}
    \pa^*\big(\C_{r}^{\nu_0}(x_0)\cap F\big)&\underset{\H^n}{=}&\big(\C_{r}^{\nu_0}(x_0)\cap\pa^*F\big)
    \cup\big(F^{(1)}\cap\pa\C_r^{\nu_0}(x_0)\big)
    \\
    F^{(1)}\cap\pa\C_r^{\nu_0}(x_0)&\underset{\H^n}{=}&\Om\cap\pa\C_r^{\nu_0}(x_0)
  \end{eqnarray}
  and
  \begin{equation*}
  \nu_0\cdot\nu_{\C_{r}^{\nu_0}(x_0)\cap F}(y)=\left\{\begin{split}
    &\nu_0\cdot\nu_F(y)\,,&\mbox{for $\H^n$-a.e. $y\in \C_{r}^{\nu_0}(x_0)\cap\pa^*F$}
    \\
    &1\,,&\mbox{for $\H^n$-a.e. $y\in F^{(1)}\cap\D_r^{\nu_0}(x_0+r\nu_0)$}
    \\
    &0\,,&\mbox{at $\H^n$-a.e. other point $y\in\pa^*\big(\C_{r}^{\nu_0}(x_0)\cap F\big)$}
  \end{split}\right .
  \end{equation*}
  see, e.g., \cite[Chapter 16]{maggiBOOK}. Therefore,
  \[
  \int_{\pa^*(\C_{r}^{\nu_0}(x_0)\cap F)}\vphi\,\nu_0\cdot\nu_{\C_{r}^{\nu_0}(x_0)\cap F}
  =-\H^n\big(\Om\cap\D_r^{\nu_0}(x_0+r\nu_0)\big)+
  \int_{\C_{r}^{\nu_0}(x_0)\cap \pa^*F}\nu_0\cdot\nu_F
  \]
  and \eqref{zero divergence} gives
  \[
   P(F;\C_r^{\nu_0}(x_0))\ge\H^n\big(\Om\cap\D_r^{\nu_0}(x_0+r\nu_0)\big)=\H^n(\D_r^{\nu_0})\,.
  \]
  We thus find
  \begin{eqnarray*}
    P(\Om;W)-P(F;W)&=&P(\Om;\C_r^{\nu_0}(x_0))-P(F;\C_r^{\nu_0}(x_0))
    \\
    &=&\int_{\D_r^{\nu_0}}\sqrt{1+|\nabla u_0|^2}-P(F;\C_r^{\nu_0}(x_0))
    \\
    &\le&\int_{\D_r^{\nu_0}}\big(\sqrt{1+|\nabla u_0|^2}-1\big)\le\,C(n)\,r^n\,\|\nabla u_0\|_{C^0(\D_r^{\nu_0})}^2\,.
  \end{eqnarray*}
  where in the last step we have used $\sqrt{1+|\xi|^2}-1\le |\xi|^2/2$ for every $\xi\in\R^n$. Since $\nabla u_0(0)=0$, we conclude that
  \[
  \|\nabla u_0\|_{C^0(\D_r^{\nu_0})}^2\le C\,r^{2\a}
  \]
  for a constant $C$ depending on the $\a$-H\"older semi-norm of $\nabla u_0$ on $\D_r^{\nu_0}$. Combining everything together we have proved
  \[
  P(\Om;W)\le P(F;W)+\Lambda\,r^{n+2\a}
  \]
  for a.e. $r\in(s,2r_0)$. Letting $r\to s^+$ we conclude the proof of \eqref{tamanini check}.
\end{proof}

We conclude this section with another useful technical remark.

\begin{proposition}\label{prop mean curvature from AminusSigma to A}
  If $E$ is a $(\Lambda,r_0,\a)$-minimizer in an open set $A\subset\R^{n+1}$, $\S$ is the singular set of $E$ in $A$, and $H_E$ is the generalized mean curvature of $E$ in $A\setminus\S$, then $H_E$ (extended to constantly take the value $0$ on $\S$, say) is the generalized mean curvature of $E$ in $A$.
\end{proposition}

\begin{proof}[Proof of Proposition \ref{prop mean curvature from AminusSigma to A}]
  This is based on a standard cut-off function and covering argument based on the fact that the Hausdorff dimension of $\Sigma$ is at most $n-7$ and, by $(\Lambda,r_0,\alpha)$-minimality in $A$, $P(E;B_r(x_0)) \leq C(n,\Lambda,r_0,\alpha) \, r^n$ for any ball $B_r(x_0) \cc A$.  We omit the details.
\end{proof}

\section{Structure of sets with small $\de(\Om)$ and reduction to normal graphs}\label{section structure and reduction} This section is devoted to discussing the reduction to normal graphs over $\SS^n$ in the stability problem for Almgren's isoperimetric principle. There is a first interesting observation, which is based on the simple idea of applying Almgren's principle to each connected component of $\pa\Om$, and leads to a sharp structural decomposition result under the quite explicit assumption that $\de(\Om)<P(B_1)$, or, equivalently, that $P(\Om)<2\,P(B_1)$. This argument is presented in Proposition \ref{prop reduction connected} below. This result allows one to focus on the case of a simply connected set $\Om$ with connected boundary. The mean curvature of $\pa\Om$ is bounded from above, but not from below. This is unavoidable, even with arbitrarily small deficit. However, we can construct a subset $E$ of $\Om$, whose boundary has bounded mean curvature and largely overlaps with $\pa\Om$. If $\de(\Om)$ is small enough, $\pa E$ will be a normal graph over $\SS^n$, described by a function $u$ with small $C^1$ norm. For this kind of boundaries we can obtain a sharp stability theory by mixing spectral analysis, elliptic regularity, and interpolation inequalities, see section \ref{section normal def of SSn}. The construction of $E$, starting from $\Om$ with small deficit, is discussed in Proposition \ref{prop E construction} below. It is based on the regularity theories for perimeter almost-minimizers and for free-boundary problems. This result seems to have an independent interest, as it should be applicable to other variational problems where one needs to truncate mean curvature.

\subsection{Applying Almgren's principle to the components of a boundary} Here we prove the following proposition, which takes care of the first part of the statement of Theorem \ref{thm structure}.

\begin{proposition}\label{prop reduction connected}
  If $\Om$ is an open bounded set in $\R^{n+1}$ with smooth boundary such that $H_\Om\le n$ and $\de(\Om)<P(B_1)$, then there exists an open bounded connected set $\Om^*$ with smooth, connected boundary such that
  \[
  |\Om^*\setminus \Om|\le C(n)\,\de(\Om)^{(n+1)/n}\qquad \Om\subset\Om^*\,.
  \]
  Moreover, $\pa\Om^*\subset \pa\Om$, so that, in particular,
  \[
  \mbox{$H_{\Om^*}\le n$ and $\de(\Om^*)\le\de(\Om)$}\,.
  \]
\end{proposition}

\begin{proof}[Proof of Proposition \ref{prop reduction connected}]
Let $\{A^j\}_{j\in J}$ be the family of the connected components of $\Om$. Clearly we can apply \eqref{almgren principle} to each $A^j$. As a consequence
\[
P(\Om)=\sum_{j\in J}P(A^j)\ge \#J\,P(B_1)
\]
so that if $\de(\Om)<P(B_1)$, then $\#\,J=1$. In other words, $\Om$ is connected.

Now let $\{S^i\}_{i\in I}$ be the family of the connected components of $\pa\Om$. Each $S^i$ is a compact, connected, orientable hypersurface in $\R^{n+1}$ such that $S^i=\pa \Om^i$ for an open set with smooth boundary $\Om^i$ such that $|\Om^i|<\infty$. Now, by continuity, either $\nu_{\Om^i}=\nu_{\Om}$ on $S^i$ or $\nu_{\Om^i}=-\nu_{\Om}$ on $S^i$, and accordingly we define a partition $\{I^+,I^-\}$ of $I$. If $i\in I^+$, then the mean curvature $H_{S^i}$ of $S^i$ computed with respect to $\nu_{\Om^i}$ satisfies $H_{S^i}\le n$ on $S^i$, and thus by Almgren's isoperimetric principle
\[
\H^n(S^i)\ge P(B_1)\,.
\]
Since $\de(\Om)<P(B_1)$ this means that $\#\,I^+\le 1$. By sliding an hyperplane from infinity until it touches $S^i$ we find that $I^+\ne\emptyset$, and thus $\#\,I^+=1$. Now, if $S^*$ denotes the only element of $\{S^i\}_{i\in I^+}$, and $\Om^*$ is the bounded open set with finite volume such that $\nu_{\Om^*}=\nu_{\Om}$, then
\[
\Om=\Om^*\setminus\bigcup_{i\in I^-}\Om^i\subset\Om^*\,,
\]
and
\[
\de(\Om)=P(\Om^*)-P(B_1)+\sum_{i\in I^-}P(\Om^i)\ge\sum_{i\in I^-}P(\Om^i)\ge c(n)\,\sum_{i\in I^-}|\Om^i|^{n/(n+1)}
\]
so that
\[
|\Om^*\setminus\Om|=\sum_{i\in I^-}|\Om^i|\le C(n)\,\de(\Om)^{(n+1)/n}\,.
\]
Finally, since $\pa\Om^*$ is connected, we have that $\Om^*$ is connected.
\end{proof}

\subsection{Truncating the mean curvature of a set}  The following result is particularly useful in ``truncating the mean curvature of a set''. The result itself will probably not be surprising for experts in the obstacle problem, but we have included a detailed proof for the sake of clarity.

\begin{proposition}
  \label{prop E construction} If $\l>0$, $\a\in(0,1)$ and $\Om$ is an open bounded set with $C^{1,\a}$-boundary in $\R^{n+1}$, then there exist minimizers in the variational problem
  \begin{equation}
    \label{variational problem lambda}
      \inf\Big\{P(E)+\l\,|E|:\Om\subset E\,,|E|<\infty\Big\}\,.
  \end{equation}
  If $E_\l$ is one such minimizer, then:
  \begin{enumerate}
    \item[(i)] $E_\l$ is contained in the convex envelope of $\Om$ with $\diam(E_\l)=\diam(\Om)$ and
    \begin{equation}
      \label{Elambda distance from Omega}
      \l\,|E_\l\setminus\Om|+\H^n(\pa^*E_\l\setminus\pa\Om)\le\de(\Om)\,.
    \end{equation}
    \item[(ii)] there exists a closed set $\S\subset\pa E_\l$ such that $E_\l$ is an open set with $C^{1,\b}$-boundary in $\R^{n+1}\setminus\S$ for some $\b\in(0,1)$. In particular, $\H^n(\pa E_\l\Delta\pa^*E_\l)=0$.
    \item[(iii)] if $\Om$ has $C^{2,\b}$-boundary in $\R^{n+1}$, then $E_\l$ is an open set with $C^{1,1}$-boundary in $\R^{n+1}\setminus\S$, and  $E_\l$ has generalized bounded mean curvature in $\R^{n+1}$ satisfying
    \begin{equation}
      \label{bound on HEl}
          \|H_{E_\l}\|_{L^\infty(\pa E_\l)}\le\max\{\|(H_\Om)^+\|_{C^0(\pa\Om)},\l\}\,.
    \end{equation}
  \end{enumerate}
\end{proposition}

\begin{proof}[Proof of Proposition \ref{prop E construction}]
  {\it Step one}: We prove the existence of $E_\l$ and conclusion (i). First, we notice that the infimum in \eqref{variational problem lambda} is finite, as $\Om$ itself is a competitor with finite energy. The convex hull $A$ of $\Om$ is bounded, and energy is decreased by intersecting $E$ with $A$. Thus we can minimize over $E\subset A$, and by standard lower semicontinuity and compactness properties of perimeter, there exists at least a minimizer $E_\l$. Since $\Om\subset E_\l\subset A$ we have $\diam(E_\l)=\diam(\Om)$. By testing $E_\l$ against $\Om$ we find
  \[
  \l\,|E_\l\setminus\Om|\le P(\Om)-P(E_\l)=P(\Om;E_\l^{(1)})-\H^n(\pa^* E_\l\setminus\pa\Om)\,,
  \]
  where we have used $P(\Om)=P(\Om;E_\l^{(1)})+\H^n(\pa\Om\cap\pa^*E_\l)$. By $E_\l\subset A$ we find $E_\l^{(1)}\subset A^{(1)}$, and thus
  \[
  P(\Om;E_\l^{(1)})\le P(\Om;A^{(1)})=P(\Om)-\H^n(\pa\Om\cap\pa A)=\H^n(\pa\Om\setminus\pa A)\le\de(\Om)\,,
  \]
  thanks to \eqref{almgren identity}. This proves \eqref{Elambda distance from Omega}.

  \medskip

  \noindent {\it Step two}: We prove conclusion (ii). By Proposition \ref{prop tamanini open set c1alfa} there exist positive constants $r_0$ and $\Lambda$ such that
  \begin{equation}
    \label{omega quasiminimo}
    P(\Om;V)\le P(H;V)+\Lambda\,r^{n+2\a}
  \end{equation}
  whenever $H\Delta\Om\cc W$ with $\diam(W)=r<r_0$. Let us consider a set $F$ such that $F\Delta E_\l\cc W$ for a bounded open set $W$, and set $r=\diam(W)<r_0$. Let us assume first that
  \begin{equation}
    \label{hp W}
    \H^n(W\cap\pa\Om)=\H^n(W\cap\pa E_\l)=0\,.
  \end{equation}
  If we set $H=F\cap\Om$, then we have that
  \[
  H\Delta \Om=\Om\setminus(F\cap\Om)=\Om\setminus F\subset E_\l\setminus F\cc W
  \]
  so that if $r<r_0$, then by \eqref{omega quasiminimo}
  \[
  P(\Om;W)\le P(F\cap\Om;W)+\Lambda\,r^{n+2\a}\,.
  \]
  Since $F\cap\Om\cap W^c=E_\l\cap\Om\cap W^c=\Om\cap W^c$, thanks to \eqref{hp W} we actually have
  \[
  P(F\cap\Om)-P(\Om)=P(F\cap\Om;W)-P(\Om;W)
  \]
  and we have thus obtained
  \begin{equation}
    \label{tamanini 1}
      P(\Om)\le P(F\cap\Om)+\Lambda\,r^{n+2\a}\,.
  \end{equation}
  At the same time $F\cup\Om$ is admissible in \eqref{variational problem lambda}, thus by the general inequality
  \[
  P(N\cap M)+P(N\cup M)\le P(N)+P(M)\qquad\forall N,M\subset\R^{n+1}\,,
  \]
  we get
  \begin{equation}
    \label{tamanini 2}
      P(E_\l)\le P(F)+P(\Om)-P(F\cap\Om)+\l\,|E_\l\Delta(F\cup\Om)|\,.
  \end{equation}
  By $F\Delta E_\l\cc W$ we have $P(F)-P(E_\l)=P(F;W)-P(E_\l;W)$, while $E_\l\Delta(F\cup\Om)\subset F\Delta E_\l\cc W$ gives us $|E_\l\Delta(F\cup\Om)|\le C(n)\,r^{n+1}$, so that \eqref{tamanini 1} and \eqref{tamanini 2} imply
    \begin{equation}
    \label{tamanini 3}
      P(E_\l;W)\le P(F;W)+C(n,\l)\,r^{n+1}+\Lambda\,r^{n+2\a}\,.
  \end{equation}
  We have thus proved that $E_\l$ is a $(\Lambda',r_0,\min\{1/2,\a\})$-minimizer in $\R^{n+1}$. As a consequence there exists a closed set $\S\subset\pa E_\l$ such that $E_\l$ is an open set with $C^{1,\b}$-boundary on $\R^{n+1}\setminus\S$ for $\b=\min\{1/2,\a\}$.

  \medskip

  \noindent {\it Step three}: We prove statement (iii). By a first variation argument based on the minimality of $E_\l$ in \eqref{variational problem lambda} one finds that
  \begin{equation} \label{ob3}
	\int_{\pa E_\l}\Div^{\pa E_\l}X\ge-\l\int_{\pa E_\l}X\cdot\nu_{E_\l}
  \end{equation}
  for every $X\in C^\infty_c(\R^{n+1};\R^{n+1})$ with $X\cdot\nu_{E_\l}\ge0 $ on $\pa E_\l$: that is, $H_{E_\l}\ge-\l$ on $\pa E_\l$ in distributional sense. More precise information is found by considering the open sets
  \begin{eqnarray*}
	A_1 &=& \big\{ x \in\R^{n+1} : \mbox{$\exists r>0$ s.t. $B_r(x)\cap\pa E_\l\subset\pa\Om$} \big\}
    \\
	A_2 &=& \big\{ x \in\R^{n+1}\setminus \S_\l : \mbox{$\exists r>0$ s.t. $B_r(x)\cap\Om=\emptyset$} \big\}=\R^{n+1}\setminus\big(\S_\l\cup\ov{\Om}\big)
  \end{eqnarray*}
  Since $\Om$ has generalized bounded mean curvature in $\R^{n+1}$ (recall the discussion around \eqref{formula HOm for Om C11}), we find that $E_\l$ has generalized bounded mean curvature in $A_1$ satisfying
  \begin{equation}
    \label{ob1}
      H_{E_\l}=H_{\Om}\qquad\mbox{on $A_1\cap\pa E_\l$}\,.
  \end{equation}
  Moreover, again by a first variation argument based on its minimality in \eqref{variational problem lambda}, $E_\l$ has generalized bounded mean curvature in $A_2$ given by
  \begin{equation}
    \label{ob2}
      H_{E_\l}=-\l\qquad\mbox{on $A_2\cap\pa E_\l$}\,.
  \end{equation}
  (In particular, $E_\l$ has smooth boundary in $A_1\cup A_2$.) We now pick
  \begin{equation}\label{x0}
  x_0\in\pa E_\l\setminus(\S\cup A_1\cup A_2)\,.
  \end{equation}
  and claims that there exists $\rho>0$ such that $E_\l$ has $C^{1,1}$-boundary in $B_\rho(x_0)$ with
  \[
  \|H_{E_\l}\|_{L^\infty(B_\rho(x_0)\cap\pa E_\l)}\le \max\big\{\l,\|(H_\Om)^+\|_{C^0(B_\rho(x_0)\cap\pa\Om)}\big\}\,.
  \]
  By combining this claim with \eqref{ob1} and \eqref{ob2}, we shall conclude that $E_\l$ has generalized bounded mean curvature in $\R^{n+1}\setminus \Sigma$. Thanks to Proposition \ref{prop mean curvature from AminusSigma to A} this last fact will complete the proof of step three.

  Given $x_0$ as in \eqref{x0}, since $x_0\not\in\S$, up to a translation and a rotation (so that $x_0=0$ and $\nu_{E_\l}(0)=-e_{n+1}$) we have that there exist $r>0$ and
  \begin{eqnarray*}
  u\in C^{1,1}(\R^n)\,,\qquad u(0)=0\,,\qquad  \nabla u(0)=0\,,\qquad \Lip(u)\le 1\,,
  \\
  v\in C^{1,\b}(\R^n)\,,\qquad v(0)=0\,,\qquad\nabla v(0)=0\,,\qquad\Lip(v)\le 1\,,
  \end{eqnarray*}
  such that, setting
  \[
  \D_r(x)=\D_r^{e_{n+1}}(x)\qquad\D_r=\D_r(0)\qquad \C_r(x)=\C_r^{e_{n+1}}(x)\qquad \C_r=\C_r(0)\,,
  \]
  (see \eqref{disks and cylinders} for the notation used here) then we have
  \begin{eqnarray*}
	\Om\cap\C_r &=& \big\{ (x,x_{n+1}) \in \C_r : x_{n+1} < u(x) \big\}
  \\
	E_\l\cap\C_r &=& \big\{ (x,x_{n+1}) \in \C_r : x_{n+1} < v(x) \big\}
  \\
  \big(\pa E_\l\setminus\pa\Om\big)\cap\C_r&=&\big\{(x,v(x)) \in \C_r:v(x)>u(x)\big\}
  \\
  \pa E_\l\cap\pa\Om\cap\C_r&=&\big\{(x,v(x)) \in \C_r:v(x)=u(x)\big\}\,.
  \end{eqnarray*}
  Since $\Om\subset E_\l$ we have $u \le v$ on $\D_r$, where thanks to \eqref{ob1}, \eqref{ob2} and \eqref{ob3} it holds
  \begin{eqnarray}
    -\Div\Big(\frac{\nabla v}{\sqrt{1+|\nabla v|^2}}\Big) \geq -\l&&\qquad\mbox{weakly on $\D_r$}  \label{obstacle_H_leq}
    \\
    -\Div\Big(\frac{\nabla v}{\sqrt{1+|\nabla v|^2}}\Big) =- \l&&\qquad\mbox{strongly on $\D_r \cap \{v<u\}$} \label{obstacle_H_equal}
    \\
    -\Div\Big(\frac{\nabla u}{\sqrt{1+|\nabla u|^2}}\Big)=h&&\qquad\mbox{strongly $\D_r$}\,.\label{obstacle Omega inequality}
  \end{eqnarray}
  Here, $h(x)=H_\Om(x+u(x)\,e_{n+1})$ for every $x\in\D_r$. We now claim that there exist
  \begin{eqnarray}
    \label{s dependency}
      s_0=s_0(n,\l,r,\|\nabla^2u\|_{C^0(\D_r)})\in(0,r/4)\qquad C_0=C_0(n,\l,r,\|\nabla^2u\|_{C^0(\D_r)})
  \end{eqnarray}
  such that
  \begin{equation}
    \label{caff1}
    \sup_{\D_s(y)}(v-u)\le C_0\,s^2\qquad\forall s\in(0,s_0)\,,y\in \D_{s_0}\cap\{u=v\}\,.
  \end{equation}
  We prove this by a classical barrier argument from the regularity theory of obstacle problems, see~\cite[Theorem 2, Lemma 2]{caffarelliOPrevisited}. Our barriers will be given by spherical caps. Let us fix $y$ as in \eqref{caff1}, and set
  \begin{eqnarray}\label{xi definition}
  \xi(y)&=&y - \frac{n}{\l} \frac{\nabla u(y)}{\sqrt{1+|\nabla u(y)|^2}}
  \\\label{psi definition}
  \psi_y(x)&=&u(y)-\sqrt{(n/\l)^2 - |x-\xi(y)|^2} + \sqrt{(n/\l)^2 - |y-\xi(y)|^2}\,,\qquad x\in\R^n
  \end{eqnarray}
  so that the graph of $\psi_y$ over $\D_{n/\l}(\xi(y))$ is a half-sphere of radius $1/\l$, which is tangent to the graph of $u$ at the point $y+u(y)\,e_{n+1}$ thanks to \eqref{xi definition}: in particular
  \begin{gather}
  \label{psi tangency}
    \psi_y(y)=u(y)\qquad\nabla\psi_y(y)=\nabla u(y)\,,
    \\
    \label{psi equation}
	-\Div\Big( \frac{\nabla \psi_y}{\sqrt{1+|\nabla \psi_y|^2}} \Big) =- \l \qquad\mbox{on $\D_{n/\l}(\xi(y))$}\,.
  \end{gather}
  Notice that, thanks to \eqref{s dependency}, we can entail
  \begin{equation}
    \label{xi cont}
    \D_{s_0}(y)\subset \D_{n/2\l}(\xi(y))\cap \D_r\,.
  \end{equation}
  Indeed if $x\in\D_{s_0}(y)$, then by $y\in\D_{s_0}$, $\nabla u(0)=0$ and \eqref{s dependency} (with $s_0$ suitably small) we have
  \[
  |x-\xi(y)|\le s_0+|y-\xi(y)|\le s_0+\frac{n}\l\,|\nabla u(y)|\le s_0+\frac{n}\l\,\|\nabla^2u\|_{C^0(\D_r)}\,s_0<\frac{n}{2\l}\,.
  \]
  This guarantees that $\psi_y$ is well-defined on $\D_{s_0}(y)\subset \D_r$. We also observe that for every $s\in(0,s_0)$,
  \begin{equation} \label{obstacle_reg_eqn4}
  \psi_y - C_0 \, s^2 \leq v\qquad\mbox{on $\D_s(y)$}\,.
  \end{equation}
  Indeed if $x\in\D_s(y)$ with $s\in(0,s_0)$, then by \eqref{xi cont} and \eqref{psi tangency}
  \begin{eqnarray*}
    v(x)&\ge&u(x)\ge u(y)+\nabla u(y)\cdot(x-y)-\|\nabla^2u\|_{C^0(\D_r)}\,\frac{s^2}2
    \\
    &=&\psi_y(y)+\nabla \psi_y(y)\cdot(x-y)-\|\nabla^2u\|_{C^0(\D_r)}\,\frac{s^2}2
    \\
    &\ge&\psi_y(x)-C_0\,s^2
  \end{eqnarray*}
  since $x\in\D_{s_0}(y)\subset \D_{n/2\l}(\xi(y))$, where
  \[
  \|\nabla^2\psi_y\|_{C^0(\D_{n/2\l}(\xi(y)))}\le C(n,\l)\,.
  \]
  Thanks to \eqref{obstacle_reg_eqn4}, if we set
  \begin{equation*}
	w= v - \psi_y+ C_0 \, s^2\qquad\mbox{on $\D_s(y)$}
  \end{equation*}
  then $w\ge0$ on $\D_{s_0}(y)$. By \eqref{obstacle_H_leq} and \eqref{psi equation}, there exists a matrix-field $A\in C^{0,\b}(\D_{s_0}(y);\R^{n\times n}_{{\rm sym}})$ with
  \begin{eqnarray*}
    \|A\|_{C^{0,\b}(\D_{s_0}(y))}\le K\qquad \frac{\Id}{K}\le A(x)\le \Id\,,\quad\forall x\in\D_{s_0}(y)\,,
  \end{eqnarray*}
  (where here and in the following $K$ denotes a generic positive constant depending on $n$, $\l$, $r$, $[\nabla v]_{C^{0,\b}(\D_r)}$ and $\|\nabla^2u\|_{C^0(\D_r)}$) such that
  \begin{eqnarray}\label{obstacle_reg_eqn6}
    \Div(A\,\nabla w)\le 0\qquad\mbox{weakly on $\D_{s_0}(y)$}\,;
  \end{eqnarray}
  and, thanks to \eqref{obstacle_H_equal},
  \begin{equation}\label{obstacle_reg_eqn5}
    \Div(A\,\nabla w)=0\qquad\mbox{strongly on $\D_{s_0}(y)\cap\{u>v\}$}\,.
  \end{equation}
  Let $w_1$ be the solution to
  \begin{equation}\label{obstacle_reg_eqn7}
    \left\{\begin{split}
    &\Div(A\,\nabla w_1)= 0&\qquad\mbox{in $\D_{s}(y)$}\,,
    \\
    &w_1 = w&\qquad\mbox{on $\pa\D_{s}(y)$}\,.
    \end{split}\right .
  \end{equation}
  By the weak maximum principle, \eqref{obstacle_reg_eqn6}, \eqref{obstacle_reg_eqn7}, and $w\ge 0$ on $\D_{s_0}(y)$,
  \begin{equation*}
	0 \leq w_1 \leq w\qquad\mbox{on $\D_{s}(y)$}\,.
  \end{equation*}
  By the Harnack inequality, $w_1 \leq w$, and $u(y) = v(y) = \psi(y)$, for every $s\in(0,s_0)$ we have
  \begin{equation*}
	\sup_{\D_s(y)} w_1 \leq K \, w_1(y) \leq K\,w(y)=K\,\big( v(y) - \psi(y) + C_0 \, s^2 \big)\le K \, s^2 \,.
  \end{equation*}
  By the strong maximum principle, \eqref{obstacle_reg_eqn5}, and \eqref{obstacle_reg_eqn7}, $w_2=w-w_1$ attains its maximum over the closure of $\D_s(y)\cap\{u<v\}$ at some point $x_1\in \D_{s}(y) \cap \{u = v\}$.  By $w_2 \leq w$, $u(x_1) = v(x_1)$, $u(y) = \psi(y)$, and $\nabla u(y) = \nabla \psi(y)$,
  \begin{equation*}
	\sup_{\D_{s}(y)} w_2 = w_2(x_1) \leq w(x_1) = v(x_1) - \psi(x_1) + C_0 \, s^2 = u(x_1) - \psi(x_1) + C_0 \, s^2 \leq K \, s^2\,.
  \end{equation*}
  By combining these last two estimates we have proved \eqref{caff1}.

  Now let us pick $x\in\D_{s_0/2}\cap\{v>u\}$ and let $y$ be the closest point to $x$ in $\{u=v\}$. Since $u(0)=v(0)$, setting $s=|x-y|$ we have $s<s_0/2$. Moreover, considering that $v$ is smooth inside $\D_r\cap\{u<v\}$, by \eqref{obstacle_H_equal} we find that
  \[
  \sum_{i,j=1}^na_{i,j}\,D^2_{ij}v=\l\qquad\mbox{on $\D_r\cap\{u<v\}$}
  \]
  where
  \[
  a=\frac{(1+|\nabla v|^2)\,\Id-\nabla v\otimes\nabla v}{(1+|\nabla v|^2)^{3/2}}\,.
  \]
  In particular,
  \[
  \sum_{i,j=1}^na_{i,j}\,D^2_{ij}(v-u)=f\qquad\mbox{on $\D_r\cap\{u<v\}$}
  \]
  where
  \[
  f=\l-\sum_{i,j=1}^na_{i,j}\,D^2_{ij}u\,.
  \]
  Since $v\in C^{1,\b}(\D_r)$, the matrix field $a$ satisfies
  \begin{equation}
    \label{a C0beta}
     \|a\|_{C^0(\D_r)}\le N=N(n)\,,\qquad [a]_{C^{0,\b}(\D_r)}\le N=N(n,[\nabla v]_{C^{0,\b}(\D_r)})
  \end{equation}
  and it is uniformly elliptic on $\D_r$, with ellipticity constant independent even from the dimension $n$ thanks to $\Lip(v)\le 1$. Similarly, $f\in C^{0,\b}(\D_r)$ with
  \begin{equation}
    \label{f C0beta}
      \|f\|_{C^{0,\b}(\D_r)}\le N=N(n,\l,\|u\|_{C^{2,\beta}(\D_r)},[\nabla v]_{C^{0,\b}(\D_r)})\,.
  \end{equation}
  By Schauder's theory for equations in non-divergence form applied on $\D_s(x)\subset\D_r\cap\{u<v\}$, see \cite[Corollary 6.3]{gt} we find
  \[
  \|\nabla^2(v-u)\|_{C^0(\D_{s/2}(x))}\le N(n,\b,r^\b\,[a]_{C^{0,\b}(\D_r(x))})\,\bigg(\frac{\|v-u\|_{C^0(\D_s(x))}}{s^2}+s^{\b}\,\|f\|_{C^{0,\b}(\D_{2s}(x))}\bigg)\,.
  \]
  By combining \eqref{s dependency}, \eqref{caff1}, \eqref{a C0beta} and \eqref{f C0beta} we find in particular that
  \[
  |\nabla^2v(x)|\le C(n,\b,\l,\|u\|_{C^{2,\beta}(\D_r)},[\nabla v]_{C^{0,\b}(\D_r)})\qquad\forall x\in\D_{s_0/2}\cap\{v>u\}\,.
  \]
  Thus $v\in C^{1,1}(\D_{s_0/2})=W^{2,\infty}(\D_{s_0/2})$, so that $E_\l$ has $C^{1,1}$-boundary in $\C_{s_0/2}$ and generalized bounded mean curvature in $\C_{s_0/2}$ satisfying
  \begin{equation}
    \label{caff2}
      H_{E_\l}(x+v(x)\,e_{n+1})=-\frac{(1+|\nabla v(x)|^2)\,\Delta v(x)-\nabla^2v(x)[\nabla v(x)]\cdot\nabla v(x)}{(1+|\nabla v(x)|^2)^{3/2}}
  \end{equation}
  for a.e. $x\in\D_{s_0/2}$ (thanks to \eqref{formula HOm for Om C11}). Notice that $\nabla u=\nabla v$ on $\{u=v\}\cap\D_{s_0/2}$ and $\nabla^2u(x)=\nabla^2v(x)$ for a.e. $x\in\{u=v\}\cap \D_{s_0/2}$. In particular \eqref{caff2} gives us
  \[
  H_{E_\l}=H_\Om\qquad\mbox{$\H^n$-a.e. on $\C_{s_0/2}\cap\pa E_\l\cap\pa\Om$}
  \]
  By a covering argument, and by taking into account that $E_\l$ is smooth on $A_1\cup A_2$ with $H_{E_\l}=H_\Om$ on $A_1\cap\pa E_\l$ and $H_{E_\l}=-\l$ on $A_2\cap\pa E_\l$, we conclude that $E_\l$ has $C^{1,1}$-boundary in $\R^{n+1}\setminus\S$, generalized bounded mean curvature in $\R^{n+1}$ which satisfies $H_{E_\l}(x)\in\{H_\Om(x),-\l\}$ for $\H^n$-a.e. $x\in\pa E_\l$. Recalling \eqref{ob3}, we also have $H_{E_\l}(x)\ge-\l$
   for $\H^n$-a.e. $x\in\pa E_\l$. This completes the proof of statement (iii), thus of the proposition.
\end{proof}

%
%
%

\begin{proof}[Proof of Theorem \ref{thm structure}]
  The first part of the statement, requiring the explicit assumption $\de(\Om)< P(B_1)$ only, was proved in Proposition \ref{prop reduction connected}.
  To complete the proof of the theorem, let us consider a sequence $\{\Om_h\}_{h\in\N}$ open bounded sets in $\R^{n+1}$ with smooth boundaries, such that $H_{\Om_h}\le n$ for every $h\in\N$, and
  \[
  \de(\Om_h)=P(\Om_h)-P(B_1)\to 0\qquad\mbox{as $h\to\infty$}\,.
  \]
  Let $\{E_h\}_{h\in\N}$ be the sequence of sets associated to $\Om_h$ by Proposition \ref{prop E construction} with $\l=n$. In this way, each $E_h$ is an open set with $C^{1,1}$-boundary in $\R^{n+1}\setminus\S_h$ (where $\S_h$ are closed sets with Hausdorff dimension at most $n-7$) and bounded generalized mean curvature in $\R^{n+1}$ satisfying
  \begin{equation}
    \label{HEh bounded by n}
      \|H_{E_h}\|_{L^\infty(\pa E_h)}\le\max\big\{\|(H_{\Om_h})^+\|_{C^0(\pa\Om_h)},n\big\}=n\,.
  \end{equation}
  By the monotonicity formula for rectifiable sets with bounded generalized mean curvature \cite[Chapter 17]{SimonLN}, we have
  \[
  \H^n(B_r(x)\cap\pa E_h)\ge c(n)\,r^n\qquad\forall r<r(n)\,,
  \]
  so that, by a covering argument,
  \[
  \diam(E_h)\le C(n)\,P(E_h)\,.
  \]
  Now, by Proposition \ref{prop E construction}, $\diam(\Om_h)=\diam(E_h)$. At the same time, $\Om_h\subset E_h$ and \eqref{Elambda distance from Omega} imply
  \[
  P(E_h)=\H^n(\pa E_h\cap\pa\Om_h)+\H^n(\pa E_h\setminus\pa\Om_h)\le P(\Om_h)+\de(\Om_h)
  \]
  so that
  \[
  \limsup_{h\to\infty}P(E_h)\le P(B_1)\,,\qquad \diam(E_h)=\diam(\Om_h)\le C(n)\,.
  \]
  Let $A_h$ be the convex hull of $E_h$. Up a translation $0\in A_h$, so that
  $\nu_{A_h}(x)\cdot x\ge 0$ for every $x\in \pa A_h$, and in particular
  \begin{equation}
    \label{Eh pos}
      \nu_{E_h}(x)\cdot x\ge0\qquad\forall x\in \pa A_h\cap\pa E_h\setminus\S_h\,.
  \end{equation}
  Since $\diam(E_h)\le C(n)$ and $0\in A_h$, we have $E_h\subset B_{C(n)}$ with $P(E_h)\le P(B_1)+1$ for every $h$ large enough. By the standard compactness theorem for sets of finite perimeter, there exists a bounded set of finite perimeter $E$ such that
  \[
  |E_h\Delta E|\to 0\qquad\mbox{as $h\to\infty$}\,,
  \]
  and hence, by lower semicontinuity of perimeter
  \begin{equation} \label{area bound}
  P(E)\le  P(B_1)\,.
  \end{equation}

  We now exploit the divergence theorem
  \[
  (n+1)|E_h|=\int_{\pa E_h}x\cdot\nu_{E_h}=\int_{\pa A_h\cap\pa E_h}x\cdot\nu_{E_h}+\int_{(\pa E_h)\setminus \pa A_h}x\cdot\nu_{E_h}
  \]
  where
  \[
  \Big|\int_{(\pa E_h)\setminus \pa A_h}x\cdot\nu_{E_h}\Big|\le \diam(E_h)\,\H^n(\pa E_h\setminus \pa A_h)\le C(n)\,\de(E_h)\,.
  \]
  where in the last step we have applied Almgren's identity \eqref{almgren identity} to $E_h$. Now, by \eqref{Eh pos} and since $n\ge H_{E_h}\ge0$ on $\pa A_h\cap\pa E_h\setminus\S_h$, for every $\e>0$ we have
  \[
  \int_{\pa A_h\cap\pa E_h}x\cdot\nu_{E_h}=\int_{\pa A_h\cap\pa E_h}\frac{H_{E_h}+\e}{{H_{E_h}+\e}}\,x\cdot\nu_{E_h}
  \ge\frac1{n+\e}\int_{\pa A_h\cap\pa E_h}(H_{E_h}+\e)\,(x\cdot\nu_{E_h})\,.
  \]
  By \eqref{HEh bounded by n},
  \[
  \Big|\int_{\pa E_h\setminus \pa A_h}(H_{E_h}+\e)\,(x\cdot\nu_{E_h})\Big|\le (n+\e)\,\diam(E_h)\,\H^n(\pa E_h\setminus \pa A_h)\le C(n)\,\de(E_h)
  \]
  and thus, combining the above identities and estimates,
  \[
  (n+1)|E_h|\ge \frac1{n+\e}\int_{\pa E_h}(H_{E_h}+\e)\,(x\cdot\nu_{E_h})-C(n)\,\de(E_h)\,.
  \]
  By the tangential divergence theorem
  \[
  \int_{\pa E_h}(H_{E_h}+\e)\,(x\cdot\nu_{E_h})=n\,P(E_h)+\e\,(n+1)\,|E_h|
  \]
  so that
  \[
  (n+1)|E_h|\ge \frac{n\,P(E_h)+\e\,(n+1)\,|E_h|}{n+\e}-C(n)\,\de(E_h)\,.
  \]
  We let $\e\to 0$ and apply Almgren's principle $P(E_h)\ge P(B_1)$ to conclude that
  \[
  (n+1)|E_h|\ge P(B_1)-C(n)\,\de(E_h)=(n+1)|B_1|-C(n)\,\de(E_h)\,,
  \]
  that is
  \begin{equation}
    \label{volume bound eqn1}
    |B_1|-|E_h|\le C(n)\,\de(E_h)\,.
  \end{equation}

  Let us now assume that $|E_h|> |B_1|$ and let $\l_h=(|B_1|/|E_h|)^{1/(n+1)}$ so that $|\l_h\,E_h|=|B_1|$ and thus $P(\l_h\,E_h)\ge P(B_1)$ by the isoperimetric inequality. By $1-\l_h^n\ge 1-\l_h$ we thus find
  \begin{eqnarray}
  \nonumber
	P(E_h)-P(B_1)=(1 - \lambda_h^n) \, P(E_h) + P(\lambda_h E_h)-P(B_1)
  \ge(1 - \lambda_h) \, P(E_h)\,.
  \end{eqnarray}
  Since $P(E_h)\to P(B_1)$ we conclude that
  \[
  C(n)\,\de(E_h)\ge 1-\l_h=1-\Big(\frac{|B_1|}{|E_h|}\Big)^{1/(n+1)}
  \]
  that is $|E_h|-|B_1|\le C(n)\,\de(E_h)$. Also taking \eqref{volume bound eqn1} into account we thus find
  \[
  \big||E_h|-|B_1|\big|\le C(n)\,\de(E_h)\,.
  \]
  In particular, $|E|=|B_1|$ and thus \eqref{area bound} and the isoperimetric theorem imply that $E=B_1$ (up to a final translation). Since $\pa E=\SS^n$ is a smooth hypersurface and $|E_h\Delta E|\to 0$ as $h\to\infty$ with $\|H_{E_h}\|_{L^\infty(\pa E_h)}\le n$, by applying Allard's regularity theorem we find that $\S_h=\emptyset$ and that there are maps $u_h:\SS^n\to\R$ such that
  \[
  \pa E_h=\big\{x+u_h(x)\,x:x\in\SS^n\big\}\qquad\lim_{h\to\infty}\|u_h\|_{C^1(\SS^n)}=0\,.
  \]
  (Referring to \cite[Lemma 2.8]{ciraolomaggi} or \cite[Lemma 4.4]{CiLeMaIC1} for more details on this point, we just mention that here the idea is that of exploiting the continuity of area excess on a fixed a cylinder along sequences of almost-minimizers. By choosing a scale such that the area excess of $\SS^n$ is suitably small with respect to the regularity threshold from Allard's theorem -- notice that here we do not have to care about multiplicities as we are working with boundaries of finite perimeter sets -- we deduce by continuity that the area excess of $E_h$ on a cylinder of such scale is going to be below Allard's regularity threshold.) This concludes the proof of the theorem.
  \end{proof}

\section{Sharp stability estimates for $C^1$-small normal deformations of $\SS^n$}\label{section normal def of SSn} This section is devoted to the proof of the estimates in Theorem \ref{thm sharp u}, that is say, to the quantitative stability problem for $C^1$-small normal deformations of the sphere. (The sharpness of these estimates is discussed in the next section.) We divide the proof into a series of lemmas, throughout which we shall always consider the following assumptions:
\begin{eqnarray}
    \label{u basic assumption}
    &&\mbox{$\Om$ is an open set with $C^{1,1}$-boundary $\pa\Om=\big\{(1+u(x))\,x:x\in\SS^n\big\}$}
    \\
    \label{H less than n x}
    &&\mbox{$H_\Om\le n$ a.e. on $\pa\Om$}
    \\
    \label{u small C1}
    &&\|u\|_{C^1(\SS^n)}\le\e(n)
    \\
    \label{S zero barycenter}
    &&\int_{\pa\Om}\,x =0\,.
\end{eqnarray}
Notice that
\[
\de(\Om)=P(\Om)-P(B_1)
\]
can be made arbitrarily small thanks to \eqref{u basic assumption} and \eqref{u small C1}, and is non-negative by Almgren's principle.

\begin{lemma}
  \label{thm quant u prelim}
  If \eqref{u basic assumption}, \eqref{H less than n x}, \eqref{u small C1} and \eqref{S zero barycenter} hold, then
  \begin{gather}
    \frac{\de(\Om)}{C(n)}\le\int_{\mathbb{S}^n} u \leq C(n) \, \de(\Om) \, , \label{avg to delta}
    \\
	\|u\|_{W^{1,2}(\SS^n)}^2\leq C(n) \, \left( \de(\Om)^2 + \de(\Om) \, \|u\|_{C^0(\SS^n)} \right) \, , \label{L2 gradient to delta}
  \end{gather}
  \begin{equation}\label{stima hd sharp}
    \|u\|_{C^0(\SS^n)}\le C(n)\,\left\{
    \begin{split}
      &\de(\Om)\,,&\qquad\mbox{if $n=1$}\,,
      \\
      &\de(\Om)\,\log\Big(\frac1{\de(\Om)}\Big)\,,&\qquad\mbox{if $n=2$}
      \\
      &\de(\Om)^{1/(n-1)}\,,&\qquad\mbox{if $n\ge 3$}\,.
    \end{split}
    \right .
  \end{equation}
\end{lemma}

\begin{proof}[Proof of Lemma \ref{thm quant u prelim}] Following the approach of Fuglede \cite{Fuglede}, the proof consists in expanding $u$ into spherical harmonics, and then obtaining the desired estimates by combining the Taylor expansions of $\de(\Om)$ and $\int_{\pa\Om} x$ with the vanishing barycenter condition \eqref{S zero barycenter} and with the inequality obtained by testing the non-negative function $n-H_\Om$ (computed in the coordinates of $\SS^n$) against $u-\inf_{\SS^n} u$. We notice that assumption \eqref{H less than n x} will not be used until step four of the proof.

\medskip

\noindent {\it Step one}: We start by expressing $u$ into spherical harmonics. Constants and coordinate functions provide the first $n+2$  spherical harmonics on $\SS^n$. Correspondingly we have
\begin{equation}
  \label{u a b R}
  u(x) = a + b \cdot x + R(x)\qquad x\in\SS^n\,,
\end{equation}
where
\[
a = \frac1{\H^n(\SS^n)}\,\int_{\mathbb{S}^n} u\qquad b_i=\frac{\int_{\SS^n}x_i\,u}{\int_{\SS^n}x_i^2}\qquad i=1,...,n+1\,,
\]
and $R\in C^1(\SS^n)$ satisfies
\begin{equation}
  \label{R averages}
  \int_{\mathbb{S}^n} R=0\qquad \int_{\mathbb{S}^n} x\,R= 0\,.
\end{equation}
The following remarks will be useful. First, by $\int_{\SS^n}(b\cdot x)^2=|b|^2\H^n(\SS^n)/(n+1)$, we have
\begin{eqnarray}
  \label{useful u2}
  \int_{\SS^n}u^2&=&\int_{\SS^n}a^2+(b\cdot x)^2+R^2=\H^n(\SS^n)\Big(a^2+\frac{|b|^2}{n+1}\Big)+\int_{\SS^n}R^2
  \\
  \label{usefule Du2}
  \int_{\SS^n}|\nabla u|^2&=&\int_{\SS^n}|b-(b\cdot x)x|^2+|\nabla R|^2\,.
\end{eqnarray}
In particular, by \eqref{useful u2}, \eqref{usefule Du2} and
\[
0=\int_{\SS^n}n(b\cdot x)^2-|b-(b\cdot x)x|^2\,,
\]
we find that
\begin{eqnarray}\nonumber
  \int_{\SS^n}n\,u^2-|\nabla u|^2&=&\int_{\SS^n}n\,a^2+n\,(b\cdot x)^2+n\,R^2-|b-(b\cdot x)x|^2-|\nabla R|^2
  \\\label{later}
  &=&\int_{\SS^n}\,n\,a^2+n\,R^2-|\nabla R|^2\,.
\end{eqnarray}
Second, since the $k$-th eigenvalue $\l_k$ of the Laplacian over $\SS^n$ satisfies $\l_k=k(n+k-1)$, and since $R$ is orthogonal to the first two eigenspaces (see \eqref{R averages}), for $R$ we have the Poincar\'e-type inequality
\begin{equation}
  \label{R poincare}
  \int_{\SS^n}|\nabla R|^2\ge 2(n+1)\int_{\SS^n}R^2\,,
\end{equation}
which is stronger than the usual Poincar\'e inequality
\[
\int_{\SS^n}|\nabla v|^2\ge n\int_{\SS^n}v^2\qquad\mbox{$\forall v\in W^{1,2}(\SS^n)$ with $\int_{\SS^n}v=0$}\,.
\]
Finally, by \eqref{u small C1} we have
\begin{equation}
  \label{a b R C1 small}
  |a|\le \e(n)\qquad |b|\le C(n)\,\e(n)\qquad \|R\|_{C^1(\SS^n)}\le C(n)\,\e(n)\,.
\end{equation}
This fact will be particularly useful in expanding the metric of $S$ as seen from $\SS^n$,
\[
G(x) = \left( G_{ij}(x) \right) = \left( (1+u(x))^2 \delta_{ij} + \nabla_{\tau_i} u(x) \nabla_{\tau_j} u(x) \right)\qquad x\in\SS^n\,.
\]
Here $\tau_1,\ldots,\tau_n$ is an orthonormal basis for $T_x \mathbb{S}^n$ at $x \in \mathbb{S}^n$, and using \eqref{u small C1}, we compute
\begin{equation}\label{metric}
  \begin{split}
  	G^{-1} &= \left( G^{ij}(x) \right) = \left( \frac{\delta_{ij}}{(1+u)^{2}}
  -\frac{\nabla_{\tau_i} u \nabla_{\tau_j} u}{(1+u)^{2} ((1+u)^2 + |\nabla u|^2)}  \right)   \\
	\det G &= (1+u)^{2n-2} ((1+u)^2 + |\nabla u|^2), \\
	\sqrt{\det G} &= (1+u)^{n-1} \sqrt{(1+u)^2 + |\nabla u|^2} \, .
  \end{split}
\end{equation}

\medskip

\noindent {\it Step two}: We now exploit the barycenter assumption \eqref{S zero barycenter} to show that
\begin{equation} \label{b controlled by DR}
	|b| \leq C(n) \, \int_{\mathbb{S}^n} \, |\nabla R|^2+\e\,{\rm O}\big(u^2+|\nabla u|^2\big)\,.
\end{equation}
Indeed, by the area formula, \eqref{S zero barycenter} takes the form
\begin{equation}\label{S zero barycenter u}
	0=\int_{\pa\Om} x=\int_{\mathbb{S}^n} \, (1+u) \, x \, \sqrt{\det G} \, .
\end{equation}
By \eqref{metric} and \eqref{u small C1} we have
\[
(1+u) \sqrt{\det G}=1+(n+1)u+(n+1)n\frac{u^2}2+\frac{|\nabla u|^2}2+\e\,{\rm O}\big( u^2+|\nabla u|^2\big)\,,
\]
where ${\rm O}(u^2+|\nabla u|^2)$ denotes a function of $\mathbb{S}^n$ bounded in absolute value by $C(n)(u^2+|\nabla u|^2)$.
By $\int_{\SS^n}x=0$ and \eqref{R averages} we find
\[
\int_{\SS^n}x\,u=a\int_{\SS^n}\,x+\int_{\SS^n}(b\cdot x)\,x+\int_{\SS^n}x\,R=\int_{\SS^n}(b\cdot x)\,x=b\,\int_{\SS^n}x_1^2
\]
and similarly
\begin{eqnarray*}
\int_{\SS^n}x\,u^2&=&a^2\int_{\SS^n}x+2a\int_{\SS^n}x(b\cdot x+R)+\int_{\SS^n}x(b\cdot x+R)^2
\\
&=&2a\,b\int_{\SS^n}x_1^2+\int_{\SS^n}x(b\cdot x+R)^2
\end{eqnarray*}
By combining \eqref{S zero barycenter u} with these identities we thus find
\[
\H^n(\SS^n)(1+n\,a)\,b=-\frac{n(n+1)}2\int_{\SS^n}x(b\cdot x+R)^2-\int_{\SS^n}x\,\frac{|\nabla u|^2}2+\e\,{\rm O}\big( u^2+|\nabla u|^2\big)
\]
so that, by \eqref{a b R C1 small} and by noticing that $|\nabla u|\le |b| +|\nabla R|$,
\begin{eqnarray*}
|b|&\le&C(n)\int_{\SS^n}|b|^2+R^2+|\nabla R|^2
+\e\,{\rm O}\big( u^2+|\nabla u|^2\big)
\\
&\le&C(n)\,\int_{\SS^n}|b|^2+|\nabla R|^2+\e\,{\rm O}\big(u^2+|\nabla u|^2\big)\,,
\end{eqnarray*}
where in the last inequality we have used \eqref{R poincare} (here using the weaker version with $n$ in place of $2(n+1)$ would have be fine as well).

\medskip

\noindent {\it Step three}: Now we compute $H_\Om$ in the coordinates of $\SS^n$, see \eqref{H formula u} below. By assumption $H_\Om\in L^1_{{\rm loc}}(\H^n\llcorner\pa\Om)$ is such that
\begin{equation}
  \label{tangential divergence theorem on S}
\int_{\pa\Om}\,\nu_\Om\cdot X\,H_\Om\,d\H^n=\int_{\pa\Om}\,\Div^{\pa\Om}X\,d\H^n\qquad\forall X\in C^\infty_c(\R^{n+1};\R^{n+1})\,.
\end{equation}
Set
\[
H^*(x):=H_\Om((1+u(x))x)\qquad\forall x\in \SS^n\,.
\]
We test \eqref{tangential divergence theorem on S} with $X\in C^\infty_c(\R^{n+1};\R^{n+1})$ such that, for some fixed $\zeta\in C^\infty_c(\R^{n+1})$ and for every $x\in\SS^n$, one has
\[
	X((1+u(x))x) = \left. \frac{d}{dt} (1+u(x)+t\zeta(x))\,x) \right|_{t=0} = \zeta(x)\,x \,.
\]
By the area formula, \eqref{tangential divergence theorem on S} writes as
\begin{equation}\label{H_eqn1}
	\int_{\mathbb{S}^n} \, H^* \, \zeta \, \nu_\Om \cdot x \, \sqrt{\det G}
	= \int_{\mathbb{S}^n} \, \sqrt{\det G} \, G^{ij} \, \left( \nabla_{\tau_j} u \nabla_{\tau_i} \zeta + (1+u) \, \zeta \,\de_{ij}\right) \,.
\end{equation}
We expand the both sides of \eqref{H_eqn1} by means of \eqref{metric} in order to find that
\begin{eqnarray*}
	&&\int_{\mathbb{S}^n} \,H^* \, \zeta \, (1+u)^n
\\\nonumber
	&=& \int_{\mathbb{S}^n} \, (1+u)^{n-3} \, \sqrt{(1+u)^2 + |\nabla u|^2} \, \left( \delta_{ij} - \frac{\nabla_{\tau_i} u \nabla_{\tau_j} u}{(1+u)^2 + |\nabla u|^2} \right)
		\left( \nabla_{\tau_j} u \nabla_{\tau_i} \zeta + (1+u) \, \zeta\,\de_{ij} \right) \\
	&=& \int_{\mathbb{S}^n} \, (1+u)^{n-3} \, \left( \frac{(1+u)^2 \, \nabla u \cdot \nabla \zeta}{\sqrt{(1+u)^2 + |\nabla u|^2}}
		- \frac{(1+u) \, |\nabla u|^2 \, \zeta}{\sqrt{(1+u)^2 + |\nabla u|^2}} + n \, (1+u) \, \sqrt{(1+u)^2 + |\nabla u|^2} \, \zeta \right) \,.
\end{eqnarray*}
Replacing $\zeta$ with $(1+u)^{-n} \, \zeta$,
\begin{equation*}
	\int_{\mathbb{S}^n} \,H^* \, \zeta
	= \int_{\mathbb{S}^n} \, \left( \frac{\nabla u \cdot \nabla \zeta}{(1+u) \, \sqrt{(1+u)^2 + |\nabla u|^2}}
		+ \frac{n \, \zeta}{\sqrt{(1+u)^2 + |\nabla u|^2}} - \frac{|\nabla u|^2 \, \zeta}{(1+u)^2 \, \sqrt{(1+u)^2 + |\nabla u|^2}} \right) \,.
\end{equation*}
Since $\zeta \in C^\infty_c(\R^{n+1})$ is arbitrary, we can write $H^*$ in divergence form as
\begin{align} \label{H formula u}
	H^*&= -\op{div}_{\mathbb{S}^n} \left( \frac{\nabla u}{(1+u) \, \sqrt{(1+u)^2 + |\nabla u|^2}} \right) + \frac{n-\frac{|\nabla u|^2}{(1+u)^2}}{\sqrt{(1+u)^2 + |\nabla u|^2}}\,,
\end{align}
which is the formula needed in the sequel.

\medskip

\noindent {\it Step four}: We conclude the proof.  By \eqref{u small C1}, \eqref{metric},  \eqref{useful u2} and \eqref{usefule Du2}
\begin{eqnarray*}
\label{delta_eqn3}
	\de(\Om) &=& \H^n(S) - \H^n(\mathbb{S}^n)
    \\
	&=& \int_{\mathbb{S}^n} \, \left( \sqrt{\det G} - 1 \right) = \int_{\mathbb{S}^n} \, \left( (1+u)^{n-1} \, \sqrt{(1+u)^2 + |\nabla u|^2} - 1 \right) \nonumber \\
	&=& \int_{\mathbb{S}^n} \, \left( n \, u + \frac{n\, (n-1)}{2} \, u^2 + \frac{1}{2} \, |\nabla u|^2 \right) + \e\,{\rm O}\big(\|u\|_{W^{1,2}(\SS^n)}^2\big) \nonumber \\
	&=& \H^n(\mathbb{S}^n)\left( n \, a + \frac{n \, (n-1)}{2} \, a^2   + \frac{n(n-1)}{2(n+1)}\, |b|^2\right)
+\frac12\int_{\SS^n}|b-(b\cdot x)x|^2
		\\
&&\nonumber+ \int_{\mathbb{S}^n} \, \left( \frac{n\, (n-1)}{2} \, R^2 + \frac{|\nabla R|^2}{2} \right) + \e\,{\rm O}\big(\|u\|_{W^{1,2}(\SS^n)}^2\big)\, ,
\end{eqnarray*}
which thanks to \eqref{b controlled by DR} gives
\begin{equation}
  \label{yyy}
  \de(\Om)=\H^n(\mathbb{S}^n)\left( n \, a + \frac{n \, (n-1)}{2} \, a^2 \right)
		+ \int_{\mathbb{S}^n} \, \left( \frac{n\, (n-1)}{2} \, R^2 + \frac{|\nabla R|^2}{2} \right) + \e\,{\rm O}\big(\|u\|_{W^{1,2}(\SS^n)}^2\big)
\end{equation}
Now let $\ell = \sup_{\mathbb{S}^n} |u^-|$, where $u^-(x) = \max \{-u(x), \, 0\}$, so that $u + \ell \geq 0$ and $\ell \leq \e$ by \eqref{u small C1}. By \eqref{H less than n x}, \eqref{H formula u} and $\|u\|_{C^1(\SS^n)}\le\e$
\begin{align*}
	0 &\leq \int_{\mathbb{S}^n} \, (n - H^*) \, (u+\ell) \\
	&= \int_{\mathbb{S}^n} \, \left( -\frac{|\nabla u|^2}{(1+u) \, \sqrt{(1+u)^2 + |\nabla u|^2}} + n \, (u+\ell) - \frac{n\, (u+\ell)}{\sqrt{(1+u)^2 + |\nabla u|^2}}
		\right. \nonumber \\& \hspace{15mm} \left. + \frac{(u+\ell) \, |\nabla u|^2}{(1+u)^2 \, \sqrt{(1+u)^2 + |\nabla u|^2}} \right) \\
	&\leq \int_{\mathbb{S}^n} \, (n\,\ell \, u+ n \, u ^2 - |\nabla u|^2) + \e\,{\rm O}\big(\|u\|_{W^{1,2}(\SS^n)}^2\big)\,.
\end{align*}
By combining this inequality with \eqref{later}, we find that
\[\int_{\SS^n}|\nabla R|^2
\le n\,\ell\,a+\int_{\SS^n} na^2+nR^2
+ \e\,{\rm O}\big(\|u\|_{W^{1,2}(\SS^n)}^2\big)\,.
\]
By \eqref{R poincare} (where now it is crucial to have $2(n+1)$ in place of $n$ in the Poincar\'e inequality)
\begin{equation}
  \label{nanni}
  \Big(1-\frac{n}{2(n+1)}\Big)\,\int_{\SS^n}|\nabla R|^2
\le
n\,\ell\,a+\,n\,\H^n(\SS^n)\,a^2
+ \e\,{\rm O}\big(\|u\|_{W^{1,2}(\SS^n)}^2\big)\,.
\end{equation}
By combining \eqref{useful u2}, \eqref{usefule Du2}, \eqref{b controlled by DR} and \eqref{nanni} we find that
\begin{equation}
  \label{new}
  \int_{\SS^n}u^2+|\nabla u|^2\le C(n)\,\big(\ell\,a+a^2\big)\le C(n)\e\,|a|\,.
\end{equation}
Hence \eqref{yyy} gives $\de(\Om)=n\H^n(\SS^n)a+\e\,{\rm O}(|a|)$, that is
\begin{equation} \label{delta_eqn4}
	\frac{\de(\Om)}{C(n)}\leq a\leq C(n) \, \de(\Om) \,.
\end{equation}
This proves \eqref{avg to delta}, and then the first inequality in \eqref{new} implies \eqref{L2 gradient to delta}.

\medskip

\noindent {\it Step five}: We finally prove \eqref{stima hd sharp}. To this end, let us recall the following Poincar\'e-type interpolation inequality from \cite[Lemma 1.4]{Fuglede}: {\it for every $v\in C^1(\SS^n)$ with $\int_{\SS^n}v=0$, one has}
\begin{equation}\label{from fuglede}
	\|v\|_{C^0(\SS^n)}
    \leq\,C(n)\,
    \begin{cases}
    \|\nabla v\|_{L^2(\SS^1)}
    &\textnormal{if $n = 1$}
    \\
    \|\nabla v\|_{L^2(\SS^2)} \,  \log\left(C(2)\,\frac{\|\nabla v\|_{C^0(\SS^2)}}{\|\nabla v\|_{L^2(\SS^2)}}\right)^{\frac12}
    &\textnormal{if  $n=2$}
    \\
    \|\nabla v\|_{C^0(\SS^n)}^{\frac{n-2}n}\,
    \|\nabla v\|_{L^2(\SS^n)}^{\frac2n}
    &\textnormal{if $n>2$}\,.
    \end{cases}
\end{equation}
We deduce \eqref{stima hd sharp} by applying \eqref{from fuglede} to $v=u-a$. For example, in the case $n>2$, by \eqref{avg to delta}, \eqref{L2 gradient to delta} and \eqref{from fuglede} we find
\begin{eqnarray*}
  \|u\|_{C^0(\SS^n)}\le a+  \|v\|_{C^0(\SS^n)}\le C(n)\,\de(\Om)+C(n)\,\Big(\de(\Om)^2+\|u\|_{C^0(\SS^n)}\de(\Om)\Big)^{1/n}\,.
\end{eqnarray*}
Assuming without loss of generality that $\de(\Om)\le\|u\|_{C^0(\SS^n)}/M(n)$ for a suitably large constant $M(n)$, we deduce that
\[
  \|u\|_{C^0(\SS^n)}\le C(n)\,\Big(\de(\Om)^2+\|u\|_{C^0(\SS^n)}\de(\Om)\Big)^{1/n}
  \le C(n)\,\|u\|_{C^0(\SS^n)}^{1/n}\de(\Om)^{1/n}
\]
and thus $\|u\|_{C^0(\SS^n)}\le\de(\Om)^{1/(n-1)}$, as desired. The cases $n=2$ and $n=1$ follow by analogous arguments.
This completes the proof of Lemma \ref{thm quant u prelim}.
\end{proof}

Taking into account Lemma \ref{thm quant u prelim}, in order to complete the proof of Theorem \ref{thm sharp u} we are left to obtain linear bounds on $\|u\|_{L^1(\SS^n)}$ and $\|u^+\|_{C^0(\SS^n)}$.

\begin{lemma}
  \label{thm quant u prelim 2}
  If $u$ and $\Om$ satisfy \eqref{u basic assumption}, \eqref{u small C1} and \eqref{S zero barycenter}, then for every $q>n/2$
  \begin{equation}
    \label{stima after}
    \|u\|_{C^0(\SS^n)}\le C(n,q)\,\Big(\|u\|_{L^2(\SS^n)}+\|H_\Om-n\|_{L^q(\pa\Om)}\Big)\,,
  \end{equation}
  and whenever $1\le p<n/(n-1)$,
  \begin{equation}
    \label{stimaLp simon}
  \|\nabla u\|_{L^p(\SS^n)}\le C(n,p)\,\big(\|u\|_{L^p(\SS^n)}+\|H_\Om-n\|_{L^1(\pa\Om)}\big)\,.
  \end{equation}
  In addition: (i) if $H_\Om\le n$ $\H^n$-a.e. on $\pa\Om$, then
  \begin{eqnarray}
  \label{moser}
  \|u^+\|_{C^0(\SS^n)}\le C(n)\,\|u\|_{L^1(\SS^n)}\,;
  \end{eqnarray}
  (ii) if $p\in(1,\infty)$ and there exists $K>0$ such that
  \begin{equation}
    \label{existence of K st}
      \|\nabla u\|_{L^p(\SS^n)}\le K\,\|u\|_{L^p(\SS^n)},
  \end{equation}
  then, provided $\|u\|_{C^1(\SS^n)}\le\e(n,p,K)$, one has
  \begin{equation}
    \label{consequence of K}
    \|u\|_{L^p(\SS^n)} \leq C(n,p,K) \, \Big( \Big| \int_{\mathbb{S}^n}u\,x \Big| + \|H_\Om-n\|_{L^1(\pa\Om)}\Big) \, .
  \end{equation}
  (iii) if $\a\in(0,1)$ and there exists $K>0$ such that $\|\nabla u\|_{C^{0,\a}(\SS^n)}\le K$, then
  \begin{equation}
    \label{stima after holder}
    \|u\|_{C^{1,\a}(\SS^n)}\le C(n,K,\a)\,\big(\|u\|_{C^0(\SS^n)}+\|H_\Om-n\|_{L^\infty(\pa\Om)}\big)\,.
  \end{equation}
  (iv) finally, if $\Lambda\ge0$, $1\le p<n/(n-1)$, and
  \begin{eqnarray}
    \label{H less than n and Lambda}
  &&-\Lambda\le H_\Om\le n\qquad\mbox{$\H^n$-a.e. on $\pa\Om$}
  \\
  \label{u small C1 Lambda 2}
  && \|u\|_{C^1(\SS^n)} \le\e(n,p,\Lambda)\,,
  \end{eqnarray}
  then
  \begin{gather}\label{stima W1p u sharp}
    \|u\|_{W^{1,p}(\SS^n)}\le C(n,p,\Lambda)\,\de(\Om)\,,
  \\
    \label{stima u+}
    \|u^+\|_{C^0(\SS^n)}\le C(n,\Lambda)\,\de(\Om)\,.
  \end{gather}
\end{lemma}


The proof is based on combining Almgren's identity \eqref{almgren identity} with two estimates from elliptic regularity theory, which are proved in Lemma \ref{elliptic lemma 2} and Lemma \ref{lemma elliptic sphere} below.

\begin{lemma} \label{elliptic lemma 2}
If $n \geq 2$, $p\in(1,n/(n-1))$, $\rho>0$, $f\in L^p(B_\rho;\R^n)$, $g\in L^1(B_\rho)$, and $v\in C^1(B_\rho)$ is a weak solution of
\begin{equation} \label{elliptic2_eqn1}
	\Delta v=\Div(f)+g\qquad\mbox{in $B_\rho$}\,,
\end{equation}
then
\begin{equation}
  \label{elliptic step one rho}
  \|\nabla v\|_{L^p(B_{\rho/2})}\le C(n,p)\Big(\rho^{-1}\,\|v\|_{L^p(B_\rho)}+\|f\|_{L^p(B_\rho)}+\rho^{1+\frac{n}p-n}\,\|g\|_{L^1(B_\rho)}\Big)\,.
\end{equation}
\end{lemma}

\begin{proof}[Proof of Lemma \ref{elliptic lemma 2}] The argument is based on the use of standard elliptic estimates, and it is detailed just for the sake of clarity. (In particular, we could not find an exact reference for the case considered in here, where we need to use the $L^1$-norm of $g$; see \eqref{H to delta} below.) By scaling we can set $\rho=1$, and then prove \eqref{elliptic step one rho} in three steps.

\medskip

\noindent {\it Step one}: We assume that $f\in C^\infty_c(B_1;\R^n)$ and $g \in C_c^{\infty}(B_1)$. Denoting by $\Gamma$  the fundamental solution of the Laplacian on $\R^n$, let us set $v^i=\Gamma\star (D_if^i)$ and $w=\Gamma\star g$, where $\star$ denotes convolution. In this way $\Delta v^i = D_if^i$ and $\Delta w = g$ on $\R^n$ (in pointwise sense), and defining $\vphi$ by the identity
\[
v =\varphi +\sum_{i=1}^n v^i + w
\]
we have that $\varphi$ is harmonic on $B_1$, and thus such that
\begin{equation}
  \label{vphi regularity}
  \|\nabla \varphi\|_{L^p(B_{1/2})} \leq C(n,p) \, \|\varphi\|_{L^p(B_1)}\,.
\end{equation}
By the Calderon-Zygmund theory, $\Delta v^i = D_if^i$ implies
\begin{equation}
  \label{vi reg}
  \|\nabla v^i\|_{L^p(B_{1/2})} \leq C(n,p) \, \Big(\|v^i\|_{L^p(B_1)}+\|f^i\|_{L^p(B_1)}\Big)\, .
\end{equation}
Since $|\nabla\Gamma(z)|=c(n)\,|z|^{1-n}$ for every $n\ge 2$, thanks to $p<n/(n-1)$ we have
\begin{align}\nonumber
	\int_{B_{1/2}} |Dw(x)|^p \, dx &\leq \int_{B_{1/2}} \left( \int_{\mathbb{R}^n} |x-y|^{1-n} \, |g(y)| \, dy \right)^p dx
\\\nonumber
	&\leq \int_{B_{1/2}} \left( \int_{\mathbb{R}^n} |x-y|^{p(1-n)} \, |g(y)| \, dy \right) \left( \int_{\mathbb{R}^n} |g(y)| \, dy \right)^{p-1} dx
\\\nonumber
	&= \|g\|_{L^1(B_1)}^{p-1} \, \int_{\mathbb{R}^n} |g(y)|\,dy\int_{B_{1/2}} |x-y|^{p(1-n)}\, dx
\\\label{w reg}
	&\leq C(n,p) \, \|g\|_{L^1(B_1)}^p\,.
\end{align}
By combining \eqref{vphi regularity}, \eqref{vi reg} and \eqref{w reg} we obtain \eqref{elliptic step one rho} when $f\in C^\infty_c(B_1;\R^n)$ and $g \in C_c^{\infty}(B_1)$.

\medskip

\noindent {\it Step two}: Now we assume that $f \in L^p(B_1;\R^n)$, $g \in L^1(B_1)$ with $f=0$ and $g=0$ on $\R^n\setminus B_{3/4}$. Let us fix an even function $\rho\in C^\infty_c(B_1)$ with $0\le \rho\le 1$ and $\int_{\R^n}\rho=1$, set $\rho_\de(z)=\de^{-n}\rho(z/\de)$ for $z\in\R^n$, and define $f_\de=f\star\rho_\de$, $g_\de=g\star\rho_\de$ and $v_\de=\bar{v}\star\rho_\de$, where $\bar v$ is the extension to zero of $v$ outside of $B_1$. If $\de<1/4$, then $f_\de\in C^\infty_c(B_1;\R^n)$, $g_\de\in C^\infty_c(B_1)$ and $v_\de$ is a weak solution in $B_{3/4}$ of $\Delta v_\de=\Div(f_\de)+g_\de$. By case one
\begin{equation}
  \label{elliptic step one case two}
  \|\nabla v_\de\|_{L^p(B_{1/2})}\le C(n,p)\,\Big(\|v_\de\|_{L^p(B_{3/4})}+\|f_\de\|_{L^p(B_{3/4})}+\|g_\de\|_{L^1(B_{3/4})}\Big)\,.
\end{equation}
Since $v_\de \to v$ in $W^{1,p}(B_{3/4})$, $f_\de\to f$ in $L^p(B_{3/4})$ and $g_\de\to g$ in $L^1(B_{3/4})$, letting $\de\to 0^+$ in \eqref{elliptic step one case two} we deduce \eqref{elliptic step one rho} in this case too.

\medskip

\noindent {\it Step three}: We finally prove \eqref{elliptic step one rho} in full generality. Let $\eta\in C^\infty_c(B_{3/4})$ with $0\le \eta\le 1$ and $\eta=1$ on $B_{1/2}$. If $w=\eta\,v$, then, in distributional sense,
\begin{eqnarray*}
\Delta w&=&\eta\,\big(\Div\,f+g\big)+2\nabla\eta\cdot\nabla v+v\,\Delta\eta
\\
&=&\Div(\eta f+2v\nabla\eta)
-f\cdot\nabla\eta-v\Delta\eta+\eta\,g
\\
&=&\Div(\bar f)+\bar g  \qquad\mbox{on $B_1$}
\end{eqnarray*}
provided $\bar f=f\,\eta+2v\nabla\eta\in L^p(B_1)$ and $\bar g=-f\cdot\nabla\eta-v\,\Delta\eta+\eta g$. Since $\bar f$ and $\bar g$ vanish outside $B_{3/4}$, by step two we can apply \eqref{elliptic step one rho} to $w$ and exploit $\nabla v=\nabla w$ on $B_{1/2}$ together with
\begin{eqnarray*}
\|w\|_{L^p(B_1)}+\|\bar f\|_{L^p(B_1)}+\|\bar g\|_{L^1(B_1)}\le C(n,p)\Big(\|v\|_{L^p(B_1)}+\|f\|_{L^p(B_1)}+\|g\|_{L^1(B_1)}\Big)\,,
\end{eqnarray*}
to complete the proof of \eqref{elliptic step one rho} in the general case.
\end{proof}

\begin{lemma}\label{lemma elliptic sphere}
For every $n\ge 2$, $K>0$ and $p\in(1,\infty)$ there exist positive constants $C(n,p,K)$ and $\e(n,p,K)$ with the following property. If
$G\in L^\infty(\SS^n)$, $\a,,\g\in C^0(\mathbb{S}^n)$ and $\b\in C^0(\mathbb{S}^n, T\mathbb{S}^n)$ are such that
\begin{eqnarray}
  \label{corollary elliptic assumption1}
	\max\big\{\|\a-1\|_{C^0(\SS^n)},\|\b\|_{C^0(\SS^n)},\|\g - n\|_{C^0(\SS^n)}\big\}\le\e(n,p,K)
\end{eqnarray}
and if $u\in C^1(\mathbb{S}^n)$ is a weak solution to
\begin{equation}\label{corollary elliptic assumption2}
	\Div_{\SS^n}(\a\,\nabla u)+\beta\cdot\nabla u+\g\,u=G\qquad\mbox{on $\SS^n$}
\end{equation}
such that
\begin{equation}
  \label{corollary bound strutturale}
  \|\nabla u\|_{L^p(\SS^n)}\le K\,\|u\|_{L^p(\SS^n)},
\end{equation}
then
\begin{equation} \label{elliptic3_eqn1}
	\|u\|_{L^p(\SS^n)} \leq C(n,p,K) \, \Big( \Big| \int_{\mathbb{S}^n}u\,x \Big| + \|G\|_{L^1(\SS^n)}\Big) \, .
\end{equation}
\end{lemma}

\begin{proof}[Proof of Lemma \ref{lemma elliptic sphere}] We argue by contradiction and assume the existence, for every $k\in\N$, of $G_k\in L^\infty(\SS^n)$, $\a_k\,\g_k\in C^0(\SS^n)$, $\b_k\in C^0(\SS^n;T\SS^n)$, and $u_k\in C^1(\SS^n)$ such that
\begin{eqnarray}\label{limit 1}
  &&\lim_{k\to\infty}\max\Big\{\|\a_k-1\|_{C^0(\SS^n)},\|\b_k\|_{C^0(\SS^n)},\|\g_k - n\|_{C^0(\SS^n)}
  \Big\}=0
  \\\nonumber
  &&\Div_{\SS^n}(\a_k\,\nabla u_k)+\beta_k\cdot\nabla u_k+\g\,u_k=G_k\qquad\mbox{weakly on $\SS^n$}
\end{eqnarray}
with $\|\nabla u_k\|_{L^{p}(\SS^n)}\le K\,\|u_k\|_{L^p(\SS^n)}$ for every $k\in\N$, and
\begin{equation}
  \label{recall}
  \frac{\|u_k\|_{L^p(\SS^n)}}k \ge \Big| \int_{\mathbb{S}^n}  u_k\,x \Big| + \|G_k\|_{L^1(\SS^n)}\,.
\end{equation}
If we set $\bar u_k=\|u_k\|_{L^p(\SS^n)}^{-1}u_k$ and $\bar G_k=\|u_k\|_{L^p(\SS^n)}G_k$, then $\|\bar G_k\|_{L^1(\SS^n)}\to 0$ and $|\int_{\SS^n} \bar u_k\,x|\to 0$ as $k\to\infty$ with
\begin{eqnarray}\label{limit 2}
  \Div_{\SS^n}(\a_k\,\nabla \bar u_k)+\beta_k\cdot\nabla\bar u_k+\g\,\bar u_k=\bar G_k\qquad\mbox{weakly on $\SS^n$}
\end{eqnarray}
and $\|\nabla \bar u_k\|_{L^p(\SS^n)}\le K$ for every $k\in\N$. Since $p\in(1,\infty)$, there exists $u\in W^{1,p}(\SS^n)$ such that $\bar u_k\to u$ in $L^p(\SS^n)$ (so that $\|u\|_{L^p(\SS^n)}=1$ and $\int_{\SS^n}  u\,x=0$) and $\nabla \bar u_k\rightharpoonup\nabla u$ in $L^p(\SS^n)$. By \eqref{limit 1} and since $\|\bar G_k\|_{L^1(\SS^n)}\to 0$ as $k\to\infty$, we deduce from \eqref{limit 2} that
\[
\Delta u+n\,u=0\qquad\mbox{on $\SS^n$}\,.
\]
Since $\|u\|_{L^p(\SS^n)}=1$, $u$ is an eigenvector of the Laplacian on $\SS^n$ corresponding to the eigenvalue $\l_1=n$. In particular, $u=c(n,p)\,x\cdot e$ for some unit vector $e$. This contradicts $\int_{\SS^n} u\,x=0$ and thus completes the proof of the lemma.
\end{proof}

\begin{proof}[Proof of Lemma \ref{thm quant u prelim 2}] Fix $e\in\SS^n$, and set
\begin{eqnarray}\label{K for a moment}
\K_r&=&\big\{x\in\R^{n+1}:|x-(x\cdot e)e|<r\,,x\cdot e>0\big\}\,,
\\\nonumber
\D_r&=&\{z\in e^\perp:|z|<r\}
\end{eqnarray}
so that $\K_r=\D_r\times \R_+\,e$. If $w_0(z)=\sqrt{1-|z|^2}$, then
\[
\SS^n\cap\K_1=\big\{x\in\K_1:x\cdot e=w_0(x-(x\cdot e)\,e)\big\}\,,
\]
and
\begin{equation}
  \label{w0 mc eqn}
  -\Div\Big(\frac{\nabla w_0}{\sqrt{1+|\nabla w_0|^2}}\Big)= n\quad\mbox{on $\D_1$}\,.
\end{equation}
Thanks to $\|u\|_{C^1(\SS^n)}\le \e(n)$, we can find $w\in C^{1,1}(\D_{1/2})$ with $\Lip(w)\le C(n)$ such that
\[
\pa\Om\cap\K_{1/2}=\big\{x\in\K_{1/2}:x\cdot e=w(x-(x\cdot e)\,e)\big\}\,.
\]
Let us define $h\in L^\infty(\D_{1/2})$ by setting, for a.e. $z\in\D_{1/2}$,
\[
h(z)=H_\Om(z+w(z)\,e)=-\Div\Big(\frac{\nabla w(z)}{\sqrt{1+|\nabla w(z)|^2}}\Big)\,.
\]
Setting for $z\in\D_{1/2}$ and $\xi\in\R^n$
\begin{equation}\label{def v F M}
v(z)=w(z)-w_0(z)\,,\qquad F(\xi)=\frac{\xi}{\sqrt{1+|\xi|^2}}\,,\qquad M(z)=\int_0^1\,\nabla F(\nabla w_0(z)+t\,\nabla v(z))\,dt
\end{equation}
we find that $F(\nabla w)-F(\nabla w_0)=M\,\nabla v$ and thus
\begin{equation}
  \label{coming back}
  h-n=-\Div(F(\nabla w)-F(\nabla w_0))=-\Div(M\nabla v)=-\Delta v-\Div((M-\Id)\nabla v)
\end{equation}
holds on $\D_{1/2}$. We now argue as follows:

\medskip

\noindent {\it Proof of \eqref{stima after}}: By the De Giorgi-Nash-Moser theorem (see, e.g. \cite[Theorem 8.17]{gt}), since $\Div(M\,\nabla v)=n-h$ on $\D_{1/2}$, we find
\[
\|v\|_{C^0(\D_{1/4})}\le C(n,q)\,\Big(\|v\|_{L^2(\D_{1/2})}+\|n-h\|_{L^q(\D_{1/2})}\Big)\qquad\forall q>\frac{n}2\,,
\]
which immediately implies \eqref{stima after} thanks to a covering argument.

\medskip

\noindent {\it Proof of \eqref{stimaLp simon}}: If we set
\[
g=n-h\qquad f=(\Id-M)\nabla v\,,
\]
then $g\in L^\infty(\D_{1/2})$ and $f\in C^0(\D_{1/2};\R^n)$ with
\[
\|g\|_{L^1(\D_{1/2})}\le C(n)\,\|H-n\|_{L^1(\pa\Om)}\qquad \|f\|_{L^p(\D_\rho)}\le C(n)\,(\e+\rho)\,\|\nabla v\|_{L^p(\D_\rho)}\,.
\]
for every $\rho\in(0,1/2)$, where we have used $\|u\|_{C^1(\SS^n)}\le\e$ and $\nabla w_0(0)=0$ to deduce
\[
\|M-\Id\|_{C^0(\D_\rho)}\le C\,\big(\|\nabla w_0\|_{C^0(\D_\rho)}+\|\nabla v\|_{C^0(\D_\rho)}\big)\le C(n)\,\big(\e+\rho\big)\,.
\]
Since $v$ solves $\Delta v=\Div(f)+g$ in $\D_{1/2}$, by Lemma \ref{elliptic lemma 2} we find that, for every $\rho\in(0,1/2)$,
\begin{eqnarray*}
\|\nabla v\|_{L^p(\D_{\rho/2})}&\le& C(n,p)\Big(\rho^{-1}\,\|v\|_{L^p(D_\rho)}+\|f\|_{L^p(\D_\rho)}+\rho^{1+\frac{n}p-n}\,\|g\|_{L^1(\D_\rho)}\Big)
  \\
  &\le& C(n,p)\Big(\rho^{-1}\,\|v\|_{L^p(D_\rho)}+(\e+\rho)\|\nabla v\|_{L^p(\D_\rho)}+\rho^{1+\frac{n}p-n}\,\|H-n\|_{L^1(\pa\Om)} \Big)\,.
\end{eqnarray*}
If we denote by $G_r(e)$ the geodesic ball on $\SS^n$ of radius $r>0$ and center $e\in\SS^n$, then this last estimate implies, in terms of $u$, that, for every $\rho\in(0,1/4)$
\begin{equation}
  \label{simon}
  \|\nabla u\|_{L^p(G_\rho(e))}\le C_0(n,p)\,
\Big(\rho^{-1}\,\|u\|_{L^p(G_{4\rho}(e))}+(\e+\rho)\|\nabla u\|_{L^p(G_{4\rho}(e))}+\rho^{1+\frac{n}p-n}\,\|H-n\|_{L^1(\pa \Om)}\Big)\,.
\end{equation}
Let us set, for $\rho_0=\rho_0(n,p)\in(0,1/4)$ to be determined,
\[
Q=\sup\big\{\|\nabla u\|_{L^p(G_\rho(e))}:\rho\in(\rho_0/2,\rho_0)\,,e\in\SS^n\big\}\,.
\]
Clearly there exists $N(n)$ such that for every $e\in\SS^n$ and $\rho<1/4$ one can find $\{e_k\}_{k=1}^{N(n)}\subset\SS^n$ such that
\[
G_\rho(e)\subset\bigcup_{k=1}^{N(n)}\,G_{\rho/4}(e_k)\,.
\]
In this way by \eqref{simon} and definition of $Q$ we find that if $\rho\in(\rho_0/2,\rho_0)$, then
\begin{eqnarray*}
  \|\nabla u\|_{L^p(G_\rho(e))}^p&\le&\sum_{k=1}^N\,\|\nabla u\|_{L^p(G_{\rho/4}(e_k))}^p
  \\
  &\le&C_0(n,p)^p
  \sum_{k=1}^{N(n)}\,\Big(\frac{\|u\|_{L^p(\SS^n)}}{\rho}+(\e+\rho_0)\|\nabla u\|_{L^p(G_{\rho}(e_k))}
  +\rho^{1+\frac{n}p-n}\,\|H-n\|_{L^1(\pa \Om)}\Big)^p
  \\
  &\le&C_0(n,p)^p\,N(n)\,\Big(\frac{2\,\|u\|_{L^p(\SS^n)}}{\rho_0}+(\e+\rho_0)\,Q
  +\rho_0^{1+\frac{n}p-n}\,\|H-n\|_{L^1(\pa \Om)}\Big)^p\,,
\end{eqnarray*}
that is
\[
Q \le
C_0(n,p)\,N(n)^{1/p}\,\Big(\frac{2\,\|u\|_{L^p(\SS^n)}}{\rho_0}+(\e+\rho_0)\,Q+\rho_0^{1+\frac{n}p-n}\,\|H-n\|_{L^1(\pa \Om)}\Big)\,.
\]
Provided $\e$ and $\rho_0$ are small enough in terms of $n$ and $p$, we conclude that
\[
\|\nabla u\|_{L^p(G_\rho(e))}\le C(n,p)\big(\|u\|_{L^p(\SS^n)}+\|H-n\|_{L^1(\pa \Om)}\big)\,,\qquad\forall \rho\in(\rho_0/2,\rho_0)\,,e\in\SS^n\,,
\]
so that, by a covering argument, we obtain \eqref{stimaLp simon}.

\medskip

\noindent {\it Proof of \eqref{moser}}: Recall that for proving \eqref{moser} we are assuming $H_\Om\le n$ $\H^n$-a.e. on $\pa\Om$, so that, by definition of $h$, we have $h(z)\le n$ a.e. on $\D_{1/2}$. Coming back to \eqref{coming back} we thus see that $v=w-w_0$ solves
\[
\Div(M\,\nabla v)=n-h\ge0\qquad\mbox{on $\D_{1/2}$}
\]
that is, $v$ is a subsolution to a quasilinear elliptic equation on $\D_{1/2}$. By Moser's iteration technique we find that
\[
\|v^+\|_{C^0(\D_{1/4})}\le C(n)\,\|v\|_{L^1(\D_{1/2})}\,,
\]
and thanks to the arbitrariness of $e$, we conclude the proof of \eqref{moser}.

\medskip

\noindent {\it Proof of \eqref{stima after holder}}: Since we are assuming that $\|\nabla u\|_{C^{0,\a}(\SS^n)}\le K$, we have $\|\nabla v\|_{C^{0,\a}(\D_{1/2})}\le C(n,K)$, and thus looking back at the definition \eqref{def v F M} of $M$, that
\[
\|M\|_{C^{0,\a}(\D_{1/2})}\le C(n,K)\,.
\]
We can thus apply \cite[Theorem 8.32]{gt} to find that
\[
\|v\|_{C^{1,\a}(\D_{1/4})}\le C(n,K,\a)\Big(\|v\|_{C^0(\D_{1/2})}+\|g\|_{L^\infty(\D_{1/2})}\Big)\,,
\]
and then deduce \eqref{stima after holder} by a covering argument.

\medskip

\noindent {\it Proof of \eqref{consequence of K}}: Let us recall \eqref{H formula u}, namely
\begin{align}
	H^*= -\op{div}_{\mathbb{S}^n} \left( \frac{\nabla u}{(1+u) \, \sqrt{(1+u)^2 + |\nabla u|^2}} \right) + \frac{n-\frac{|\nabla u|^2}{(1+u)^2} }{\sqrt{(1+u)^2 + |\nabla u|^2}} \,.
		\label{Hdue}
\end{align}
If we set
\[
\a= \frac{1}{(1+u) \, \sqrt{(1+u)^2 + |\nabla u|^2}}\qquad \g=\frac{n}u\,\Big(1-\frac1{1+u}\Big)\qquad G=H^*-n\,,
\]
and define $\b:\SS^n\to T\SS^n$ so that
\[
\frac{|\nabla u|^2}{(1+u)^2 \, \sqrt{(1+u)^2 + |\nabla u|^2}}+n\Big(\frac1{1+u}-\frac1{\sqrt{(1+u)^2+|\nabla u|^2}}\Big)=\beta\cdot\nabla u\,,
\]
then \eqref{Hdue} takes the form
\begin{equation}
  \label{on SSn 1}
  \Div_{\SS^n}(\a\,\nabla u)+\beta\cdot\nabla u+\g\,u=G\qquad\mbox{on $\SS^n$}
\end{equation}
where $\a\in C^0(\SS^n)$, $\b\in C^0(\SS^n;T\SS^n)$, $\g\in C^0(\SS^n)$ and $G\in L^\infty(\SS^n)$ are such that
\begin{eqnarray}
\label{alfa beta and gamma small}
  \max\big\{\|\a-1\|_{C^0(\SS^n)},\|\b\|_{C^0(\SS^n)},\|\g-n\|_{C^0(\SS^n)}\big\}\le C(n)\,\|u\|_{C^1(\SS^n)}
  \\\label{G small}
  \|G\|_{L^1(\SS^n)}\le C(n)\,\|H-n\|_{L^1(\pa \Om)} \qquad \|G\|_{L^\infty(\SS^n)}=\|H-n\|_{L^\infty(\pa\Om)}\,.
\end{eqnarray}
Thus, given $K>0$ such that \eqref{existence of K st} holds, the validity of \eqref{consequence of K} follows by assuming $\|u\|_{C^1(\SS^n)}\le\e(n,p,K)$ and thanks to Lemma \ref{lemma elliptic sphere}.

\medskip

\noindent {\it Conclusion of the proof}: We finally assume the validity of \eqref{H less than n and Lambda} and \eqref{u small C1 Lambda 2} and prove \eqref{stima W1p u sharp} and \eqref{stima u+}. We first notice that by \eqref{almgren identity}, denoting by $A$ the convex envelope of $\Om$, we have
\[
\de(\Om)\ge\H^n(\pa\Om\setminus \pa A)+\int_{\pa A\cap \pa\Om}\bigg(1-\Big(\frac{H_\Om}{n}\Big)^n\bigg)\,.
\]
Since $0\le H_\Om\le n$ on $\pa A\cap\pa\Om$, thanks to \eqref{H less than n and Lambda} we find
\begin{equation}
\|H_\Om-n\|_{L^1(\pa\Om)}\leq  C(n) \Big(1+\|(H_\Om)^-\|_{L^\infty(\pa\Om)}\Big)\, \de(\Om)\le C(n,\Lambda)\,\de(\Om) \, . \label{H to delta}
\end{equation}
Next we claim that
\begin{gather}\label{stima Lp u sharp}
    \|u\|_{L^p(\SS^n)}\le C(n,p,\Lambda)\,\de(\Om)\,.
\end{gather}
To show this let us assume without loss of generality that
\begin{equation}
  \label{wlog}
  \de(\Om)\le \|u\|_{L^p(\SS^n)}\,,
\end{equation}
so that \eqref{stimaLp simon} and \eqref{H to delta} imply in particular
\begin{equation}
  \label{stima0}
  \|\nabla u\|_{L^p(\SS^n)}\le  C_*(n,p,\Lambda)\,\|u\|_{L^p(\SS^n)}\,.
\end{equation}
Thanks to \eqref{u small C1 Lambda 2}, \eqref{consequence of K} holds with $K=C_*(n,p,\Lambda)$, and gives us, taking \eqref{H to delta} into account,
\begin{equation}
\label{stima1}
\|u\|_{L^p(\SS^n)}\le C(n,p,\Lambda)\,\Big(\Big|\int_{\SS^n} u\,x\Big|+\de(\Om)\Big)
\end{equation}
By \eqref{L2 gradient to delta}, \eqref{b controlled by DR}, and $\|u\|_{C^1(\SS^n)}\le\e$
\[
\Big|\int_{\SS^n}u\,x\Big| \le C(n)\,\|u\|_{W^{1,2}(\SS^n)}^2 \le C(n)\,\Big(\de(\Om)^2+\|u\|_{C^0(\SS^n)}\de(\Om)\Big)\le C(n)\,\e\,\de(\Om)
\]
which combined with \eqref{stima1} gives us \eqref{stima Lp u sharp}. By combining \eqref{stima Lp u sharp} with \eqref{stimaLp simon} and \eqref{H to delta} we find \eqref{stima W1p u sharp}. By combining \eqref{moser} and \eqref{stima Lp u sharp} we find \eqref{stima u+}. The proof is complete.
\end{proof}

We now combine the Lemma \ref{thm quant u prelim}, Lemma \ref{thm quant u prelim 2} and Proposition \ref{prop E construction} to prove the estimates in the statement of Theorem \ref{thm sharp u}. Their sharpness, which is also part of Theorem \ref{thm sharp u}, is addressed in the next section.

\begin{proof}
  [Proof of Theorem \ref{thm sharp u}, estimates] Let $\Om$ be such that \eqref{u basic assumption}, \eqref{H less than n x}, \eqref{u small C1} and \eqref{S zero barycenter} hold, and such that $\pa\Om$ is smooth. By Lemma \ref{thm quant u prelim} we find that \eqref{sharp barycenter}, \eqref{sharp w12} and \eqref{sharp c0} hold. We are thus left to prove
  \[
  \max\big\{\|u\|_{L^1(\SS^n)},\|(u)^+\|_{C^0(\SS^n)}\big\}\le C(n)\,\de(\Om)\,.
  \]
  Since $\|u\|_{C^1(\SS^n)}\le\e(n)$ implies $\de(\Om)\le\de_0(n)$ for $\de_0(n)$ as in Theorem \ref{thm structure}-(ii), we can apply Theorem \ref{thm structure}-(ii) (at this point we need $\Om$ to be better than $C^{1,1}$-regular, compare with Proposition \ref{prop E construction}-(iii)) to find an open set $E$ with $C^{1,1}$-boundary in $\R^{n+1}$ such that
  \begin{eqnarray}\nonumber
    &&\Om\subset E\,,\qquad\diam(\Om)=\diam(E)\,,
    \\\label{bb0}
    &&|E\setminus\Om|+\H^n\big(\pa E\setminus\pa\Om\big)\le C(n)\,\de(\Om)
    \\\nonumber
    &&\|H_E\|_{L^\infty(\pa E)}\le n\,.
  \end{eqnarray}
  By \eqref{sharp w12} we have
  \[
  |\Om\Delta B_1|\le C(n)\,\|u\|_{L^1(\SS^n)}\le C(n)\,\|u\|_{L^2(\SS^n)}\le C(n)\,\sqrt{\de(\Om)}
  \]
  so that by $\Om\subset E$
  \[
  |E\Delta B_1|\le|E\setminus\Om|+|\Om\Delta B_1|\le C(n)\,\sqrt{\de(\Om)}\,,
  \]
  We can thus argue as in the proof of Theorem \ref{thm structure} and apply Allard's theorem to deduce that, if $\e(n)$ (and thus $\de(\Om)$) is small enough, then for some $v\in C^{1,1}(\SS^n)$ we have
  \begin{equation}
    \label{bb1}
      \pa E=\big\{(1+v(x))\,x:x\in\SS^n\big\}\qquad \|v\|_{C^1(\SS^n)}\le C(n)\,\e(n)\,.
  \end{equation}
  (Notice that conclusion (ii) in Theorem \ref{thm structure} is analogous to \eqref{bb1} but holds only after a translation. Here we do not need to translate neither $E$ or $\Om$, as we know by assumption that $\Om$ is close to $B_1$.)
  
  We would now like to apply Lemma \ref{thm quant u prelim 2} to $E$, but the barycenter $x_E$ of $\pa E$, defined by
  \[
  x_E=\frac1{P(E)}\int_{\pa E}x
  \]
  may be non-zero. We notice however that
  \begin{equation}
    \label{xE small}
    |x_E|\le C(n)\,\de(\Om)\,.
  \end{equation}
  Indeed, by \eqref{bb1},
  \begin{eqnarray}\nonumber
  \de(E)&=&P(E)-P(B_1)\le\H^n(\pa E\cap\pa\Om)+C(n)\,\de(\Om)-P(B_1)
  \\\nonumber
  &\le&P(\Om)+C(n)\,\de(\Om)-P(B_1)
  \\\label{bb2}
  &\le&C(n)\,\de(\Om)\,,
  \end{eqnarray}
  so that $P(B_1)/2\le P(E)\le 2\,P(B_1)$ and we can directly focus on the size of $\int_{\pa E}x$. To this end, we first notice that
  by the assumption $\int_\Om x=0$, we have
  \[
  \int_{\pa E}x= \int_{\pa E\cap\pa\Om}x+\int_{\pa E\setminus\pa\Om}x=\int_{\pa E\setminus\pa\Om}x-\int_{\pa\Om\setminus\pa E}x\,.
  \] 
  Now, by \eqref{bb0} and \eqref{bb1}
  \[
  \Big|\int_{\pa E\setminus\pa\Om}x\Big|\le \|v\|_{C^0(\SS^n)}\,\H^n(\pa E\setminus\pa\Om)\le C(n)\,\de(\Om)\,,
  \]
  while
  \[
  \Big|\int_{\pa\Om\setminus\pa E}x\Big|\le \|u\|_{C^0(\SS^n)}\,\H^n(\pa\Om\setminus\pa E)
  \]
  where $\Om\subset E\subset A$ (with $A$ the convex envelope of $\Om$) implies
  \[
  \pa\Om\setminus\pa E=(\pa \Om)\cap E\subset(\pa\Om)\cap A=\pa\Om\setminus\pa A
  \]
  and where Almgren's identity \eqref{almgren identity} gives $\H^n(\pa \Om\setminus\pa A)\le\de(\Om)$. Putting everything together, we deduce \eqref{xE small}.
  
  We can thus apply Lemma \ref{thm quant u prelim 2} to $E-x_E$, and find
  \[
  \|v\|_{W^{1,1}(\SS^n)}\le C(n)\,\big(\de(E)+|x_E|\big)\,,\qquad \|v^+\|_{C^0(\SS^n)}\le C(n)\,\big(\de(E)+|x_E|\big)\,.
  \]
  which combined with \eqref{xE small} and \eqref{bb2} gives
  \begin{eqnarray*}
    \|u\|_{L^1(\SS^n)}&\le&C(n)\,|\Om\Delta B_1|\le C(n)\,\big(|E\setminus\Om|+|E\Delta B_1|\big)
    \\
    &\le&C(n)\,\big(\de(\Om)+\|v\|_{L^1(\SS^n)}\big)\le C(n)\,\de(\Om)\,,
  \end{eqnarray*}
  that is \eqref{sharp L1}. Similarly, since $\Om\subset E$ we have that
  \begin{eqnarray*}
  \|u^+\|_{C^0(\SS^n)}=\sup\{|x|-1:x\in\Om\}\le\sup\{|x|-1:x\in E\}=\|v^+\|_{C^0(\SS^n)}
  \le C(n)\,\de(\Om)\,,
  \end{eqnarray*}
  that is \eqref{sharp c0+}. This completes the proof of the estimates in Theorem \ref{thm sharp u}.
  \end{proof}

  \begin{proof}
    [Proof of Theorem \ref{thm main}] Let $\Om$ be a bounded open set with smooth boundary in $\R^{n+1}$ such that $H_\Om\le n$ and $\de(\Om)\le\de(n)$.  By Theorem \ref{thm structure} there exists an open bounded set $E$ with boundary of class $C^{1,1}$ such that
    $\Om\subset E$, $\diam(\Om)=\diam(E)$, $\|H_E\|_{L^\infty(\pa E)}\le n$, $|E\setminus\Om|\le C(n)\de(\Om)$, $\H^n\big(\pa E\setminus\pa\Om\big)\le C(n)\,\de(\Om)$, and $\pa E=\{(1+u(x))x:x\in\SS^n\}$ where $u\in C^1(\SS^n)$ is such that $\|u\|_{C^1(\SS^n)}\le\e(n)$ for $\e(n)$ as in Theorem \ref{thm sharp u}. In particular, $\|u\|_{L^1(\SS^n)}\le C(n)\de(E)$ and $\|u^+\|_{C^0(\SS^n)}\le C(n)\de(E)$. We conclude by arguing as in the last part of the previous proof.
  \end{proof}

\section{Sharpness of Theorem \ref{thm sharp u}}\label{section example}

The goal of this section is proving the sharpness of Theorem \ref{thm sharp u}. Given that the sharpness of \eqref{sharp L1}, \eqref{sharp c0+}, \eqref{sharp c0} (limited to the case $n=1$) and \eqref{sharp w12} is easily checked by considering the set $\Om=B_{1+t}$ as $t\to 0^+$, we focus on proving the sharpness of \eqref{sharp c0} when $n\ge 2$, namely
   \[
   \|u\|_{C^0(\SS^n)}\le C(n)\,\left\{
    \begin{split}
      &\de(\Om)\,\log\left(\frac{C(2)}{\de(\Om)}\right)&\qquad\mbox{if $n=2$}
      \\
      &\de(\Om)^{1/(n-1)}&\qquad\mbox{if $n>2$}\,.
    \end{split}
    \right .
    \]
    We are going to do this by constructing a family of open sets with $C^{1,1}$-boundary $\{\Om_t\}_{t \in (0,t_0)}$, such that $\pa\Om_t=\{(1+u_t(x))\,x:x\in\SS^n\}$ for $u_t\in C^{1,1}(\SS^n)$ such that
    \begin{equation}
      \label{est}
      C(n)\,   \|u_t\|_{C^0(\SS^n)}\ge \,\left\{
    \begin{split}
      &\de(\Om_t)\,\log\left(\frac{1}{\de(\Om_t)}\right)&\qquad\mbox{if $n=2$}\,,
      \\
      &\de(\Om_t)^{1/(n-1)}&\qquad\mbox{if $n>2$}\,.
    \end{split}
    \right.
    \end{equation}
    For the sake of simplicity we shall just write $\Om$ and $u$ in place of $\Om_t$ and $u_t$.

    We construct $\pa\Om$ as a surface of revolution obtained by modifying $\SS^n$ in the positive cylinder above a small $n$-dimensional disk. More precisely, we decompose $\R^{n+1}=\R^n\times\R$, denote by $\D_r$ the ball of radius $r>0$ centered at the origin in $\R^n$, and set $\SS^{n-1}=\pa\D_1$.  We introduce some parameters $t$, $r_0$ and $r_1$ satisfying
    \begin{eqnarray}\label{r0 r1}
	0 < t< \frac1{K(n)}\qquad 0 < r_1 < r_0 < \frac{1}{K(n)}
    \end{eqnarray}
    for a suitably large dimensional constant $K(n)$. Later on $r_1$ will be specified as a function of $n$, $r_0,$ and $t$.  We let
\begin{equation} \label{vphi0}
	\vphi_0(r)=\sqrt{1-r^2}\quad r\in[0,1)\,.
\end{equation}
so that $\{ (r\om,\vphi_0(r)) : r \in [0,1), \, \omega \in \mathbb{S}^{n-1} \}$ is the unit upper half sphere in $\R^{n+1}$.  We define $\vphi:[0,1)\to\R$ by setting
\begin{equation}\label{ny1}
	\vphi(r) =\left\{
    \begin{split}
      &\vphi_0(r)\,,&\qquad r\in[r_0,1)\,,
      \\
      &\vphi_0(r)-t\,h(r)&\qquad r\in[r_1,r_0]\,,
      \\
      &\vphi_0(r_1)-t\,h(r_1)+\frac{r_1}\mu-\sqrt{\Big(1+\frac1{\mu^2}\Big)r_1^2-r^2}&\qquad r\in[0,r_1)\,,
    \end{split}
    \right .
\end{equation}
where we let $h \in C^2([r_1,r_0])$ be a function such that
\begin{equation} \label{h at r0}
	h(r_0) = h'(r_0) = 0
\end{equation}
and we define $\mu$ in terms of $h$ and $r_1$ by setting
\begin{equation} \label{mu}
	\mu = \vphi_0'(r_1) - t \, h'(r_1)\,.
\end{equation}
Notice that if $h\in C^2([r_1,r_0])$, then \eqref{h at r0} and \eqref{mu} guarantee that $\vphi\in C^{1,1}([0,1))$ with $\vphi'(0)=0$. We further specify that
\begin{equation} \label{h decreasing convex}
	0\le h(r)\le 1 \qquad h'(r)\le 0\qquad h''(r)\ge 0\qquad\forall r\in[r_1,r_0]\,,
\end{equation}
so that defining $S$ by
\begin{align}\label{ny2}
	S \cap \big(\D_{r_0} \times (0,\infty)\big) &= \{ (r \omega, \vphi(r)) : r \in [0,r_0], \, \omega \in \mathbb{S}^{n-1} \} \, ,
	\\\nonumber
	S \setminus\big(\D_{r_0} \times (0,\infty)\big) &= \partial B_1 \setminus \big(\D_{r_0} \times (0,\infty)\big)\,,
\end{align}
we find that $S=\pa\Om^*$ for an open set $\Om^*$ with $C^{1,1}$-boundary as depicted in
\begin{figure}
	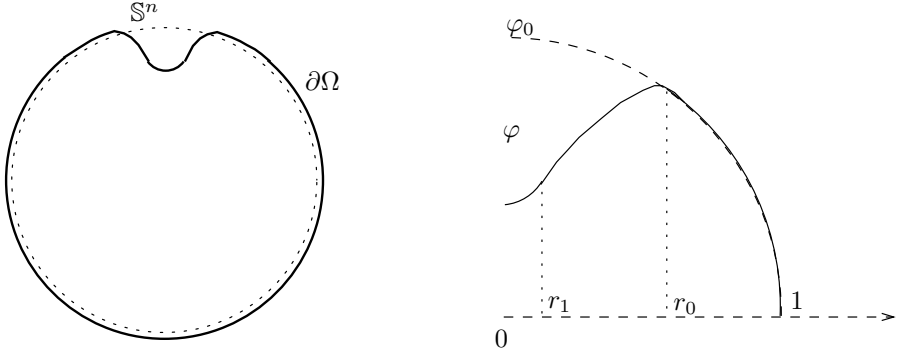\caption{{\small The function $\vphi$ is obtained by carefully joining two circular arcs of opposite curvature. The domain $\Om$ is obtained by slightly scaling out the resulting surface of revolution.}}\label{fig s}
\end{figure}
Figure \ref{fig s}.  Observe that by the definition of $\vphi$, $S \cap (\D_{r_1} \times (0,\infty))$ is a spherical cap.  By a classical computation,
   %
%
the mean curvature of $\pa\Om^*$ (as usual computed with respect to $\nu_{\Om^*}$) at the point $r\om+\vphi(r)\,e_{n+1}$ corresponding to $\om\in\SS^{n-1}$ and $r\in(0,1)$ is given by
\begin{equation}\label{H omega star}
	H(r)=H_{\Om^*}(r\om+\vphi(r)\,e_{n+1}) = \frac{-\vphi''(r)}{(1+\vphi'(r)^2)^{3/2}} - \frac{(n-1) \, \vphi'(r)}{r \sqrt{1+\vphi'(r)^2}} \,.
\end{equation}
Of course, since $\vphi=\vphi_0$ on $(r_0,1)$, we have $H(r)=n$ for $r\in(r_0,1)$. Since $\Om^*\subset B_1$, the boundary $\pa\Om^*$ is more curved than $\SS^n$ near $r=r_0$, and thus we have $H(r)>n$ for $r$ sufficiently close to, and less than, $r_0$.

Setting $\Om=(1+t)\Om^*$, we claim that for a suitable choice of $h$, we can achieve
\begin{equation} \label{Omega is admissible}
	H_\Om\le n\qquad\mbox{$\H^n$-a.e. on $\pa\Om$}\,.
\end{equation}
Since $(1+t)\,H_\Om=H_{\Om^*}$, we need
\begin{equation} \label{Omega is admissible 2}
	H(r) \leq n + n \, t \qquad\forall r\in(r_1,r_0)\,.
\end{equation}
By combining \eqref{H omega star} and \eqref{Omega is admissible 2}, we see that finding $h$ amounts to solving the differential inequality
\begin{equation} \label{diff ineq 1}
	H(r) = \frac{-\vphi''_0(r) + t\,h''(r)}{(1+(\vphi'_0(r) + t\,h'(r))^2)^{3/2}} + \frac{(n-1) \, (-\vphi'_0(r) + t\,h'(r))}{r \sqrt{1+(\vphi'_0(r) + t\,h'(r))^2}} \leq n+t\,n\,.
\end{equation}
We will find a solution $h$ which roughly behaves like the fundamental solution for the Laplacian, i.e. like $\log(1/r)$ when $n = 2$ and $r^{2-n}$ when $n > 2$.  The precise choice of $h$ is found by considering the Taylor's expansion of \eqref{diff ineq 1}. It is convenient to impose some structural conditions on $h$ in order to control the higher order terms in such expansion. Recalling that $h'(r)\le 0$ by \eqref{h decreasing convex}, we will require that
\begin{equation} \label{h' small}
	t\,|h'(r)| \leq t\,|h'(r_1)| \leq \frac{1}{K(n)}\,,
\end{equation}
and since we expect $h$ to behave like the fundamental solution of the Laplacian, we will also require that
\begin{equation} \label{h fund soln}
	\max\{|h'(r)|,\,r\,|h''(r)|\} \leq K(n) \, \frac{r_0^n}{r^{n-1}}\,,
\end{equation}
where recall that $K(n)$ is a large positive constant to be determined. Notice that by \eqref{mu} and \eqref{h' small} we definitely have
\begin{equation}
  \label{mu small}
  |\mu|\le \frac{3}{K(n)}\,.
\end{equation}

Now, let us rewrite the expression of $H(r)$ in \eqref{diff ineq 1} as
\[
H(r)=\Big((1-r^2)^{-3/2} + t \, h''\Big)\,g(r,t)^{-3/2}+(n-1) \, \Big((1-r^2)^{-1/2} + t \, r^{-1} \, h'\Big)\,g(r,t)^{-1/2}
\]
where
\begin{equation*}
	g(r,t) = (1-r^2)^{-1} + 2 \, t \, r \, (1-r^2)^{-1/2} \, h'(r) + t^2 \, h'(r)^2 \,,
\end{equation*}
and observe that by \eqref{r0 r1} and \eqref{h' small}
\begin{equation} \label{g near one}
	|g(r,t) - 1| \leq \frac{r^2}{1-r^2} + \frac{2\,r}{K\,(1-r^2)^{1/2}} + \frac{1}{K^2}=\Big(\frac{r}{\sqrt{1-r^2}}+\frac1{K}\Big)^2 \leq \frac{5}{K^2}
\end{equation}
for all $r \in (r_1,r_0)$ and $t \in (0,1)$.  Applying Taylor's theorem
\begin{equation*}
	f(t) = f(0) + f'(0) \, t + \int_0^t (t-s) \, f''(s) \, ds
\end{equation*}
to $f(t) = g(r,t)^{-k/2}$ for $k = 1,3$ and using \eqref{h' small} and \eqref{g near one},
\begin{align*}
	&\left| g(r,t)^{-k/2} - (1-r^2)^{k/2} + k \, t \, r \, (1-r^2)^{(k+1)/2} \, h'(r) \right| \nonumber \\
	&= \left| \int_0^t (t-s) \, h'(r)^2 \, \left( -k \, g(r,s)^{-k/2-1} + k(k+2) \, g(r,s)^{-k/2-2} \, (r \, (1-r^2)^{-1/2} + s \, h'(r))^2 \right) \, ds \right| \nonumber \\
	&\leq \frac{k+1}{2} \, t^2 \, h'(r)^2
\end{align*}
where in the last inequality one choose $K(n)$ large enough to make $r_0$ and $\mu$ (recall \eqref{mu small}) sufficiently small.   Hence,
\begin{eqnarray*}
	H(r) &\leq& \left( (1-r^2)^{-3/2} + t \, h'' \right) \left( (1-r^2)^{3/2} - 3 \, t \, r \, (1-r^2)^2 \, h' + 2 \, t^2 \, (h')^2 \right) \\&&
		+ (n-1) \, (1-r^2)^{-1/2} \left( (1-r^2)^{1/2} - t \, r \, (1-r^2) \, h' + t^2 \, (h')^2 \right) \\ &&
		+ (n-1) \, t \, r^{-1} \,h' \left( (1-r^2)^{1/2} - t \, r \, (1-r^2) \, h' - t^2 \, (h')^2 \right) \\
	&\leq& n + t \, (1-r^2)^{3/2} \, h'' + t \, h' \left( -3 \, r \, \sqrt{1-r^2} + (n-1) r^{-1} (1-r^2)^{3/2} \right) \\&&
		+ t^2 \, \left( -3\,r\,(1-r^2)^2\,h'\,h'' + 2\,(1-r^2)^{-3/2}\,(h')^2 + (n-1)\,(1-r^2)^{-1/2}\,(h')^2 \right) \\&&
		+ t^3 \, \left( 2\,(h')^2\,h'' - (n-1)\,r^{-1}\,(h')^3 \right) \,.
\end{eqnarray*}
Using \eqref{r0 r1} and \eqref{h fund soln},
\begin{align*}
	H(r) &\leq n + t \, (1-r^2)^{3/2} \, h'' + t \, h' \left( -3 \, r \, \sqrt{1-r^2} + (n-1) r^{-1} (1-r^2)^{3/2} \right) \\ &\hspace{15mm}
		+ (n+5) \, t^2 \, K^2 \, r_0^{2n} \, r^{2-2n} + (n+1) \, t^3 \, K^3 \, r_0^{3n} \, r^{2-3n}\,.
\end{align*}
Therefore we can guarantee \eqref{Omega is admissible 2} if
\begin{align} \label{example ode}
	&(1-r^2)^{3/2} \, h'' + h' \left( -3 \, r \, \sqrt{1-r^2} + (n-1) r^{-1} (1-r^2)^{3/2} \right) \nonumber \\&\hspace{15mm}
	+ (n+5) \, t \, K^2 \, \frac{r_0^{2n}}{r^{2n-2}} + (n+1) \, t^2 \, K^3 \, \frac{r_0^{3n}}{r^{3n-2}} = n\,.
\end{align}
We will treat the last two terms in \eqref{example ode} separately since $r^{2-3n}$ increases faster than $r^{2-2n}$ as $r \downarrow 0$ and thus, as will become apparent below, we need to use the full factor $t^2$ to control the last term of \eqref{example ode}.  Multiplying both sides by $r^{n-1}$ we get
\begin{equation*}
	 \frac{d}{dr} \left( r^{n-1} \, (1-r^2)^{3/2} \, h'(r) \right) = n\,r^{n-1} - (n+5) \, t \, K^2 \, r_0^{2n} \, r^{1-n} - (n+1) \, t^2 \, K^3 \, r_0^{3n} \, r^{1-2n} \,.
\end{equation*}
Integrating over $(r,r_0)$ and taking \eqref{h at r0} into account we find that, when $n > 2$
\begin{equation*}
	-r^{n-1} \, (1-r^2)^{3/2} \, h'(r)
	= r_0^n - r^n - \frac{n+5}{n-2} \, t \, K^2 \,\Big(\frac{r_0^{2n}}{r^{n-2}} - r_0^{n+2}\Big) - \frac{n+1}{2n-2} \, t^2 \, K^3 \,
\Big(\frac{r_0^{3n}}{r^{2n-2}} - r_0^{n+2}\Big) \,,
\end{equation*}
that is
\begin{eqnarray} \label{h deriv}
	h'(r) &=& - (1-r^2)^{3/2} \, \left( \frac{r_0^n}{r^{n-1}} - r - \frac{n+5}{n-2} \, t \, K^2 \,\Big(\frac{r_0^{2n}}{r^{2n-3}} -
\frac{r_0^{n+2}}{ r^{n-1}}\Big) \right. \nonumber \\&&\left.
	\hspace{5cm}- \frac{n+1}{2n-2} \, t^2 \, K^3 \, \Big(\frac{r_0^{3n}}{r^{3n-3}} - \frac{r_0^{n+2}}{r^{n-1}}\Big) \right)
\end{eqnarray}
If instead $n = 2$, then
\begin{eqnarray} \label{h deriv 2}
	h'(r) = - (1-r^2)^{3/2} \, \left( \frac{r_0^2}{r}- r - 7 \, t \, K^2 \, \frac{r_0^4}r \, \log\left(\frac{r_0}{r}\right)- \frac{3}{2} \, t^2 \, K^3 \, \Big(\frac{r_0^6}{r^{3}} - \frac{r_0^4}r\Big) \right)\,.
\end{eqnarray}
Integrating again over $(r,r_0)$, and using again \eqref{h at r0}, we find that, if $n>2$,
\begin{eqnarray} \label{h stated}
	h(r) &=& \int_r^{r_0} (1-s^2)^{3/2} \, \Big( \frac{r_0^n}{s^{n-1}} - s - \frac{n+5}{n-2} \, t \, K^2 \,\Big(\frac{r_0^{2n}}{s^{2n-3}} - \frac{r_0^{n+2}}{s^{n-1}}\Big)  \nonumber \\&&
	\hspace{5cm}- \frac{n+1}{2n-2} \, t^2 \, K^3 \, \Big(\frac{r_0^{3n}}{s^{3n-3}} - \frac{r_0^{n+2}}{s^{n-1}}\Big) \Big) \, ds\,,
\end{eqnarray}
while if $n=2$, then
\begin{eqnarray} \label{h stated 2}
	h(r) &=& \int_r^{r_0} (1-s^2)^{3/2} \, \left(\frac{r_0^2}{s}-s- 7 \, t \, K^2 \, \frac{r_0^4}{s}\, \log\left(\frac{r_0}{s}\right)- \frac{3}{2} \, t^2 \, K^3 \, \Big(\frac{r_0^6}{s^{3}} - \frac{r_0}{s}\Big) \right) \, ds\,.
\end{eqnarray}

Having obtained these formulas, we now used them to define $h$ and then check that in this way we obtain the desired family of sets.

More precisely, we argue as follows. For a large dimensional constant $K(n)$, we pick positive parameters $t$ and $r_0$ so that
\begin{equation}
  \label{new t r0}
     \frac{t}{r_0} <\frac1{K^2}\qquad r_0<\frac1{K}\,.
\end{equation}
(In particular, $t<1/K^3$.)  Next, pick any $\s$ such that
\begin{equation}
  \label{sigma lb}
   \frac{t}{r_0} <\s<\frac{1}{K^2}\,,
\end{equation}
and define $r_1$ by
\begin{equation} \label{r1 stated}
	r_1=\left(\frac{t}{\sigma}\right)^{1/(n-1)} \, r_0^{n/(n-1)}\qquad\mbox{so that}\qquad t=\frac{r_1^{n-1}}{r_0^n}\,\s\,.
\end{equation}
This choice of $r_1$ is motivated by the fact that we will need $t\,|h'(r_1)|\approx t\,r_0^nr_1^{1-n}<1/K$ (recall that $h'(r)\approx r_0^n\,r^{1-n}$). Notice that  $r_1$, $r_0$ and $t$ satisfy \eqref{r0 r1} thanks to \eqref{new t r0} and \eqref{r1 stated}.

Next we define $h\in C^2([r_1,r_0])$ by means of \eqref{h stated} if $n>2$ and of \eqref{h stated 2} if $n=2$. We claim that \eqref{h at r0}, \eqref{h decreasing convex}, \eqref{h' small} and \eqref{h fund soln}. Once this is checked, thanks to the above computations and setting $\Om=(1+t)\Om^*$ with $\Om^*$ defined thanks to \eqref{ny1} and \eqref{ny2}, we will be able to deduce that $\Om$ is an open set with $C^{1,1}$-boundary satisfying $H_\Om\le n$.

Let us thus check that $h$ satisfies \eqref{h at r0}, \eqref{h decreasing convex}, \eqref{h' small} and \eqref{h fund soln}. The validity of \eqref{h at r0} is immediately checked from the definition of $h$, while the other assertions will follow by showing that
\begin{eqnarray}
  \label{ny3}
  &&t\,|h'(r_1)|\le\frac1{K}\,,
  \\
  \label{ny4}
  &&-\frac{r_0^n}{r^{n-1}}\le h'(r)<-\Big(1-\frac{r}{r_0}\Big)\frac{r_0^n}{2\,r^{n-1}}\,,\qquad\forall r\in(r_1,r_0)\,,
  \\
  \label{ny5}
  &&0<h''(r)\le K\,\frac{r_0^n}{r^n}\,,\qquad\forall r\in(r_1,r_0)\,.
\end{eqnarray}
For proving \eqref{ny3} we just set $r = r_1$ into \eqref{h deriv} and then, thanks to \eqref{sigma lb} and \eqref{r1 stated}, we find that, if $n>2$,
\begin{eqnarray*}
	t \, |h'(r_1)| &\leq& t \, \frac{r_0^n}{r_1^{n-1}} +r_1\,t+ \frac{n+5}{n-2} \, t^2 \, K^2 \, \frac{r_0^{2n}}{r_1^{2n-3}} + \frac{n+1}{2n-2} \, t^3 \, K^3 \, \frac{r_0^{3n}}{r_1^{3-3n}} \\
	&=& \sigma +\sigma\frac{r_1^n}{r_0^n}+ \frac{n+5}{n-2} \, K^2 \, \sigma^2 \, r_1 + \frac{n+1}{2n-2} \, K^3 \, \sigma^3 \\
	&\leq& 3 \, \sigma < \frac{1}{K}\,.
\end{eqnarray*}
A similar computation in the case $n=2$ gives $t \, |h'(r_1)| \leq 3 \, \sigma \leq 1/K$ if $n = 2$. This proves \eqref{ny3}. The lower bound in \eqref{ny4} follows trivially by \eqref{h deriv} and \eqref{h deriv 2},
\begin{equation*}
	h'(r) \geq - (1-r^2)^{3/2} \, \frac{r_0^n}{r^{n-1}} \geq -\frac{r_0^n}{r^{n-1}} \,.
\end{equation*}
Concerning the fact that $h'(r)< 0$, we notice that by exploiting
\begin{equation*}
	1 \leq \frac{1 - (r/r_0)^k}{1-r/r_0} = \sum_{j=0}^{k-1} \left(\frac{r}{r_0}\right)^j \leq k\qquad\forall k\in\N\,,r\in(0,r_0)\,,
\end{equation*}
we find that, if $n>2$ and thanks to \eqref{h deriv},
\begin{eqnarray*}
	h'(r) &=& - (1-r^2)^{3/2} \, \bigg( \frac{r_0^n}{r^{n-1}} \, \Big( 1 - \frac{r^n}{r_0^n}\Big)
		- \frac{n+5}{n-2} \, t \, K^2 \, \frac{r_0^{2n}}{r^{2n-3}} \, \Big( 1 - \frac{r^{n-2}}{r_0^{n-2}}\Big)  \nonumber \\&&
		\hspace{6cm}- \frac{n+1}{2n-2} \, t^2 \, K^3 \, \frac{r_0^{3n}}{r^{3n-3}} \, \Big( 1 - \frac{r^{2n-2}}{r_0^{2n-2}}\Big) \bigg) \\
	&\leq& - (1-r^2)^{3/2} \, \frac{r_0^n}{r^{n-1}} \, \Big( 1 - \frac{r}{r_0} \Big) \Big( 1 - (n+5) \, t \, K^2 \, \frac{r_0^n}{r^{n-2}}
		- (n+1) \, t^2 \, K^3 \, \frac{r_0^{2n}}{r^{2n-2}} \Big)\,.
\end{eqnarray*}
Hence by \eqref{sigma lb} and \eqref{r1 stated}, and provided $K$ is large enough, we find
\begin{eqnarray*}
	h'(r)&\leq& - (1-r_0^2)^{3/2} \, \frac{r_0^n}{r^{n-1}} \, \Big( 1 - \frac{r}{r_0} \Big) \left(1 - (n+5)\, K^2 \, \sigma \, r_1
		- (n+1) \, K^3 \, \sigma^2 \right)
\\
&\le& -\Big( 1 - \frac{r}{r_0} \Big)\,\frac{r_0^n}{2\,r^{n-1}}
\end{eqnarray*}
where we have used $K^2r_1\,\s\le K^2 r_0\,\s\le 1/K$. Similarly, by the concavity of the logarithm ($\log(s)\le s-1$ for every $s>0$) and by definition of $r_1$, we find that if $r\in(r_1,r_0)$, then
\[
t\,r_0^2\,\frac{\log(r_0/r)}{1-(r/r_0)}\le \frac{t\,r_0^3}r\le t\,\frac{r_0^3}{r_1}=\s\,r_0\le\frac1{K^3}
\]
while
\[
\frac{3}{2} \, t^2 \, K^3 \, \frac{r_0^4}{r^{2}} \, \Big(1 -\frac{r^2}{r_0^2} \Big)\le
\frac{3}{2} \, t^2 \, K^3 \, \frac{r_0^4}{r_1^{2}}=\frac32\,K^3\,\s^2\,r_0^2\le\frac{3}{2K}\,,
\]
so that, by \eqref{h deriv 2},
\begin{eqnarray*}
	h'(r) &=&
    - (1-r^2)^{3/2} \, \frac{r_0^2}{r} \,
   \Big( 1 -\frac{r^2}{r_0^2} - 7 \, t \, K^2 \,r_0^2 \log\Big(\frac{r_0}{r}\Big)
	- \frac{3}{2} \, t^2 \, K^3 \, \frac{r_0^4}{r^{2}} \, \Big(1 -\frac{r^2}{r_0^2} \Big) \Big)
\\&=&
    - (1-r^2)^{3/2} \, \frac{r_0^2}{r} \,\Big(1-\frac{r}{r_0}\Big)
   \Big( 1 +\frac{r}{r_0} - 7 \, t \, K^2 \,r_0^2 \frac{\log(r/r_0)}{1-(r/r_0)}
	- \frac{3}{2} \, t^2 \, K^3 \, \frac{r_0^4}{r^{2}} \, \Big(1+\frac{r}{r_0} \Big) \Big)
    \\
	&\leq&
    - (1-r_0^2)^{3/2} \, \frac{r_0^2}{r} \,\Big(1-\frac{r}{r_0}\Big)
    \Big( 1   - \frac{7}{K}- \frac{3}{K} \Big)\le-\Big(1-\frac{r}{r_0}\Big)\frac{r_0^2}{2r}\,
\end{eqnarray*}
provided $K$ is large enough. This completes the proof of \eqref{ny4}. We now prove \eqref{ny5}.
We first check the upper bound: by dropping the positive terms on the left-hand side of \eqref{example ode} (and using also $h'<0$ to this end), we find that, since $h'(r)\ge -r_0^n/r^{n-1}$ and $r_0\le 1/K$,
\begin{equation*}
	h''(r) \leq \frac{n}{(1-r^2)^{3/2}}+ (n-1)\frac{|h'(r)|}r\leq \frac{n}{(1-K^{-2})^{3/2}}+ (n-1)\frac{r_0^n}{r^n}\le
  2n \, \frac{r_0^n}{r^{n}}\le K\,\frac{r_0^n}{r^{n}}\,,
\end{equation*}
provided $K$ is large enough. Concerning the lower bound in \eqref{ny5}, we exploit also the upper bound in \eqref{ny4} to find
\begin{eqnarray*}
	(1-r^2)^{3/2} \, h''&=& n-(n-1)\,(1-r^2)^{3/2} \,r\,h'
\\
&&+3\, r \, \sqrt{1-r^2}\,h'
		- (n+5) \, t \, K^2 \,\frac{r_0^{2n}}{r^{2n-2}} - (n+1) \, t^2 \, K^3 \, \frac{r_0^{3n}}{r^{3n-2}} \\
&\ge&
n+(n-1)\,(1-r^2)^{3/2} \,r\,\Big(1-\frac{r}{r_0}\Big)\,\frac{r_0^n}{2\,r^{n-1}}
\\
&&-3\,\sqrt{1-r^2}\,\frac{r_0^n}{r^{n-2}}
		- (n+5) \, t \, K^2 \,\frac{r_0^{2n}}{r^{2n-2}} - (n+1) \, t^2 \, K^3 \, \frac{r_0^{3n}}{r^{3n-2}}
\\
&\ge&
n+\frac{r_0^n}{r^n}\bigg(\frac{n-1}4\,\Big(1-\frac{r}{r_0}\Big)\,-3\,r^2\,
- (n+5) \, t \, K^2 \,\frac{r_0^{n}}{r^{n-2}} - (n+1) \, t^2 \, K^3 \, \frac{r_0^{2n}}{r^{2n-2}}\bigg)\,.
\end{eqnarray*}
Proceeding from this last inequality, we notice that if $r\in(\tau\,r_0,r_0)$, $\tau=1/2$, then by $t<1/K^3$ we find
\begin{eqnarray*}
(1-r^2)^{3/2}\,h''(r)&\ge& n-\frac{r_0^n}{r^n}
\bigg(3\,r^2+(n+5) \, t \, K^2 \,\frac{r_0^{n}}{r^{n-2}}+(n+1) \, t^2 \, K^3 \, \frac{r_0^{2n}}{r^{2n-2}}\bigg)
\\
&\ge& n-\frac{r^2}{\tau^n}
\bigg(3+\frac{n+5}{K\,\tau^{n-2}}+\frac{n+1}{K^3\,\tau^{2n-2}}\bigg)
\ge n-\frac{1}{K^2\,\tau^n}
\bigg(3+\frac{n+5}{K\,\tau^{n-2}}+\frac{n+1}{K^3\,\tau^{2n-2}}\bigg)\,.
\end{eqnarray*}
By choosing $K$ large enough with respect to $n$, we find $h''(r)>0$ for every $r\in(\tau\,r_0,r_0)$. We now pick $r\in(r_1,\tau\,r_0)$, and in this case we argue that, thanks to \eqref{r1 stated},
\begin{eqnarray*}
	(1-r^2)^{3/2} \, h''&\ge&
n+\frac{r_0^n}{r^n}\bigg(\frac{n-1}4\,(1-\tau)-3\,r_0^2\,
		- (n+5) \, t \, K^2 \,\frac{r_0^{n}}{r_1^{n-2}} - (n+1) \, t^2 \, K^3 \, \frac{r_0^{2n}}{r_1^{2n-2}}\bigg)
\\
&=&
n+\frac{r_0^n}{r^n}\bigg(\frac{n-1}4\,(1-\tau)-\frac{3}{K^2}
		- (n+5) \, K^2 \,\s\,r_1- (n+1) \, \s^2 \, K^3\bigg)\ge n\,,
\end{eqnarray*}
provided $K$ is large enough with respect to $n$. We have thus showed that $h''\ge0$, thus completing the proof of \eqref{ny5}.

So far we have proved that if $K$ is a sufficiently large positive dimensional constant, and we use \eqref{new t r0}, \eqref{sigma lb}, \eqref{r1 stated}, \eqref{h stated}, \eqref{h stated 2}, \eqref{ny1}, \eqref{mu} and \eqref{ny2} to choose $r_0$, $t$, $\s$ and $h$, and to correspondingly define  $\Om$, then $\Om$ is an open set with $C^{1,1}$-boundary such that $H_\Om\le n$. In this construction $t$ is ranging over the interval $(0,\s\,r_0)$, see \eqref{sigma lb}. We now check \eqref{est}.

First note that by \eqref{mu}, \eqref{h deriv}, \eqref{h deriv 2}, and \eqref{r1 stated}, $\mu$ satisfies
\begin{equation} \label{mu expansion}
	\mu = \vphi_0'(r_1) - t \, h'(r_1)
	= \sigma - \frac{n+1}{2n-2} \, K^3 \, \sigma^3 + O(t^{1/(n-1)}) \,.
\end{equation}
Using $-r_0^n \, r^{1-n} \leq h'(r) < 0$ and \eqref{mu expansion}, we compute that, if $n>2$,

First we notice that
\begin{eqnarray*}
	P(\Om^*)-P(B_1)&=&\H^{n-1}(\SS^{n-1}) \int_0^{r_0} \left( \sqrt{1 + \vphi'(r)^2} - \sqrt{1 + \vphi'_0(r)^2} \right) \, r^{n-1} \, dr
	\\
    &\le& C(n) \int_0^{r_0}     \big|\vphi'(r)^2 - \vphi'_0(r)^2 \big| \, r^{n-1} \, dr
\end{eqnarray*}
where we have multiplied and divided by $\sqrt{1 + (\vphi')^2}+\sqrt{1 + (\vphi'_0)^2}\ge 2$.
Now, by \eqref{mu small} and \eqref{r1 stated}
\begin{eqnarray}\nonumber
\int_0^{r_1} \big| (\vphi')^2 - (\vphi'_0)^2 \big| \, r^{n-1} \, dr&\le&
\int_0^{r_1} \Big( \frac{1}{(1+\mu^{-2}) r_1^2 - r^2} + \frac{1}{1-r^2} \Big) \, r^{n+1} \, dr
\\\nonumber
&\le&\Big(\frac{\mu^2}{r_1^2}+2\Big)\,\int_0^{r_1}r^{n+1}\,dr
\le C(n)(\mu^2\,r_1^n+r_1^{n+2})
\\\label{ny7}
&\le&C(n)\,r_1^n
\le C(n,\s)\,t^{n/(n-1)}\,,
\end{eqnarray}
while the fact that $|h'(r)|\le r_0^n/r^{n-1}$ for $r\in(r_1,r_0)$ (recall \eqref{ny3}) gives
\begin{eqnarray*}
\int_{r_1}^{r_0} \big| (\vphi')^2 - (\vphi'_0)^2 \big| \, r^{n-1} \, dr
    &\le&\int_{r_1}^{r_0} \Big(\frac{2 \, t \, r \,|h'(r)|}{\sqrt{1-r^2}} + t^2 \, h'(r)^2 \Big) \, r^{n-1} \, dr
    \\
    &\le& C(n)\,\int_{r_1}^{r_0} t\,r_0^n\,r+t^2\,\frac{r_0^{2n}}{r^{n-1}}\,dr
    \le C(n)\,\Big(t\,r_0^{n+2}+t^2\,r_0^{2n}\int_{r_1}^{r_0} r^{1-n}\Big)\,.
\end{eqnarray*}
By definition of $r_1$, if $n>2$ we find
\[
t^2\,r_0^{2n}\int_{r_1}^{r_0} r^{1-n}\le C(n)\,t^2\,\frac{r_0^{2n}}{r_1^{n-2}}=C(n)\,t^{2-(n-2)/(n-1)}\,r_0^{2n-n(n-2)/(n-1)}
\le C(n)\,t^{n/(n-1)}\,.
\]
and if  $n=2$ and by $t/r_0\s\le r_0^2$ (recall \eqref{sigma lb})
\[
t^2\,r_0^{2n}\int_{r_1}^{r_0} r^{1-n}=t^2\,r_0^4\,\log(r_0/r_1)=t^2\,r_0^4\,\log\Big(\frac{\s}{t\,r_0}\Big)
=\s\,t\,r_0^3\,\frac{t\,r_0}{\s}\,\log\Big(\frac{\s}{t\,r_0}\Big)\le \s\,t\,r_0^3\,.
\]
Summarizing we have proved
\[
\int_{r_1}^{r_0} \big| (\vphi')^2 - (\vphi'_0)^2 \big| \, r^{n-1} \, dr\le \k(n)\,t
\]
for a constant $\k(n)$ that can be made arbitrarily small by choosing $K(n)$ large enough. By combining this estimate with \eqref{ny7}
we find
\[
P(\Om^*)-P(B_1)\le C(n,\s)\,t\,.
\]
Since we can enforce $(1+t)^n\le 1+2nt$ for every $t<1/K$, we find
\begin{eqnarray*}
  \de(\Om)=P(\Om)-P(B_1)=(1+t)^n\,P(\Om^*)-P(B_1)\le 2\,n\,t\,P(\Om^*)+P(\Om^*)-P(B_1)\le C(n,\s)\,t\,,
\end{eqnarray*}
that is
\begin{equation}
  \label{de expansion}
  \limsup_{t\to 0^+}\frac{\de(\Om)}t\le C(n,\s)\,.
\end{equation}
Next we notice that $\pa\Om=\{(1+u(x))x:x\in\SS^n\}$ for a function $u$ such that
\begin{eqnarray}\nonumber
\|u\|_{C^0(\SS^n)}&=&\vphi_0(0)-(1+t)\vphi(0)=1-(1+t)\,\vphi(0)
\\\nonumber
&=&1-(1+t) \, \sqrt{1-r_1^2} + t \, (1+t) \, h(r_1)+(1+t) \Big(\sqrt{1+\mu^2}-1\Big)\frac{r_1}{\mu}
\\
&\ge&\label{hd eqn1}
t\,h(r_1)-t\,.
\end{eqnarray}
By \eqref{h stated}, when $n>2$ we have
\begin{eqnarray}\label{hd eqn2}
	t\,h(r_1) &=&
\int_{r_1}^{r_0} \Big( t \, \frac{r_0^n}{s^{n-1}}-t\,s - \frac{n+1}{2n-2} \, t^3 \, K^3 \, \frac{r_0^{3n}}{s^{3n-3}} \Big) \, ds
\\
&&
		+ \int_{r_1}^{r_0} \left( (1-s^2)^{3/2} - 1 \right) \Big( t \, \frac{r_0^n}{s^{n-1}}-t\,s - \frac{n+1}{2n-2} \, t^3 \, K^3 \, \frac{r_0^{3n}}{s^{3n-3}} \Big) \, ds \nonumber
\\
&&
		- \int_{r_1}^{r_0} \frac{n+5}{n-2} \, t^2 \, K^2 \, (1-s^2)^{3/2} \, \Big(\frac{r_0^{2n}}{s^{2n-3}} - \frac{r_0^{n+2}}{s^{n-1}}\Big) \, ds \nonumber \\
	&=& I_1 + I_2 + I_3\,.
\end{eqnarray}
By \eqref{r1 stated},
\begin{eqnarray} \label{hd I1}
	I_1 &=& \int_{r_1}^{r_0} \Big( t \, \frac{r_0^n}{s^{n-1}}-t\,s - \frac{n+1}{2n-2} \, t^3 \, K^3 \, \frac{r_0^{3n}}{s^{3n-3}} \Big) \, ds \nonumber \\
	&=& \frac{t\,r_0^n}{n-2}\Big(\frac1{r_1^{n-2}}-\frac1{r_0^{n-2}}\Big)-\frac{t}2\,(r_0^2-r_1^2)- \frac{(n+1) t^3 \, K^3 \, r_0^{3n}}{(2n-2)(3n-4)}\Big(\frac1{r_1^{3n-4}}-\frac1{r_0^{3n-4}}\Big)
\nonumber \\
	&\ge& \Big(\frac{t}{\sigma}\Big)^{1/(n-1)} \, r_0^{n/(n-1)} \Big( \frac{\sigma}{n-2} - \frac{(n+1) \, K^3 \, \sigma^3}{(2n-2)(3n-4)} \Big)
-C(n,\s)\,t\,.
\end{eqnarray}
Using $1 - (1-s^2)^{3/2} \leq (3/2) \, s^2$ for $s \in (r_1,r_0)$ and also using \eqref{r1 stated}, for $I_2$ we find
\begin{eqnarray} \label{hd I2}
	|I_2| &\leq& \frac{3}{2} \int_{r_1}^{r_0} \left( t \, \frac{r_0^n}{s^{n-3}} + \frac{n+1}{2n-2} \, t^3 \, K^3 \, \frac{r_0^{3n}}{s^{3n-5}} \right)
\, ds
\end{eqnarray}
where
\begin{eqnarray*}
t\,r_0^n\int_{r_1}^{r_0} \frac{ds}{s^{n-3}}&=&\s\,r_1^{n-1}\int_{r_1}^{r_0} \frac{ds}{s^{n-3}}\le C(n,\s)\,
\left\{\begin{split}
&r_1^2&\quad\mbox{if $n=3$}
\\
&r_1^3\log(r_0/r_1)&\quad\mbox{if $n=4$}
\\
&r_1^3&\quad\mbox{if $n>4$}
\end{split}\right.\\
&
\le&C(n,\s)\,t^{2/(n-1)}\,.
\end{eqnarray*}
and
\begin{eqnarray*}
t^3\,r_0^{3n}\,\int_{r_1}^{r_0} \frac{ds}{s^{3n-5}}&\le&\s^3\,r_1^{3(n-1)}\int_{r_1}^{r_0} \frac{ds}{s^{3n-5}}
\le C(n,\s)\,r_1^3
\le C(n,\s)\,t^{2/(n-1)}\,.
\end{eqnarray*}
Thus,
\begin{equation}
  \label{stima I2}
  |I_2|\le C(n,\s)\,t^{2/(n-1)}\,.
\end{equation}
Similarly,
\begin{eqnarray} \label{hd I3}
	|I_3|&\leq&C(n)\,t^2\,r_0^{2n} \int_{r_1}^{r_0} \frac{ds}{s^{2n-3}} \, ds
\le C(n,\s)\,r_1^{2n-2}\int_{r_1}^{r_0} \frac{ds}{s^{2n-3}} \, ds\\
&\le&\nonumber C(n,\s)\,r_1^2\le
C(n,\s)\,t^{2/(n-1)}\,.
\end{eqnarray}
By combining \eqref{hd eqn2}, \eqref{hd I1}, \eqref{stima I2}, and \eqref{hd I3} we conclude that if $n>2$, then
\begin{eqnarray}\nonumber
\|u\|_{C^0(\SS^n)}&\ge&
\Big(\frac{t}{\sigma}\Big)^{1/(n-1)} \, r_0^{n/(n-1)} \Big( \frac{\sigma}{n-2} - \frac{(n+1) \, K^3 \, \sigma^3}{(2n-2)(3n-4)} \Big)
-C(n,\s)\,t^{2/(n-1)}
\\\label{ny x}
&\ge& \frac{t^{1/(n-1)}}{C(n,\s)}\,,
\end{eqnarray}
up to consider a suitably large value of $K$, and where we have used $\s<1/K^2$ and $t<1/K$. We can obtain a similar inequality when $n=2$. It suffices to notice that, this time starting from \eqref{h stated 2},
\begin{eqnarray*}
t\,h(r_1)&=&\int_{r_1}^{r_0} \Big(\frac{t\,r_0^2}{s}-ts- \frac{3}{2} \, t^3 \, K^3 \, \Big(\frac{r_0^6}{s^{3}} - \frac{r_0}{s}\Big) \Big) \, ds
\\
&&+\int_{r_1}^{r_0} \Big((1-s^2)^{3/2}-1\Big) \, \left(\frac{t\,r_0^2}{s}-ts- \frac{3}{2} \, t^3 \, K^3 \, \Big(\frac{r_0^6}{s^{3}} - \frac{r_0}{s}\Big) \right) \, ds
\\
&&-7\,\int_{r_1}^{r_0} (1-s^2)^{3/2}\, t^2 \, K^2 \, \frac{r_0^4}{s}\, \log\left(\frac{r_0}{s}\right)\, ds=I_1+I_2+I_3
\end{eqnarray*}
where now using $t\,r_0^2=\s\,r_1$ we find
\begin{eqnarray*}
I_1&=&\int_{r_1}^{r_0} \Big(\frac{t\,r_0^2}{s}-ts- \frac{3}{2} \, t^3 \, K^3 \, \Big(\frac{r_0^6}{s^{3}} - \frac{r_0}{s}\Big) \Big) \, ds
  \ge t\,r_0^2\,\log\Big(\frac{\s}{r_0\,t}\Big)-C(n,\s)\,t
\\
|I_2|&\le&\frac32\int_{r_1}^{r_0} s^2 \Big|\frac{t\,r_0^2}{s}-ts- \frac{3}{2} \, t^3 \, K^3 \, \Big(\frac{r_0^6}{s^{3}} - \frac{r_0}{s}\Big) \Big| \, ds\le C(n,\s)\,t\,,
\end{eqnarray*}
while since $s^{-1}\log(r_0/s)$ is decreasing on $s\in(0,r_0)$,
\[
|I_3|\le C(n,\s)\,t^2\,r_0^4\int_{r_1}^{r_0}\frac1{s}\,\log\left(\frac{r_0}{s}\right)\, ds\le C(n,\s)\,t^2\,r_0^4\,\frac{r_0-r_1}{r_1}\,\log\left(\frac{r_0}{r_1}\right)\le C(n,\s)\,t\,r_0^4\,\log\Big(\frac{\s}{r_0\,t}\Big)\,.
\]
Hence, provided $K$ is large enough,
\begin{eqnarray}
  t\,h(r_1)\ge t\,\Big(r_0^2-C(n,\s)\,r_0^4\Big)\,\log\Big(\frac{\s}{r_0\,t}\Big)-C(n,\s)\,t\ge\frac{t\,\log(\s/r_0t)}{C(n,\s)}\,,
\end{eqnarray}
which combined with \eqref{hd eqn1} gives us that, if $n=2$, then
\begin{equation}
  \label{ny y}
  \|u\|_{C^0(\SS^2)}\ge \frac{t\,\log(\s/r_0t)}{C(n,\s)}\,.
\end{equation}
By combining \eqref{de expansion} with \eqref{ny x} and \eqref{ny y} we complete the proof of \eqref{est}.

\section{A sharp result for boundaries with almost constant mean curvature}\label{section sharp alexandrov} Here we prove Theorem \ref{thm alex} and Theorem \ref{thm alex L2}, starting from the latter.

\begin{proof}[Proof of Theorem \ref{thm alex L2}] Let $\Om$ be an open set with $C^{1,1}$-boundary in $\R^{n+1}$ with $\int_{\pa\Om}x=0$ and
\[
\pa \Om=\big\{(1+u(x))x:x\in\SS^n\big\}, \quad \|u\|_{C^1(\SS^n)} \le \e(n)
\]
for a function $u\in C^1(\SS^n)$. If we let $\e=\e(n)$ be as in Lemma \ref{thm quant u prelim}, and we argue as in the first three steps of the proof of Lemma \ref{thm quant u prelim} (where the assumption $H_\Om\le n$ of Lemma \ref{thm quant u prelim} was not invoked), then, writing $u=a+b\cdot x+R$ as in \eqref{u a b R}, so that
\[
a= \frac1{\H^n(\SS^n)}\,\int_{\mathbb{S}^n} u\qquad b_i=\frac{\int_{\SS^n}x_i\,u}{\int_{\SS^n}x_i^2}\quad i=1,...,n\,,
\]
we have the estimates \eqref{R poincare} and \eqref{b controlled by DR}
\begin{eqnarray}\label{R poincare 2}
  (2n+1)\int_{\SS^n}R^2&\le&\int_{\SS^n}|\nabla R|^2\,,
  \\\label{b controlled by DR 2}
  |b| &\leq& C(n) \, \int_{\mathbb{S}^n} \, |\nabla R|^2+\e\,{\rm O}\big(\|u\|_{W^{1,2}(\SS^n)}^2\big)\,,
\end{eqnarray}
as well as the identities \eqref{H formula u} and \eqref{later}
\begin{align} \label{H formula u 2}
	H^*= -\op{div}_{\mathbb{S}^n} \left( \frac{\nabla u}{(1+u) \, \sqrt{(1+u)^2 + |\nabla u|^2}} \right) + \frac{ n-\frac{|\nabla u|^2}{(1+u)^2} }{\sqrt{(1+u)^2 + |\nabla u|^2}} \,
\end{align}
\begin{eqnarray}\label{later 2}
  n\,\H^n(\SS^n)\,a^2+\int_{\SS^n}|\nabla u|^2-n\,u^2=\int_{\SS^n}|\nabla R|^2-n\,R^2
\end{eqnarray}
where $H^*(x)=H_\Omstar(x+u(x)\,x)$ for each $x\in\SS^n$. Subtracting $n$ from both sides of \eqref{H formula u 2} and multiplying by $u$, we find that
\begin{eqnarray*}
\int_{\SS^n}(H^*-n)\,u=\int_{\SS^n}|\nabla u|^2-n\,u^2+ \e\,{\rm O}\big(\|u\|_{W^{1,2}(\SS^n)}^2\big)\,.
\end{eqnarray*}
By \eqref{later 2}, \eqref{R poincare 2} and \eqref{b controlled by DR 2},  we thus have
\begin{eqnarray}\nonumber
  \int_{\SS^n}|\nabla R|^2+R^2+|b|&\le&
  C(n)\,\Big(a^2+\int_{\SS^n}(H^*-n)\,u\Big)+\e\,{\rm O}\big(\|u\|_{W^{1,2}(\SS^n)}^2\big)
  \\&\le&\label{wait1}
  C(n)\,\Big(a^2+\|H_\Om-n\|_{L^2(\pa\Om)}\,\|u\|_{L^2(\SS^n)}\Big)+\e\,{\rm O}\big(\|u\|_{W^{1,2}(\SS^n)}^2\big)\,,
\end{eqnarray}
where we have used $\|H^*-n\|_{L^2(\SS^n)}\le C(n)\,\|H_\Om-n\|_{L^2(\pa\Om)}$.  By integrating \eqref{H formula u 2} over $\SS^n$ after subtracting $n$ from both of its sides, we find that
\begin{eqnarray}\label{stima a Alex 0}
  \Big|n\int_{\SS^n}u-\int_{\SS^n}(n-H^*)\Big|\le C(n)\,\int_{\SS^n}u^2+|\nabla u|^2\,.
\end{eqnarray}
By the Cauchy-Schwarz inequality
\[
  \Big| \int_{\SS^n}(n-H^*) \Big|\le C(n) \, \|H^*-n\|_{L^2(\SS^n)} \le C(n)\,\|H_\Om-n\|_{L^2(\pa\Om)} \,.
\]
By combining this estimate with \eqref{stima a Alex 0} we find that
\[
|a|\le C(n)\,\Big(\int_{\SS^n}u^2+|\nabla u|^2+\|H_\Om-n\|_{L^2(\pa\Om)}\Big)\,.
\]
which together with \eqref{wait1} gives us
\begin{eqnarray*}
  \int_{\SS^n}|\nabla u|^2+u^2&\le& C(n)\,
  \int_{\SS^n}|\nabla R|^2+R^2+|b|^2+a^2
  \\
  &\le&
  C(n)\,\Big(\|H_\Om-n\|_{L^2(\pa\Om)}^2+\|H_\Om-n\|_{L^2(\pa\Om)}\,\|u\|_{L^2(\SS^n)}\Big)+\e\,{\rm O}\big(\|u\|_{W^{1,2}(\SS^n)}^2\big)
\end{eqnarray*}
and thus
\begin{equation}\label{ie}
  \|u\|_{W^{1,2}(\SS^n)}\le  C(n)\, \|H_\Om-n\|_{L^2(\pa\Om)}\,.
\end{equation}
This proves \eqref{alex L2 bound}. We now prove \eqref{alex C0 bound new}. Let us pick $p$ as in \eqref{p hp new}, and notice we can apply \eqref{stima after} from Lemma \ref{thm quant u prelim 2} to deduce
\begin{equation}
  \label{stima after x}
      \|u\|_{C^0(\SS^n)}\le C(n,q)\,\Big(\|u\|_{L^2(\SS^n)}+\|H_\Om-n\|_{L^q(\SS^n)}\Big)\,,\qquad\forall q>\frac{n}2\,.
\end{equation}
By setting $q=p$ if $n\ge 4$, or by fixing any $q\in(n/2,2)$ otherwise, we immediately deduce
\[
 \|u\|_{C^0(\SS^n)}\le C(n,p)\,\|H-n\|_{L^p(\SS^n)}\,,
\]
by combining \eqref{alex L2 bound} (that is \eqref{ie}), H\"older inequality and \eqref{stima after x}. We conclude the proof of Theorem \ref{thm alex L2} by noticing that if we now assume $\|u\|_{C^{1,\a}(\SS^n)}\le K$ for some $\a\in(0,1)$ and $K>0$, then \eqref{alex C1alpha} follows immediately by combining \eqref{stima after holder} from Lemma \ref{thm quant u prelim 2} with \eqref{alex C0 bound new}.
\end{proof}

\begin{proof}
  [Proof of Theorem \ref{thm alex}] By applying \cite[Theorem 2.5]{ciraolomaggi} while taking into account that $P(\Om)\le 2\tau P(B_1)$ we find that, up to a translation setting $\int_{\pa\Om}x=0$, $\pa\Om=\{(1+u(x))x:x\in\SS^n\}$ for a function $u\in C^{1,1}(\SS^n)$ such that $\|u\|_{C^{1,1/2}(\SS^n)}\le C(n)$ and $\|u\|_{C^1(\SS^n)}$ is arbitrarily small provided $\de_{{\rm cmc}}(\Om)$ is suitably small. We are thus in the position to apply conclusion \eqref{alex C1alpha} from Theorem \ref{thm alex L2} to $\Om$ (with the choice $\a=1/2$) to conclude the proof.
\end{proof}

\bibliography{references}
\bibliographystyle{is-alpha}

\end{document}